\RequirePackage{amsthm}

\documentclass[sn-mathphys-ay]{sn-jnl}

\usepackage{lmodern}
\usepackage{graphicx}%
\usepackage{multirow}%
\usepackage{amsmath,amssymb,amsfonts}%
\usepackage{mathrsfs}%
\usepackage[title]{appendix}%
\usepackage{xcolor}%
\usepackage{textcomp}%
\usepackage{manyfoot}%
\usepackage{booktabs}%
\usepackage{algorithm}%
\usepackage{algorithmicx}%
\usepackage{algpseudocode}%
\usepackage{listings}%
\usepackage{hyperref}
\usepackage{cleveref}
\usepackage{accents}
\usepackage{commath}
\usepackage{anyfontsize}
\theoremstyle{thmstyleone}%
\newtheorem{theorem}{Theorem}
\newtheorem{proposition}[theorem]{Proposition}%
\newtheorem{corollary}[theorem]{Corollary}
\newtheorem{lemma}[theorem]{Lemma}

\theoremstyle{thmstyletwo}%
\newtheorem{remark}{Remark}%

\theoremstyle{thmstylethree}%

\let\cal\mathcal

\newcommand{\one}{\boldsymbol{1}}

\def\Xint#1{\mathchoice
{\XXint\displaystyle\textstyle{#1}}%
{\XXint\textstyle\scriptstyle{#1}}%
{\XXint\scriptstyle\scriptscriptstyle{#1}}%
{\XXint\scriptscriptstyle\scriptscriptstyle{#1}}%
\!\int}
\def\XXint#1#2#3{{\setbox0=\hbox{$#1{#2#3}{\int}$ }
\vcenter{\hbox{$#2#3$ }}\kern-.6\wd0}}

\usepackage{listings}
\usepackage{color}

\definecolor{codegreen}{rgb}{0,0.6,0}
\definecolor{codegray}{rgb}{0.5,0.5,0.5}
\definecolor{codepurple}{rgb}{0.58,0,0.82}
\definecolor{backcolour}{rgb}{0.95,0.95,0.92}

\lstdefinestyle{mystyle}{
  backgroundcolor=\color{backcolour},   commentstyle=\color{codegreen},
  keywordstyle=\color{magenta},
  numberstyle=\tiny\color{codegray},
  stringstyle=\color{codepurple},
  basicstyle=\footnotesize,
  breakatwhitespace=false,         
  breaklines=true,                 
  captionpos=b,                    
  keepspaces=true,                 
  numbers=left,                    
  numbersep=5pt,                  
  showspaces=false,                
  showstringspaces=false,
  showtabs=false,                  
  tabsize=2
}

\newcommand{\proofparagraph}[1]{\noindent~~\newline\newline\noindent\textbf{ \underline{#1}.}~\newline\newline}
\newcommand{\pparagraph}[1]{\noindent~~\newline\textbf{#1.}}

\def\/{|\!|\!|}
\def\${|\!|\!|}
\def\l|{\left|\!\left|\!\left|}
\def\r|{\right|\!\right|\!\right|}

\def\R{{\mathbb{R}}}

\def\T{{\mathcal{T} }}

\def\L{{\mathcal{L}}}
\def\Z{{\mathbb{Z}}}

\newcommand{\mc}[1]{\mathcal{#1}}
\newcommand{\Rplus}{\mathbb{R}_{+}}

\newcommand{\uhp}{\mathbb{H}}

\newcommand{\E}[0]{\mathbb{E}}

\newcommand\scal[2][ ]{\ifthenelse{\equal{#1}{ }}{\langle#2\rangle}{}
        \ifthenelse{\equal{#1}{b}}{\bigl\langle#2\bigr\rangle}{}
        \ifthenelse{\equal{#1}{B}}{\Bigl\langle#2\Bigr\rangle}{}
        \ifthenelse{\equal{#1}{bb}}{\biggl\langle#2\biggr\rangle}{}
        \ifthenelse{\equal{#1}{BB}}{\Biggl\langle#2\Biggr\rangle}{}}

\newcommand\indp{\protect\mathpalette{\protect\independenT}{\perp}}
\def\independenT#1#2{\mathrel{\rlap{$#1#2$}\mkern2mu{#1#2}}}

\newcommand{\Prob}{\mathbb{P}}
\newcommand{\Proba}[1]{\mathbb{P}\left[#1\right]}

\newcommand{\iid}{\text{i.i.d.}}

\newcommand{\Expe}[1]{\mathbb{E}\left[ #1\right] }

\newcommand{\Exp}{\mathbb{E}}

\newcommand{\wt}{\widetilde}

\newcommand\thickbar[1]{\accentset{\rule{.4em}{.8pt}}{#1}}

\def\threebars{|\!|\!|}
\def\thrb{\big|\!\big|\!\big|}
\def\thrB{\Big|\!\Big|\!\Big|}
\def\thrbb{\bigg|\!\bigg|\!\bigg|}
\def\thrBB{\Bigg|\!\Bigg|\!\Bigg|}
\newcommand\nn[3][ ]{%
\ifthenelse{\equal{#1}{ }}{\|{#3}\|_{#2}}{}%
\ifthenelse{\equal{#1}{b}}{\bigl\|{#3}\bigr\|_{#2}}{}%
\ifthenelse{\equal{#1}{B}}{\Bigl\|{#3}\Bigr\|_{#2}}{}%
\ifthenelse{\equal{#1}{bb}}{\biggl\|{#3}\biggr\|_{#2}}{}%
\ifthenelse{\equal{#1}{BB}}{\Biggl\|{#3}\Biggr\|_{#2}}{}}%
\newcommand\nnn[3][ ]{%
\ifthenelse{\equal{#1}{ }}{\threebars{#3}\threebars_{#2}}{}%
\ifthenelse{\equal{#1}{b}}{\mathopen\thrb{#3}\mathclose\thrb_{#2}}{}%
\ifthenelse{\equal{#1}{B}}{\mathopen\thrB{#3}\mathclose\thrB_{#2}}{}%
\ifthenelse{\equal{#1}{bb}}{\mathopen\thrbb{#3}\mathclose\thrbb_{#2}}{}%
\ifthenelse{\equal{#1}{BB}}{\mathopen\thrBB{#3}\mathclose\thrBB_{#2}}{}}%

\newcommand{\Holder}{H\"{o}lder }

\newcommand{\Levy}{L\'{e}vy}

\def\mop#1{\mathop{\hbox{\textrm{#1}}}\nolimits}
\def\moplim#1{\mathop{\hbox{\textrm{#1}}}}
\def\harr#1#2{\smash{\mathop{\hbox to .5in{\rightarrowfill}}
        \limits^{\scriptstyle#1}_{\scriptstyle#2}}}

\makeatletter
\newcommand*{\relrelbarsep}{.386ex}
\newcommand*{\relrelbar}{%
  \mathrel{%
    \mathpalette\@relrelbar\relrelbarsep
  }%
}
\newcommand*{\@relrelbar}[2]{%
  \raise#2\hbox to 0pt{$\m@th#1\relbar$\hss}%
  \lower#2\hbox{$\m@th#1\relbar$}%
}
\providecommand*{\rightrightarrowsfill@}{%
  \arrowfill@\relrelbar\relrelbar\rightrightarrows
}
\providecommand*{\leftleftarrowsfill@}{%
  \arrowfill@\leftleftarrows\relrelbar\relrelbar
}
\providecommand*{\xrightrightarrows}[2][]{%
  \ext@arrow 0359\rightrightarrowsfill@{#1}{#2}%
}
\providecommand*{\xleftleftarrows}[2][]{%
  \ext@arrow 3095\leftleftarrowsfill@{#1}{#2}%
}
\makeatother
\def\@boxsqr#1#2#3{\vbox to0pt{
\fboxsep0pt%
\definecolor{tempc}{rgb}{#3}%
\colorbox{tempc}{\vbox to #2{\vss\hbox to #1{\hss}}}%
}}
\long\def\fillbox{\@ifnextchar[{\@fillbox{}{}}{\@fillbox{}{}[0.5mm,0.5mm]}}
\long\def\mfillbox{\@ifnextchar[%
	{\@fillbox{$\displaystyle\bgroup}{\egroup$}}%
	{\@fillbox{$\displaystyle\bgroup}{\egroup$}[0.5mm,0.5mm]}%
}
\long\def\@fillbox#1#2[#3,#4]#5#6{%
	\leavevmode%
  \setbox\@tempboxa\hbox{#1#6#2}%
  \@tempdima\wd\@tempboxa%
  \@tempdimb\ht\@tempboxa%
  \@tempdimc\dp\@tempboxa%
  \advance\@tempdima#3\advance\@tempdima#3%
  \advance\@tempdimb#4\advance\@tempdimb#4\advance\@tempdimb\@tempdimc%
  \advance\@tempdimc#4%
  \setbox\@tempboxa\hbox{%
	\@boxsqr{\number\@tempdima}{\number\@tempdimb}{#5}%
	\kern#3\raise\@tempdimc\box\@tempboxa}%
  \kern-#3\raise-\@tempdimc\box\@tempboxa%
  }
\let\phi\varphi

\def\rho{\varrho}

\def\doncl{\quad\Leftrightarrow\quad}

\def\ie{{\it i.e.}\@ifnextchar,{}{\ }}
\def\eg{{\it e.g.}\@ifnextchar,{}{\ }}

\def\({\left(}
\def\){\right)}
\newcommand{\para}[1]{\left(#1\right) }
\newcommand{\spara}[1]{\left[ #1\right] }

\newcommand{\maxp}[1]{\max\left(#1\right) }
\newcommand{\minp}[1]{\min\left(#1\right) }

\newcommand\cinf[1][ ]{\ifthenelse{\equal{#1}{ }}{{\cal C}^\infty}{{\cal C}^\infty(#1)}}
\newcommand\cOinf[1][ ]{\ifthenelse{\equal{#1}{ }}{{\cal C}_0^\infty}{{\cal C}_0^\infty(#1)}}
\newcommand\symb[2][\bf]{{\mathchoice{\hbox{#1#2}}{\hbox{#1#2}}%
        {\hbox{\scriptsize#1#2}}{\hbox{\tiny#1#2}}}}

\def\L^#1{{\textrm L}^{\!#1}}

\def\newop#1{\expandafter\def\csname#1\endcsname{\mop{#1}}}
\def\newoplim#1{\expandafter\def\csname#1\endcsname{\moplim{#1}}}
\newop{deg}
\newop{div}

\newop{sin}
\newop{arcsin}
\newop{rank}
\newop{tr}
\newop{cos}
\newop{arccos}
\newop{tan}
\newop{sinh}
\newop{arsinh}
\newop{cosh}
\newop{arcosh}
\newop{tanh}
\newop{arctan}
\newop{artanh}
\newop{log}
\newop{exp}
\newop{sgn}
\newop{dim}
\newop{ker}
\newoplim{min}
\newoplim{max}
\newoplim{sup}
\newoplim{inf}
\newoplim{lim}
\def\liminf{\moplim{lim inf}}
\def\limsup{\moplim{lim sup}}

\newcommand{\e}{\varepsilon}

\def\CU{{\mathcal{U} } }

 \def\T{\mathbf{T}}

 \def\E{\mathbf{E}}
 \def\P{\mathbf{P}}

 \newcommand{\tor}{\text{ \normalfont{or} }}
 \newcommand{\tand}{\text{ \normalfont{and} }}
 \newcommand{\tif}{\text{ \normalfont{if} }}
 \newcommand{\tifc}{\text{ ~,\normalfont{if}~}}
 \newcommand{\tfor}{\text{ \normalfont{for} }}

  \newcommand{\twhen}{\text{ \normalfont{when} }}
 \newcommand{\tforsome}{\text{ for some }}
 \newcommand{\twith}{\text{ with }}
 \newcommand{\tas}{\text{ as }}

 \newcommand{\tthen}{\text{ then }}

 \newcommand{\mb}[1]{\mathbb{#1}}
  \newcommand{\longequal}{=\joinrel=}

\newcommand{\tcwhen}{\text{~,\normalfont{when}~}}

\newcommand{\sectm}[1]{\texorpdfstring{#1}{OOES}}

\newcommand{\branchmat}[1]{ \left\{\begin{matrix}
 #1
\end{matrix}\right. }

 \newenvironment{eqalign}{\begin{equation}\begin{aligned}}{\end{aligned}\end{equation}}

\newtheorem{innercustomgeneric}{\customgenericname}
\providecommand{\customgenericname}{}
\newcommand{\newcustomtheorem}[2]{%
  \newenvironment{#1}[1]
  {%
   \renewcommand\customgenericname{#2}%
   \renewcommand\theinnercustomgeneric{##1}%
   \innercustomgeneric
  }
  {\endinnercustomgeneric}
}

\newcustomtheorem{customthm}{Theorem}
\newcustomtheorem{customlemma}{Lemma}

 \newcommand{\ba}{\[\begin{aligned}}
\newcommand{\ea}{\end{aligned}\]}

 \newcommand{\ben}{\begin{enumerate}}
 
 \newcommand{\een}{\end{enumerate}}

 \newcommand{\bena}{\begin{enumerate}[label={\bf (\alph*) } ]}
 
 \newcommand{\benr}{\begin{enumerate}[\textbf{i.}]}

\newcommand{\br}[1]{\boldsymbol{\mathrm{#1}} }

\newcommand{\ind}[1]{\chi\left\{#1 \right\}}

 \newcommand{\lii}[2][\infty]{\lim\limits_{#2\to #1}}
\newcommand{\liz}[2][0]{\lim\limits_{#2\to #1}}

 \newcommand{\liml}[1]{\lim\limits_{#1} } 

 \newcommand{\supl}[1]{\sup\limits_{#1}}
 \newcommand{\supls}[1]{\sup\limits_{\substack{#1}}}
 
\newcommand{\infl}[1]{\inf\limits_{#1}}

\newcommand{\infp}[1]{\inf\left\{#1\right\} }

  \newcommand{\maxl}[1]{\max\limits_{#1}}

\newcommand{\codis}{\stackrel{d}{\to}}

\newcommand{\eqdis}{\stackrel{d}{=}}

\newcommand{\diam}{\mathrm{diam}  }

\newcommand{\dint}{\mathrm{d}  }

\newcommand{\dx}{\mathrm{d}x  }
\newcommand{\dy}{\mathrm{d}y  }

\newcommand{\dt}{\mathrm{d}t  }

\newcommand{\du}{\mathrm{d}u  }

\newcommand{\ds}{\mathrm{d}s  }
\newcommand{\dr}{\mathrm{d}r  }

\newcommand{\deta}{\mathrm{d}\eta  }

\newcommand{\dtheta}{\mathrm{d}\theta  }
\newcommand{\dlambda}{\mathrm{d}\lambda  }

\newcommand{\etam}[1]{\eta \left(\left[ #1\right]\right) }

\newcommand{\etamu}[2]{\eta^{#2} \left(\left[ #1\right]\right) }

\newcommand{\etamul}[3]{\eta^{#2}_{#3}\left(\left[ #1\right]\right) }

  \newcommand{\prodnk}[1]{\prod_{k=0}^{n}}

\newcommand{\normthrb}[1]{{\left\vert\kern-0.25ex\left\vert\kern-0.25ex\left\vert #1 
    \right\vert\kern-0.25ex\right\vert\kern-0.25ex\right\vert}}

\makeatletter
\providecommand{\vrectangle}{\mkern1mu\mathpalette\v@rectangle\relax\mkern1mu}
\newcommand{\v@rectangle}[2]{%
  \hbox{
  \fboxrule=0.5\fontdimen 8
    \ifx#1\displaystyle\textfont\else
    \ifx#1\textstyle\textfont\else
    \ifx#1\scriptstyle\scriptfont\else
    \scriptscriptfont\fi\fi\fi 3
  \fboxsep=-\fboxrule
  \fbox{$\m@th#1\phantom{(}$}%
  }
}
\makeatother

\newcommand{\expo}[1]{\exp\left\{ #1\right\}}

\newcommand{\ceil}[1]{\lceil#1 \rceil }

\newcommand\tqed[1]{%
  \leavevmode\unskip\penalty9999 \hbox{}\nobreak\hfill
  \quad\hbox{#1}}

\def \P {{\bf P}}
\def\md{\mid}
\def\Bb#1#2{{\def\md{\bigm| }#1\bigl[#2\bigr]}}
\def\BB#1#2{{\def\md{\Bigm| }#1\Bigl[#2\Bigr]}}
\def\Bs#1#2{{\def\md{\mid}#1[#2]}}

\makeatletter
\DeclareRobustCommand{\cev}[1]{%
  \mathpalette\do@cev{#1}%
}
\newcommand{\do@cev}[2]{%
  \fix@cev{#1}{+}%
  \reflectbox{$\m@th#1\vec{\reflectbox{$\fix@cev{#1}{-}\m@th#1#2\fix@cev{#1}{+}$}}$}%
  \fix@cev{#1}{-}%
}
\newcommand{\fix@cev}[2]{%
  \ifx#1\displaystyle
    \mkern#23mu
  \else
    \ifx#1\textstyle
      \mkern#23mu
    \else
      \ifx#1\scriptstyle
        \mkern#22mu
      \else
        \mkern#22mu
      \fi
    \fi
  \fi
}

\makeatother

\newcommand{\EE}{\ensuremath{\mathbb{E}}}

\makeatletter
\newcommand{\setword}[2]{%
  \phantomsection
  #1\def\@currentlabel{\unexpanded{#1}}\label{#2}%
}
\makeatother

\makeatletter
\let\old@rule\@rule
\def\@rule[#1]#2#3{\textcolor{black}{\old@rule[#1]{#2}{#3}}}
\makeatother

\begin{document}

\title[Article Title]{Inverse of the Gaussian multiplicative chaos: Moments}


\author*[1]{\fnm{Ilia} \sur{Binder}}\email{ilia@math.utoronto.ca}

\author*[1]{\fnm{Tomas} \sur{Kojar}}\email{tomas.kojar@mail.utoronto.ca}

\affil*[1]{\orgdiv{Math Department}, \orgname{University of Toronto}, \orgaddress{\street{40 St. George Street}, \city{Toronto}, \postcode{M5S 2E4}, \state{Ontario}, \country{Canada}}}


\abstract{In this article, we systematically study the general properties and single-point moments of the inverse of Gaussian multiplicative chaos.}




\maketitle
\tableofcontents

\part{Introduction}
This article introduces the systematic study of the inverse of the Gaussian multiplicative chaos (GMC). 
It is also the first part of the project of extending the work in \citep{AJKS} (see \citep{binder2023inverselehto}). Let us start with an informal discussion of our objects. We will give  precise definitions and statements later on (see also \citep{AJKS,bacry2003log}). Even though the initial motivation is a conformal welding problem mentioned below, we believe this article helps open up the study of an important foundational aspect of GMC. \\
We focus on the GMC on the real line. Namely, we consider a Gaussian field $U(t)$ on the real line $[0,\infty)$ with logarithmic covariance $\Expe{U(t)U(s)}=-\ln\abs{t-s}$, the corresponding GMC measure $\eta$
\begin{equation}
\eta(I):=\liz{\e}\int_{I}e^{\gamma U_{\e}(x)-\frac{\gamma^{2}}{2}\Expe{\para{U_{\e}(x)}^{2}}}\dx, I\subset [0,\infty)
\end{equation}
and its inverse $Q=Q_{\eta}$ i.e. $Q(\eta(x))=x=\eta(Q(x))$ (see precise definitions below). We try to translate the known properties of GMC to its inverse. For example, does the inverse have a scaling law? Does it have a density? Does it satisfy any decoupling properties? What are some bounds for its moments? We go through all these questions in this article and an ensuing one that will focus on decoupling.
\pparagraph{Inverse conformal welding}Here we go over the motivation that led to the study of the inverse. The Gaussian random field $H$ is the random Gaussian field on the circle with the covariance
\begin{equation}
\Expe{H(z)H(z')}=-\ln\abs{z-z'},
\end{equation}
where $z, z'\in\mathbb{C}$ have modulus $1$.  The exponential $\expo{\gamma H}$ gives rise to a random measure $\tau$ on the unit circle $\T$, given by
\begin{equation}
\tau(I):=\liz{\e}\int_{I}e^{\gamma H_{\e}(x)-\frac{\gamma^{2}}{2}\Expe{\para{H_{\e}(x)}^{2}}}\dx,
\end{equation}
for Borel subsets $I\subset \T=\mathbb{R}/ \mathbb{Z}=[0,1)$ and $H_{\e}$ is a suitable regularization. This measure is within the family of \textit{Gaussian multiplicative chaos} measures (GMC) (for expositions see the lectures \citep{robert2010gaussian,rhodes2014gaussian}). So finally, in \citep{AJKS} they consider the random homeomorphism $h:[0,1)\to [0,1)$ defined as the normalized measure
\begin{equation}
h(x):=\frac{\tau[0,x]}{\tau[0,1]}, x\in [0,1),
\end{equation}
and prove that it gives rise to a Beltrami solution and a \textit{conformal welding } map. The goal is to extend this result to its inverse $h^{-1}$ and then to the composition $h_{1}^{-1}\circ h_{2}$ where $h_{1},h_{2}$ are two independent copies. The motivation for that is for obtaining a parallel point of view of the beautiful work by Sheffield \citep{sheffield2016conformal} of gluing two quantum disks to obtain an SLE loop. \\
We let $Q_{\tau}(x):[0,\tau([0,1])]\to [0,1]$ denote the inverse of the measure $\tau:[0,1]\to [0,\tau([0,1])]$ i.e.
\begin{equation}
Q_{\tau}(\tau[0,x])=x\tand \tau[0,Q_{\tau}(y)]=y,
\end{equation}
for $x\in [0,1]$ and $y\in [0,\tau([0,1])].$  The existence of the inverse $Q$ follows from GMC $\tau$ being non-atomic \cite[theorem 1]{bacry2003log} and strictly monotone. The monotonicity follows because $\eta$ satisfies bi-\Holder over dyadic intervals $I$ \cite[theom 3.7]{AJKS}
\begin{eqalign}
c_{1}\abs{I}^{a_{1}}\leq \tau(I) \leq c_{2}\abs{I}^{a_{2}}.   
\end{eqalign}
This also implies continuity for the inverse $Q$.\\ We use the notation Q because the measure $\tau$ can be thought of as the "CDF function" for the "density" $\expo{\gamma H}$ and thus its inverse $\tau^{-1}=Q$ is the quantile (also using the notation $\tau^{-1}$ would make the equations less legible later when we start including powers and truncations). We will also view this inverse as a hitting time for the function $t\mapsto\tau([0,t])$
\begin{equation}
Q_{\tau}(x)=T_{x}:=\inf\set{t\geq 0: \tau[0,t]\geq x}.
\end{equation}
The inverse homeomorphism map $h^{-1}:[0,1]\to [0,1]$ is defined as
\begin{equation}\label{inversehomeo}
h^{-1}(x):=Q_{\tau}(x\tau([0,1]))\tfor x\in [0,1]
\end{equation}
For the map $h(x)$ we have the periodic extension $h(1+x)=1+h(x)$ for $x\in \mb{R}$ and so by solving for x its inverse will satisfy the same periodicity $h^{-1}(x+1)=h^{-1}(x)+1$.

\proofparagraph{Known results on the inverse}\\
 The spectrum for various multifractal inverses is computed in the papers \citep{riedi1997inversion,mandelbrot1997inverse,barral2009singularity,olsen2010inverse}. One main result is the \textit{inverse formula} for the dimension of the spectrum \citep{riedi1997inversion}. Let
\begin{equation}
T_{\eta,a}:=\set{t\in [0,1]: \liml{\delta\to 0}\frac{\ln(\eta([t,t+\delta]))}{\ln(\delta)}=a    }.
\end{equation}
From \cite[corrolary 7]{riedi1997inversion} we have the inversion formula for $Q$ the inverse of $\eta$
\begin{equation}
\dim(T_{Q,a})=a\dim(T_{\eta,\frac{1}{a}}).
\end{equation}
For the GMC measure $\eta$ we have the Legendre transform formula \cite[section 4.2]{rhodes2014gaussian},\citep{barral1999moments,barral2007singularity}
\begin{equation}
\dim(T_{\eta,a})=\inf\limits_{q\in \R}\set{aq-\tau_{\eta}(q)},
\end{equation}
where we let
\begin{equation}
\tau_{\eta}(q):=\branchmat{(1+\frac{\gamma}{\sqrt{2}})^{2}q& \tifc q\leq \frac{-\sqrt{2}}{\gamma}\\
\zeta_{\delta}(-q)-1& \tifc q\in \spara{ \frac{-\sqrt{2}}{\gamma}, \frac{\sqrt{2}}{\gamma}} \\
(1-\frac{\gamma}{\sqrt{2}})^{2}q& \tifc  \frac{\sqrt{2}}{\gamma}\leq q  }
\end{equation}
and the moment exponent $\zeta(q):=q-\frac{\gamma^{2}}{2}(q^{2}-q)$ \citep{bacry2003log}. Therefore, for the inverse we have
\begin{equation}
\dim(T_{Q,a})=a\infl{q\in \R}\set{\frac{q}{a}-\tau_{\eta}(q)}. \end{equation}
The inverse also shows up when studying the Liouville Brownian motion as the time-change ("clock") (see \citep{berestycki2015diffusion}): consider 2d Brownian motion $B_t$ independent of the logarithmic Gaussian field $U$, then they work with the "clock"/quadratic-variation
\begin{equation}
\mu^{-1}(a):=\inf\set{t\geq 0: \int_{0}^{t}e^{U(B_{s})}\ds\geq a  }
\end{equation}
and define the \textbf{Liouville Brownian motion} to be the time-changed $\wt{B}_{a}:=B_{\mu^{-1}(a)}$.
\section{Outline and Main results}
 Since the inverse of GMC didn't seem to appear in other problems, it was studied very little and so we had to find and build many of its properties. Our guide for much for this work was:
\begin{itemize}
    \item  trying to translate and transfer the known properties of the GMC measure to its inverse,

    \item the Markovian structure for the hitting times of Brownian motion's (such as the Wald's equation and the independent of the increments of hitting times),

\item and finally trying to get whatever property was required for the framework set up by \citep{AJKS} to go through successfully.
\end{itemize}
This was a situation where a good problem became the roadmap for finding many interesting properties for the inverse of GMC and thus GMC itself. Here is an outline of each of the sections of this article:

\begin{itemize}
    \item In \Cref{part:generalpropertiesinverse}, we find the inverse analogues of the known scaling laws (which imply a density formula too for the inverse) and comparison formulas for the inverse between different fields as done in \cite[lemma 3.5]{AJKS} for GMC measures.

 \item In \Cref{part:momentinverse}, for the field $U$ on the real line, we compute in detail the moments for increments $Q(a,b):=Q(b)-Q(a)$ for general $a<b$. The challenge here is the lack of translation invariance $Q(a,b)\stackrel{d}{\neq}Q(0,b-a)$.

    \item In \Cref{part:rateofconvergence}, we study the rate of convergence of the lower truncated inverse
    \begin{eqalign}
     \Expe{\abs{Q_{n+m}(x)-Q_{n}(x)}^{\ell}}\to 0 \tas n\to +\infty,
    \end{eqalign}
where $Q_{n}$ corresponds to the lower-truncated field $U_{\delta_{n}}^{1}$ for some sequence $\delta_{n}\to 0$ that we will discuss later.

    \item In the appendix, we prove GMC estimates that we need for this work.

\end{itemize}
\subsection{Main results}
The main result is the study of the moments for a)the inverse $Q(x)=Q_{U}(x)$ and b)the inverse increments $Q(a,b)=Q_{U}(b)-Q_{U}(a)$ for the field $U$ on the real line. Let $\beta:=\frac{\gamma^{2}}{2}$.
\begin{theorem}
The inverse has finite moments $\Exp[(Q^{\delta}(0,x))^{p}]<\infty $ for the interval
\begin{eqalign}
p\in \para{-\frac{(1+\beta)^{2}}{4\beta}, \infty},  \end{eqalign}
where the left endpoint is given by the supremum of the moment exponent of the $\eta$ measure:
\begin{eqalign}
\frac{(1+\beta)^{2}}{4\beta}=\supl{q\in (0,\frac{1}{\beta})}\zeta(q),
\end{eqalign}
where $\zeta(q):=q-\frac{\gamma^{2}}{2}(q^{2}-q)$ is the multifractal exponent.
\begin{itemize}

\item Positive moments: For general non-negative moments $p\geq 0$ and constants $\delta,x\geq 0$, we have the upper bound for $Q^{\delta}$:

\begin{eqalign}
\Exp[(Q^{\delta}(0,x))^{p}]\lesssim\branchmat{\delta^{p-q} x^{q}+\delta^{p-\wt{q}}x^{\wt{q}}&\delta\geq 1\\ x^{q}+x^{\wt{q}} &\delta \leq 1},
\end{eqalign}
where  $p,q,\wt{q}$ satisfy $q+\beta(q^{2}+q)<p<\wt{q}$.
\item Negative moments: For $p\geq 0$ and $\delta,x\geq 0$, we have the upper bound for $Q^{\delta}$:
\begin{eqalign}
\Expe{(Q^{\delta}(0,x))^{-p}}\lesssim \branchmat{\delta^{q-p} x^{-q}+\delta^{\wt{q}-p}x^{-\wt{q}}&\delta\geq 1\\ x^{-q}+x^{-\wt{q}}&\delta\leq 1 }.
\end{eqalign}
When $\delta\geq 1$ $p,q,\wt{q}$ satisfy the following: when $q,\wt{q}\in (0,1)$, we have $\zeta(q):=q-\beta(q^{2}-q)>p>\zeta(\wt{q})$ and when $q,\wt{q}\in (1, \frac{1}{\beta})$, we have $\zeta(q)>p>\wt{q}$. \\
When $\delta\leq 1$
and $q,\wt{q}\in (1,\frac{1}{\beta})$, we have the same constraint $\zeta(q)>p>\wt{q}$ but when  $q,\wt{q}\in (0,1)$, we have $\zeta(q)>p>\zeta(\wt{q})$.
\end{itemize}
\end{theorem}
\begin{remark}
One can obtain precise exponents by optimizing. From
\begin{eqalign}
&q+\beta(q^{2}+q)<p ,   \\
\tand &q-\beta(q^{2}-q)>p
\end{eqalign}
we solve to get
\begin{eqalign}
&q<\frac{1}{2}\para{\sqrt{1+\frac{4p+2}{\beta}+\frac{1}{\beta^{2}}  }-1-\frac{1}{\beta}   }  ,\end{eqalign}
and
\begin{eqalign}
q>    \frac{1}{2}\para{\sqrt{1+\frac{-4p+2}{\beta}+\frac{1}{\beta^{2}}  }-1-\frac{1}{\beta}   }
\end{eqalign}
respectively.
\end{remark}
 For the increments we only obtain concrete bounds for a restricted range on the moments.
\begin{theorem}
Let $p\in \left(-\frac{(1+\beta)^{2}}{4\beta}, \infty \right)$, the height $\delta\leq 1$ and $a,x>0$.
\begin{itemize}
    \item (Positive moments) For all $p>0$
\begin{eqalign}
\Expe{(Q^{\delta}(a,a+x))^{p}}\leq& c_{1}x^{p_{1}}C_{pos}(a,\delta)+c_{2}x^{p_{2}},
\end{eqalign}
for $p_{1},p_{2}$ constrained as 
\begin{equation}
 \zeta(-p_{1})+p>1\tand p_{2}>p.   
\end{equation}

\item (Negative moments) For all  $p\in \left(0,\frac{(1+\beta)^{2}}{4\beta}\right)$
\begin{eqalign}
\Expe{(Q^{\delta}(a,a+x))^{-p}}\leq& \tilde{c}_{1}x^{-\tilde{p}_{1}}+\tilde{c}_{2}x^{-\tilde{p}_{2}},
\end{eqalign}
for $\tilde{p}_{1},\tilde{p}_{2}$ constrained as
\begin{equation}
\zeta(\tilde{p}_{1})>p+1    \tand     p>\tilde{p}_{2}.
\end{equation}
\end{itemize}
The constants dependence is
\begin{eqalign}
C_{pos}(a,\delta)&:=\maxp{\delta^{\zeta(-p_{1})+p-2},C_{inf,1}(\alpha,\delta,\rho),C_{inf,2}(\alpha,\delta,\rho) }   \\
C_{neg}(a,\delta)&:=C_{sup}(a,\delta,\rho)
\end{eqalign}
with the constants $C_{inf,1}$ , $C_{inf,2}$,$C_{sup}$ coming from \Cref{prop:shiftedGMCmoments}. 
\end{theorem}
At the end we also study the decay rate of $\Proba{\abs{Q_{n+m}(x)-Q_{n}(x)}\geq r}$ and of the $L^{\ell}$-difference \\
$\Expe{\abs{Q_{n+m}(x)-Q_{n}(x)}^{\ell}}$ for $\ell\in (1,2)$, $x\in [0,1] ,r>0$ as $n\to +\infty$. Using the Cauchy-criterion  we will obtain that $Q$ is the $L^{\ell}$-Cauchy limit of $Q_{n}$. This is analogous to studying the rate of $L^{2}$-convergence for GMC measures $\eta_{n}\stackrel{L^{2}}{\to} \eta$ which is known to be  $\Expe{(\eta_{\e}(1)-\eta_{\e/2}(1))^{2}}\leq c\e^{2-\gamma^{2}}$ \cite{berestycki2021gaussian,rhodes2015lectures}. Extending it to $\ell=2$, will likely require the use of good/bad thick points as in the GMC case.
\begin{theorem}[Cauchy sequence]Fix $\gamma<\frac{1}{12\sqrt{2}}$, $\delta>0$ and let any strictly decreasing sequence $\delta_{n}=2^{\delta n}$.
 Fix an interval $[\e,M]$ for some $\e,M>0$ and $x\in [\e,M]$. Then the sequence  $\set{Q_{n}(x)}_{n\geq 1}$ is $L^{\ell}$-Cauchy for any $\ell\in (1,2)$
\begin{eqalign}
Q_{n}(x)\stackrel{L^{\ell}}{\to}Q(x),
\end{eqalign}
where $Q_{n}$ corresponds to the lower-truncated field $U_{\delta_{n}}^{1}$.
\end{theorem}
\begin{remark}
The constrain $\gamma<\frac{1}{12\sqrt{2}}$ originates from needing to use \cref{prop:maxmoduluseta} for $p\in (0,1)$.  
\end{remark}
\subsection{Acknowledgements}
We firstmost thank Eero Saksman and Antti Kupiainen. We had numerous useful discussions over many years. Eero asked us to compute the moments of the inverse and that got the ball rolling seven years ago. We also thank J.Aru, J.Junnila and V.Vargas for their discussions and comments.

\section{Notations and GMC results }\label{notations}
\subsection{White noise expansion on the unit circle and real line }
Gaussian fields with logarithmic covariance indexed over time or even functions $f\in \mathbb{H}$, for some Hilbert space $\mathbb{H}$, have appeared in many places in the literature: for the \textit{Gaussian free field} which is a Gaussian field indexed over the $L^{2}$ Hilbert space equipped with the Dirichlet inner product \cite{sheffield2007gaussian,lodhia2016fractional,duplantier2017log} and for the \textit{Gaussian multiplicative chaos} which is about Gaussian fields with logarithmic covariance plus some continuous bounded function $\ln\frac{1}{\abs{x-y}}+g(x,y)$ \cite{robert2010gaussian,bacry2003log,rhodes2014gaussian,aru2020gaussian}. In this work we will work with the Gaussian field indexed over sets in the upper half-plane found in the works \cite{barral2002multifractal,bacry2003log} and used in the random welding context in \cite[section 3.2]{AJKS}. For the hyperbolic measure $\dlambda(x,y):=\frac{1}{y^{2}}\dx \dy$ in the upper half-plane $\uhp$ we consider a Gaussian process $\set{W(A)}_{A\in B_{f}(\uhp)}$ indexed by Borel sets of finite hyperbolic area:
\begin{eqalign}
    B_{f}(\uhp):=\{A\subset \uhp: \lambda(A)<\infty; \supl{(x,y),(x',y')\in A}|x-x'|<\infty\}
\end{eqalign}
with covariance
\begin{eqalign}
    \Expe{W(A_{1})W(A_{2})}:=\lambda(A_{1}\cap A_{2}).
\end{eqalign}
Its periodic version $W_{per}$ has covariance
\begin{eqalign}
\Expe{W_{per}(A_{1})W_{per}(A_{2})}:=\lambda(A_{1}\cap \bigcup_{n\in \Z} (A_{2}+n)).
\end{eqalign}
\subsubsection*{Logarithmic field on the real line}Consider the triangular region
\begin{eqalign}
\mathcal{S}:=\{(x,y)\in \uhp: x\in [-1/2,1/2]\tand y>2\abs{x}\}
\end{eqalign}
and define the field on the real line
\begin{eqalign}
U(x):=W(\mathcal{S}+x)\tfor x\in \mb{R}.
\end{eqalign}
It has logarithmic covariance (\cite{bacry2003log})
\begin{eqalign}
\Exp[U(x)U(y)]=\ln\para{\frac{1}{\min(\abs{x-y},1)}}.    
\end{eqalign}
This is the main field we will work with throughout this article. Because of the divergence along the diagonal, we upper and lower truncate by considering the field evaluated over shifts of the region
\begin{eqalign}
\mathcal{S}_{\e}^{\delta}(x_{0}):=\{(x,y)\in \uhp: x\in [x_{0}-\frac{\delta}{2},x_{0}+\frac{\delta}{2}]\tand \para{2\abs{x_{0}-x}}\vee \e<y\leq \delta\}.
\end{eqalign}
and $\mathcal{S}_{\e}^{\delta}=\mathcal{S}_{\e}^{\delta}(0)$. The covariance of $U_{\e}^{\delta}(x)=W(\mathcal{S}_{\e}^{\delta}+x)$ is the following (\cite{bacry2003log}).
\begin{lemma}\label{linearU}
The truncated covariance satisfies for $\delta\geq \e$ and all $x_{1},x_{2}\in \mathbb{R}$:
\begin{eqalign}\label{eq:Ucovariance}
\Exp[U_{ \varepsilon}^{  \delta }(x_{1} )U_{ \varepsilon}^{  \delta }(x_{2} )  ]=R_{ \varepsilon}^{  \delta }(\abs{x_{1}-x_{2}}):=\left\{\begin{matrix}
\ln(\frac{\delta }{\varepsilon} )-\para{\frac{1}{\e}-\frac{1}{\delta}}\abs{x_{2}-x_{1}} &\tifc \abs{x_{2}-x_{1}}\leq \varepsilon\\
 \ln(\frac{\delta}{\abs{x_{2}-x_{1}}}) +\frac{\abs{x_{2}-x_{1}}}{\delta}-1&\tifc \e\leq \abs{x_{2}-x_{1}}\leq \delta\\
 0&\tifc  \delta\leq \abs{x_{2}-x_{1}}.
\end{matrix}\right.    .
\end{eqalign}
In the case of $\e=0$, we shorthand write $R^{  \delta }(\abs{x_{1}-x_{2}})=R_{0}^{  \delta }(\abs{x_{1}-x_{2}})$.
\end{lemma}
\begin{remark}
This implies that we have a bound for its difference:
\begin{eqalign}
\Expe{\abs{U_{\varepsilon_{1}}^{\delta}(x_{1})-U_{\varepsilon_{2}}^{\delta}(x_{2})}^{2}}\leq 2 \frac{|x_{2}-x_{1}|+|\varepsilon_{1}-\varepsilon_{2}|}{\varepsilon_{2}\wedge \e_{1}}.
\end{eqalign}
Therefore, being Gaussian this difference bound is true for any $a>1$
\begin{eqalign}
\Expe{|U_{\varepsilon_{1}}^{\delta}(x_{1})-U_{\varepsilon_{2}}^{\delta}(x_{2})|^{a}}\leq c\para{ \frac{|x_{2}-x_{1}|+|\varepsilon_{1}-\varepsilon_{2}|}{\varepsilon_{2}\wedge \e_{1}}     }^{a/2}.
\end{eqalign}
and so we can apply the Kolmogorov-Centsov lemma from \cite[Lemma C.1]{hu2010thick}:
\begin{eqalign}
|U_{\varepsilon_{1}}(x_{1})-U_{\varepsilon_{2}}(x_{2})|\leq M~ (\ln(\frac{1}{\varepsilon_{2}}))^{\zeta}\frac{(|x_{2}-x_{1}|+|\varepsilon_{1}-\varepsilon_{2}|)^{\gamma}}{ \varepsilon_{2}^{\gamma+\varepsilon} },
\end{eqalign}
where $\frac{1}{2}< \frac{\e_{1}}{\e_{2}}<2$, $0<\gamma<\frac{1}{2}$, $\varepsilon,\zeta>0 $ and $M=M(\varepsilon,\gamma,\zeta)$.
\end{remark}
Also, we can consider strictly increasing larger heights i.e. $U^{h}_{0}$ denotes the field of height $h\geq \delta_{1}$. The infinitely large field  $U^{\infty}_{0}$ does not exist because it has infinite variance.\\

This field is needed because it allows us to work on the entire real line and not worry about the boundary effects of working on the unit circle $[0,1]/ 0\sim 1$.

\subsubsection*{Infinite cone field}
For the region
\begin{eqalign}
\mathcal{A}_{\e}^{\delta}:=\{(x,y)\in \uhp: x\in [-\frac{\delta}{2},\frac{\delta}{2}]\tand \para{2\abs{x}}\vee \e<y\}
\end{eqalign}
we let $U_{\omega,\e}^{\delta}:=\omega_{\e}^{\delta}(x):=W(\mathcal{A}+x)$. This field was considered in \cite{bacry2003log}. It has the covariance
\begin{eqalign}\label{eq:exactscalinglogafieldcova}
\Exp[\omega_{ \varepsilon}^{  \delta }(x_{1} )\omega_{ \varepsilon}^{  \delta }(x_{2} )  ]=\left\{\begin{matrix}
\ln(\frac{\delta }{\varepsilon} )+1-\frac{\abs{x_{2}-x_{1}} }{\e}&\tifc \abs{x_{2}-x_{1}}\leq \varepsilon\\
 \ln(\frac{\delta}{\abs{x_{2}-x_{1}}}) &\tifc \e\leq \abs{x_{2}-x_{1}}\leq \delta\\
  0&\tifc  \delta\leq \abs{x_{2}-x_{1}}
\end{matrix}\right.    .
\end{eqalign}
This field is interesting because it has "exact scaling laws" as we will see below. And so it helps in estimating the moments for the GMC corresponding to $U$ from above.
\subsubsection*{Trace of the Gaussian free field on the unit circle}\label{sec:fieldonunitcircle}
For the wedge shaped region
\begin{eqalign}
H:=\{(x,y)\in \uhp: x\in [-1/2,1/2]\text{ and }y>\frac{2}{\pi}tan(\pi x)\},
\end{eqalign}
we define the \textit{trace GFF} $H$ to be
\begin{eqalign}
H(x):=W_{per}(H+x), x\in \R/\Z.
\end{eqalign}
A regularized version of $H$ comes from truncating all the lower scales. Let $A_{r,\varepsilon}:=\{(x,y)\in \uhp: r>y>\varepsilon\}$ and $H_{\varepsilon}^{r}:=H\cap A_{r,\varepsilon}$ then we set:
\begin{eqalign}
H_{\varepsilon}^{r}(x):=W(H_{\varepsilon}^{r}+x)
\end{eqalign}
and $H_{\varepsilon}:=h^{\infty}_{\varepsilon}$. Its covariance is as follows: For points $x,\xi\in [0,1)$, with 0 and 1 identified,  and $\varepsilon\leq r$ we find for $y:=|x-\xi|$:
\begin{eqalign}\label{eq:covarianceunitcircle}
\Exp[H_{\varepsilon}(x)H_{r}(\xi)]:=&\branchmat{2\log(2)+\log\frac{1}{2sin(\pi y)}  &\tifc y> \frac{2}{\pi} arctan(\frac{\pi}{2}\e)\\ \log(1/\e)+(1/2)\log(\pi^{2}\varepsilon^{2}+4)+\frac{2}{\pi}\frac{arctan(\frac{\pi}{2}\e)}{\varepsilon}&\tifc y\leq \frac{2}{\pi} arctan(\frac{\pi}{2}\e)\\
-\log(\pi)- y \varepsilon^{-1}-\log(cos(\frac{\pi}{2}y))&}.
\end{eqalign}
From the above covariance computations, we find the difference bound for $1/2\leq \frac{\e_{1}}{\e_{2}}\leq 2$:
\begin{eqalign}
\Exp[|H_{\varepsilon_{1}}(x_{1})-H_{\varepsilon_{2}}(x_{2})|^{2}]\lessapprox~\frac{|x_{2}-x_{1}|+|\varepsilon_{1}-\varepsilon_{2}|}{\varepsilon_{2}\wedge \e_{1} }
\end{eqalign}
Therefore, by Kolmogorov-Centsov \cite[Lemma C.1]{hu2010thick}, there is a continuous modification with the modulus:
\begin{eqalign}
|H_{\varepsilon_{1}}(x_{1})-H_{\varepsilon_{2}}(x_{2})|\leq M (\ln(\frac{1}{\varepsilon_{1}}))^{\zeta}\frac{(|x_{2}-x_{1}|+|\varepsilon_{1}-\varepsilon_{2}|)^{\gamma}}{ \varepsilon_{1}^{\gamma+\varepsilon} },
\end{eqalign}
where $0<\gamma<\frac{1}{2}$, $\varepsilon,\zeta>0 $ and $M=M(\varepsilon,\gamma,\zeta)$. A similar modulus is true for the field $U(x)$ on the real line constructed above.\\
This field is important because of the original problem being the conformal welding on the unit circle. So all our estimates need to return to the GMC with the field $H$.

\subsubsection*{Measure and Inverse notations}
 For a field $X$, let us define the corresponding Gaussian multiplicative chaos (GMC) as
\begin{eqalign}
\eta_{X}(I):=\liz{\e}\int_{I}e^{\overline{X}_{\e}(x)}\dx,\quad \overline{X}_{\e}:=\gamma X_{\e}-\frac{\gamma^{2}}{2}E[X_{\e}^{2}],\quad I\subset [0,1) ,
\end{eqalign}
provided the limit exists.

For the real-line field $U^{\delta}$ we use the following notation to match \cite{AJKS}:
\begin{eqalign}
\eta^{\delta}(A):=\liz{\e}\eta_{\e}^{\delta}(A):=\liz{\e}\int_{A}e^{\gamma U_{\e}^{\delta}(x)-\frac{\gamma^{2}}{2}\ln\frac{1}{\e}}\dx, \quad A\subset \mathbb{R}.
\end{eqalign}
Its inverse $Q^{\delta}:\Rplus\to \Rplus$ is defined as
\begin{eqalign}
Q^{\delta}(x)=Q^{\delta}_{x}=:=\infp{t\geq 0 : \eta^{\delta}\spara{0,t}\geq x}
\end{eqalign}
and we will also consider increments of the inverse over intervals $I:=(y,x)$
\begin{eqalign}
Q^{\delta}(I)=Q^{\delta}(y,x):=Q^{\delta}(x)-Q^{\delta}(y).
\end{eqalign}
For the inverse of the lower-truncated GMC $\eta_{\e}^{\delta}$, we will write
$Q^{\delta}_{\e,x}:=Q^{\delta}_{\e}(x).$

In the ensuing articles we will work with a sequence of truncations $\delta_{n}\to 0$, and so we will shorthand denote $U^{n}:=U^{\delta_{n}}$, $\eta^{n}$and $Q^{n}$. This is because as in \cite{AJKS}, we need to decouple nearby measures: to have the following two objects  
\begin{eqalign}
\eta^{\delta_{n+1}}(a_{n+1},b_{n+1})\tand \eta^{\delta_{n}}(a_{n},b_{n})    
\end{eqalign}
be independent with $a_{n+1}<b_{n+1}<a_{n}<b_{n}$ decreasing to zero, we simply require $b_{n+1}+\delta_{n+1}<a_{n}$. So we see that as $\delta_{n}\to 0$, this becomes easier to achieve. Whereas if we kept $\delta_{n}=\delta$, this would be false as $a_{n}\to 0$.\\
Similarly, when lower truncating $U_{n}:=U_{\delta_{n}}$, we will write $\eta_{n}$ and $Q_{n}$. By slight abuse of notation, we denote
\begin{eqalign}
U(s)\cap U(t):=W(\para{U+s}\cap \para{U+t})    \tand U(s)\setminus U(t):=W(\para{U+s}\setminus \para{U+t}).
\end{eqalign}
\subsubsection*{Semigroup formula}\label{not:semigroupformula}
In the spirit of viewing $Q_{a}=Q(a)$ as a hitting time we have the following semigroup formula and notation for increments
\begin{eqalign}
Q^{\delta}(y,x)=Q^{\delta}(x)-Q^{\delta}(y)=\infp{t\geq 0: \etamu{Q^{\delta}(y),Q^{\delta}(y)+t}{\delta}\geq x-y }=:Q_{x-y}^{\delta }\bullet Q_{y}^{\delta}.
\end{eqalign}
Generally for any nonnegative $T\geq 0$ we use the notation
\begin{eqalign}
Q_{x}^{\delta} \bullet T:=\infp{t\geq 0: \etamu{T,T+t}{\delta}\geq x }.
\end{eqalign}
To be clear the notation $Q_{x}^{\delta }\bullet Q_{y}^{\delta}$ is \textit{not} equal to the composition $Q^{\delta }\circ Q_{y}^{\delta}=Q^{\delta }(Q_{y}^{\delta})=Q^{\delta }(a),$ which is the first time that the GMC $\eta$ hits the level $a:=Q_{y}^{\delta}$.

\subsubsection*{Scaling laws}
Let $\omega^{\delta,\lambda}_{\e}(x)$ be the field with the following covariance \begin{eqalign}\label{eq:truncatedscaledomega}
\Expe{\omega^{\delta,\lambda}_{\e}(x_{1})\omega^{\delta,\lambda}_{\e}(x_{2})}=\left\{\begin{matrix}
\ln(\frac{\delta }{\varepsilon} )+1-\frac{1}{\e}\abs{x_{2}-x_{1}}&\tifc \abs{x_{2}-x_{1}}\leq \varepsilon\\ ~\\
 \ln(\frac{\delta}{\abs{x_{2}-x_{1}}}) &\tifc \e\leq \abs{x_{2}-x_{1}}\leq \frac{\delta}{\lambda}\\~\\
  0&\tifc \frac{\delta}{\lambda}\leq \abs{x_{2}-x_{1}}
\end{matrix}\right.    .
\end{eqalign}
In our study of the scaling laws we will use the following statement (\cite[theorem 4]{bacry2003log}, \cite[proposition 3.3]{robert2010gaussian}.)
\begin{proposition}\label{exactscaling}\label{prop:GMClogscalinglaw}
For $\lambda\in (0,1)$ and fixed $x_{0}$ and all Borel sets $A\subset B_{\delta/2}(x_0)$ we have
\begin{eqalign}
 \set{\eta^{\delta}_{\omega}(\lambda A)}_{A\subset B_{\delta/2}(x_0)}\eqdis \set{\lambda e^{\overline{\Omega_{\lambda}}}\eta^{\delta,\lambda}_{\omega}(A)}_{A\subset B_{\delta/2}(x_0) },
\end{eqalign}
where  $\overline{\Omega_{\lambda}}:=\gamma\sqrt{ \ln(\frac{1}{\lambda})}N(0,1)-\frac{\gamma^{2}}{2}\ln(\frac{1}{\lambda})$ and the measure $\eta_{\omega_{\e}}^{\delta,\lambda}$ has the underlying field $\omega^{\delta,\lambda}_{\e}(x)$
\end{proposition}
 In the rest of the article, we let
\begin{eqalign}\label{logfield}
&G_{\lambda}:=\frac{1}{\lambda}\expo{-\overline{\Omega_{\lambda}}}.
\end{eqalign}
We will also use the field $U_{ \varepsilon}^{\delta, \lambda}$ with covariance $R_{\e}^{\delta,\lambda}(\abs{x_{1}-x_{2}}):=\Expe{U_{ \varepsilon}^{  \delta,\lambda }(x_{1} )U_{ \varepsilon}^{  \delta,\lambda }(x_{2} )  }$
\begin{eqalign}\label{eq:truncatedscaled}
R_{\e}^{\delta,\lambda}(\abs{x_{1}-x_{2}})=\left\{\begin{matrix}
\ln(\frac{\delta }{\varepsilon} )-\para{\frac{1}{\e}-\frac{1}{\delta}}\abs{x_{2}-x_{1}}+(1-\lambda)(1-\frac{\abs{x_{2}-x_{1}}}{\delta})&\tifc \abs{x_{2}-x_{1}}\leq \varepsilon\\ ~\\
 \ln(\frac{\delta}{\abs{x_{2}-x_{1}}})-1+\frac{\abs{x_{2}-x_{1}}}{\delta}+(1-\lambda)(1-\frac{\abs{x_{2}-x_{1}}}{\delta}) &\tifc \e\leq \abs{x_{2}-x_{1}}\leq \frac{\delta}{\lambda}\\~\\
  0&\tifc \frac{\delta}{\lambda}\leq \abs{x_{2}-x_{1}}
\end{matrix}\right.    .
\end{eqalign}
\begin{remark}\label{rem:negativecov}
We note that after $\abs{x_{2}-x_{1}}\geq \delta$ this covariance is negative and thus discontinuous at $\abs{x_{2}-x_{1}}=\frac{\delta}{\lambda}$ and so the GMC $\eta^{\delta,\lambda}$ cannot be defined (so far as we know in the literature, GMC has been built for positive definite covariances \cite{allez2013lognormal}). So we are forced to evaluate those fields only over sets $A$ with length $\abs{A}\leq \delta$ (see \cite[lemma 2]{bacry2003log} where they also need this restriction to define the scaling law.).
\end{remark}
 The $U_{ \varepsilon}^{\delta, \lambda}$ is related to $U^{\delta}_{\e}$ in the following way
\begin{eqalign}\label{eq:truncatedscalinglawGMC}
U^{\delta}_{\lambda\e}(\lambda x)\eqdis Z_{\lambda}+  U_{ \varepsilon}^{  \delta,\lambda }(x),
\end{eqalign}
where $Z_{\lambda}:=N(0,\ln\frac{1}{\lambda}-1+\lambda)$ is a Gaussian independent of $ U_{ \varepsilon}^{  \delta,\lambda }$. The corresponding lognormal is denoted by
\begin{eqalign}\label{logfieldshifted}
c_{\lambda}:=&\frac{1}{\lambda}\expo{-\overline{Z_{\lambda}}}.
\end{eqalign}
We can now restate \eqref{eq:truncatedscalinglawGMC}:
\begin{lemma}\label{lem:scalinglawudelta}[Scaling transformation for $U^{\delta}_{\e}$]
For $\lambda\in (0,1)$, $x_{0}\in\mathbb{R}$ and all Borel sets $A\subset B_{\delta/2}(x_0)$ we have
\begin{eqalign}
 \set{\eta^{\delta}_{U}(\lambda A)}_{A\subset B_{\delta/2}(x_0)}\eqdis \set{\lambda e^{\overline{Z_{\lambda}}}\eta^{\delta,\lambda}_{U}(A)}_{A\subset B_{\delta/2}(x_0) },  \quad \eta_{U}^{\delta,\lambda}:=\eta_{U^{\delta,\lambda}_{\e}}.\end{eqalign}
\end{lemma}
 For any $p\in \R$ we have the moment formula
\begin{eqalign}
\Expe{\expo{p\overline{Z_{\lambda}}}  }=\expo{p\beta r_{\lambda} (p-1)}.
\end{eqalign}
More generally, for positive continuous bounded function $g_{\delta}:\Rplus\to \Rplus$ with $g_{\delta}(x)=0$ for all $x\geq \delta$, we let
\begin{eqalign}
\Expe{U_{ \varepsilon}^{\delta,g }(x_{1} )U_{ \varepsilon}^{  \delta,g }(x_{2} )  }=\left\{\begin{matrix}
\ln(\frac{\delta }{\varepsilon} )-\para{\frac{1 }{\e}-\frac{1}{\delta}}\abs{x_{2}-x_{1}}+g_{\delta}(\abs{x_{2}-x_{1}})&\tifc \abs{x_{2}-x_{1}}\leq \varepsilon\\
 \ln(\frac{\delta}{\abs{x_{2}-x_{1}}})-1+\frac{\abs{x_{2}-x_{1}}}{\delta}+g_{\delta}(\abs{x_{2}-x_{1}}) &\tifc \e\leq \abs{x_{2}-x_{1}}\leq \delta\\
  0&\tifc \delta\leq \abs{x_{2}-x_{1}}
\end{matrix}\right.  .
\end{eqalign}

\subsubsection*{Singular integral}
Let the singular/fusion integrals for GMC measure be defined by (\cite[lemma A.1]{david2016liouville})
\begin{eqalign}\label{eq:singularintegralfusion}
\eta_{R(a)}(A):=\int_{A}e^{\gamma^{2}\Expe{U(s)U(a)}}e^{\gamma U(s)}\ds,\hfil A\subset \Rplus,\hfil a\in\Rplus.\end{eqalign}
As before, let $Q_{R}(x)$ denote its inverse.

Using the Girsanov/tilting-lemma (\cite[Lemma 2.5]{berestycki2021gaussian}) we have
\begin{eqalign}
\Expe{\ind{\eta(x)\geq t}e^{\thickbar{U}(a)}}=\Proba{\eta_{R(a)}(x)\geq t}.
\end{eqalign}

\subsubsection*{Filtration notations}
Following \cite[section 4.2]{AJKS}, for a Borel set $S\subset \mathbb{H}$ we define $\mathcal{U}_{S}=\mathcal{U}(S)$ to be the $\sigma$-algebra generated by the randoms variables $U(A)$, where $A$ runs over Borel subsets $A\subseteq S$. For short if $X\in \CU_{S}$ we will call such a variable \textit{measurable}. Moreover, since we will be using the triangular region $\mathcal{S}$ over intervals $[a,b]$, we define
\begin{eqalign}
\mathcal{U}_{\e}^{\delta}([a,b]):=\mathcal{U}(\bigcup_{x_{0}\in [a,b]}\mathcal{S}_{\e}^{\delta}(x_{0})).
\end{eqalign}
We fix decreasing sequence $(\delta_{n}=\rho_{*}^{n})_{n\geq 1}$ for some $\rho_{*}<1$. We denote the $\sigma-$algebra of upper truncated fields
\begin{eqalign}
\CU^{n}=\CU^{n}_{\infty}:=\CU_{\mathcal{S}^{\delta_{n}}_{0}}.
\end{eqalign}
All the lower truncations are measurable with respect to the filtration of the larger scales
\begin{eqalign}
U^{n}(s)=U^{\delta_{n}}_{0}(s)\in \CU^{k}([a,b]) \tfor n\geq k,s\in [a,b],
\end{eqalign}
and the same is true for the measure and its inverse
\begin{eqalign}
\set{\eta^{n}(a,b)\geq t}\in \CU^{k}([a,b])\tand \set{Q^{n}(0,x)\geq t}\in \CU^{k}([0,t]).  \end{eqalign}
We define measurability with respect to random times in a standard fashion. Namely, for a single time $Q^{k}(b)$ we have the usual notion of the stopping time filtration
\begin{eqalign}\label{eq:filtrationnotation}
\mathcal{F}\para{[0,Q^{k}(b)]}:=\left\{A\in \bigcup_{t\geq 0}  \CU^{k}([0,t]): A\cap\{Q^{k}(b) \leq t\}\in  \CU^{k}([0,t]), \forall t\geq 0\right\}.
\end{eqalign}
For example, $\CU^{n}\para{[0,Q^{k}(b)-s]}\subset \mathcal{F}\para{[0,Q^{k}(b)]}$ for all $n\geq k$ and $s\in [0,Q^{k}(b)]$. We also have, for $n\geq k$,
\begin{eqalign}
\set{Q^{k}(b)\geq Q^{n}(a)+s} \in  \mathcal{F}\para{[0,Q^{k}(b)]}.
\end{eqalign}

This perspective helps us achieve a decoupling. For two interval increments $Q^{k}(a_{k},b_{k}),Q^{n}(a_{n},b_{n})$ from possibly different scales $n\geq k$ , we we will need to consider the \textit{gap event}
\begin{eqalign}
G_{k,n}:=\set{Q^{k}(a_{k})-Q^{n}(b_{n})\geq \delta_{n}} \in  \mathcal{F}\para{[0,Q^{k}(a_{k})]}.
\end{eqalign}

 \part{General properties of the inverse}\label{part:generalpropertiesinverse}
 In this section we prove some basic properties for the inverse $Q:=Q_{\eta}$.
\section{Existence of the inverse}
Let us first note that the random map $Q: [0,\infty)\to [0,\infty)$ exists and is continuous. It exists because $\eta:[0,\infty)\to [0,\infty)$ is non-atomic and has full support (\cite[theorem 1]{bacry2003log}) and so it is strictly increasing. It is continuous because as mentioned in \cite[Theorem 3.7]{AJKS}, the map $\eta$ is bi-\Holder.\\
We can also view the inverse $Q$ as the weak-limit of $Q_{\e}$, the inverses of lower truncated $\eta_{\e}$:
\begin{eqalign}
\Proba{Q_{\e}(x)\geq t }=\Proba{x\geq \eta_{\e}(0, t) } \to \Proba{x\geq \eta(0, t) }= \Proba{Q(x)\geq t}
\end{eqalign}
using that $\eta$ is the $L^{2}$-Cauchy limit of $\eta_{\e}$ (eg. \cite[Theorem 2.1] {berestycki2021gaussian}).\\
Finally, in \Cref{rateconvCauchy}, we also give a more direct proof that the inverse $(Q_{\e})_{\e\geq 0}$ converges in  $L^{\ell}$ as $\e\to 0$ for $\ell\in (1,2)$ and $\gamma<\frac{1}{12\sqrt{2}}$.
\section{The \sectm{$\delta$}-Strong Markov property }\label{inversemeasureproperties}\label{INVARIANCE_PROPERTIES}
GMC does not satisfy Strong Markov Property (SMP) (see \Cref{notMarkov}). One of our main technical tools is the following replacement, a version of the strong Markov property with a time delay.
\begin{proposition} \label{it:SMP} (\textit{$\delta$-Strong Markov property})
For parameters $\delta_{1},\delta_{2},a,t,\e\geq 0$ and $r\geq \minp{\delta_{1},\delta_{2}}$, we have the independence:
\begin{equation}\label{deltaSMP}
\etamul{Q_{\e}^{\delta_{2}}(a)+r,Q_{\e}^{\delta_{2}}(a)+r+t}{\delta_{1}}{\e}\indp Q_{\e}^{\delta_{2}}(a).\tag{$\delta$-SMP}
\end{equation}
We will call it the \textit{$\delta$-Strong Markov property} ($\delta$-SMP). In other words, $Q_{s}^{\delta_{1}}\bullet \para{Q_{a}^{\delta_{2}}+r}$ is independent of $Q_{a}^{\delta_{2}}$.
\end{proposition}
\begin{remark}
We prove this for general $\delta_{1},\delta_{2}\geq 0$ because we will need this in the decoupling work.    
\end{remark}
\begin{remark}
We have to take $r\geq \delta$ since the independence $\etam{\ell+r,\ell+r+t}\indp\etam{0,\ell}$ for deterministic $\ell$ only holds in this range of $r$.
\end{remark}
\begin{remark}\label{notMarkov}
The usual SMP with $r=0$
\begin{eqalign}
\eta_{\e}^{\delta}\para{Q_{\e}^{\delta}(a),Q_{\e}^{\delta}(a)+t}\indp Q_{\e}^{\delta}(a), a,t\geq 0,
\end{eqalign}
does not hold because even the field $U_{\e}^{\delta}$ does not satisfy Markov property. Indeed, a centered Gaussian process is Markov if and only if its covariance  function $\Gamma: \mathbb{R}\times\mathbb{R} \to \mathbb{R}$ satisfies the equality \cite[theorem 8.4]{ basu2003introduction}:
\begin{eqalign}
\Gamma(s,u)\Gamma(t,t)=\Gamma(s,t)\Gamma(t,u)
\end{eqalign}
 for all $s<t<u$. The covariance for $U_{\e}^{\delta}(x)$ is
\begin{eqalign}
\Gamma(x_{1},x_{2}):=\Exp[U_{ \varepsilon}^{  \delta }(x_{1} )U_{ \varepsilon}^{  \delta }(x_{2} )  ]=\left\{\begin{matrix}
\ln\para{\frac{\delta}{\varepsilon} }+\frac{x_{2}-x_{1}}{\delta}-\frac{x_{2}-x_{1}}{\varepsilon} &, \abs{x_{2}-x_{1}}\leq \varepsilon\\
\ln(\frac{\delta}{x_{2}-x_{1}}) +\frac{x_{2}-x_{1}}{\delta}-1&, \delta>\abs{x_{2}-x_{1}}\geq \e
\end{matrix}\right.
\end{eqalign}
so for $s<t<u$ with $\abs{u-s}<\e$ we get the strict inequality:
\begin{eqalign}
\Gamma(s,t)\Gamma(t,u)=\Gamma(s,u)\Gamma(t,t)+\para{\frac{1}{\delta}-\frac{1}{\e}}^{2}\para{t-s}\para{u-t}>\Gamma(s,u)\Gamma(t,t).
\end{eqalign}
So its integrated process $\eta(L)=\int_{0}^{L}e^{U_{\e}^{\delta}(x)}\dx$ will likely not be Markov either since $\eta$ is a strictly increasing injective functional of the process and injectivity preserves Markov property. This is only a heuristic. 
\end{remark}
\begin{proof}[Proof of \ref{it:SMP}: the $\delta$-Strong Markov property for $\sectm{\eta}$]
For ease of notation we write $\eta^{1}=\eta_{\e}^{\delta_{1}},Q^{2}=Q_{\e}^{\delta_{2}}$. This proof is in the spirit of the usual proof of strong Markov property for Brownian motion via dyadic decomposition of the stopping time. Namely, we use a discrete canonical approximation $T_{n}\downarrow Q(a)$:
\begin{eqalign}\label{dyadicapproximationTn}
T_{n}:=\frac{m+1}{2^{n}}\twhen \frac{m}{2^{n}}\leq Q^{2}(a) < \frac{m+1}{2^{n}}.
\end{eqalign}
For fixed $n$ the range of $T_{n}$ is $D_{n}(0,\infty)$, all the dyadics of nth-scale in the open ray $(0,\infty)$.

We fix $u_{1},u_{2}\geq 0$ and decompose:
\begin{multline}
\Proba{\eta^{1}([T_{ n}+r,T_{ n}+r+t])\geq u_{1}, T_{n}\geq u_{2}}=\\\sum_{\ell\in D_{n}(u_{2},\infty)}\Proba{\eta^{1}\para{\ell+r,\ell+r+t}\geq u_{1}, \eta^{2}\para{0,\ell-\frac{1}{2^{n}}}\leq a<\eta^{2}\para{0,\ell} }.
\end{multline}
Now here we use the independence $\eta^{1}\para{\ell+r,\ell+r+t}\indp\eta^{2}\para{0,\ell}$ for $r\geq \minp{\delta_{1},\delta_{2}}$ to get the decoupling
\begin{multline}
\sum_{\ell\in D_{n}(u_{2},\infty)}\Proba{\eta^{1}\para{\ell+r,\ell+r+t}\geq u_{1}}\Proba{ \eta^{2}\para{0,\ell-\frac{1}{2^{n}}}\leq a<\eta^{2}\para{0,\ell} }=\\\Proba{\eta^{1}([T_{ n}+r,T_{ n}+r+t])\geq u_{1}}\Proba{T_{n}\geq u_{2}}.
\end{multline}
Finally, in terms of the limits, we use that $\eta(x,y)=\eta(y)-\eta(x)$ is continuous in $(x,y)$, and so we can take the limit $T_{n}\downarrow Q(a)$ and apply Portmanteau theorem.
\end{proof}
\subsubsection{Decoupling}
Next we go over the decoupling that will be used in the proof of moments for the inverse. We use the filtration notation $\mathcal{F}\para{[0,Q(a)]}$ from \cref{eq:filtrationnotation}. For fixed $\ell\geq 0$, we split
\begin{eqalign}
 \eta\para{\ell,\ell+t}=&\int_{0}^{t}e^{\overline{U(\ell+s)\cap U(\ell)}}e^{\overline{U(\ell+s)\setminus U(\ell)}}\ds\\ =&\int_{0}^{t}e^{\overline{U(\ell+s)\cap U(\ell)}}\deta_{+}(s),   
\end{eqalign}
where we let $\deta_{+}(s):=e^{\overline{U(\ell+s)\setminus U(\ell)}}\ds$. This measure $\eta_{+}$ is a GMC measure for the field $X_{s}:=U(\ell+s)\setminus U(\ell)$ that has covariance
\begin{eqalign}
\Expe{X_{s}X_t}=R(\abs{t-s})-R(\maxp{t,s})\geq 0.    
\end{eqalign}
For the next proposition, for fixed Gaussian field $X$ we use the notation
\begin{eqalign}
\Exp_{X}\tand \Prob_{X},    
\end{eqalign}
to denote that we integrate with respect to the law of $X$.
\begin{proposition}\label{prop:decouplingsplit}
We fix an event $A(Q(a))\in \mathcal{F}\para{[0,Q(a)]}$. We have   
\begin{eqalign}
&\Proba{\eta([Q(a),Q(a)+t])\geq u, A(Q(a))}    \\
=&\Exp_{\tilde{U}}\Prob_{U}\spara{\int_{0}^{t}e^{U(Q(a)+s)\cap U(Q(a))}\deta_{\tilde{U}}(s)\geq u, A(Q(a)) },
\end{eqalign}
where the outer expectation $\Exp_{\tilde{U}}$ is with respect to an independent realization $\tilde{U}$ of U field i.e. $\set{\tilde{U}(s)\setminus \tilde{U}(0) }_{s\geq 0}\eqdis\set{U(s)\setminus U(0) }_{s\geq 0}$, the inner probability is wrt to the field $U$ and 
\begin{eqalign}
\deta_{\tilde{U}}(s):=e^{\overline{\tilde{U}(s)\setminus \tilde{U}(0)}}\ds.  \end{eqalign}
\end{proposition}
\begin{proof}
We again start by approximating and decomposing over the dyadics
\begin{eqalign}\label{eq:decspli1}
&\Proba{\eta([T_{ n},T_{ n}+t])\geq u, A(T_{n})}\\
=&\sum_{\ell\in D_{n}(0,\infty)}\Proba{\eta\para{\ell,\ell+t}\geq u, A(\ell),\eta\para{0,\ell-\frac{1}{2^{n}}}\leq a<\eta\para{0,\ell} }\\
=&\sum_{\ell\in D_{n}(0,\infty)}\Proba{\int_{0}^{t}e^{U(\ell+s)\cap U(\ell)}\deta_{+}(s)\geq u, A(\ell),\eta\para{0,\ell-\frac{1}{2^{n}}}\leq a<\eta\para{0,\ell} }.
\end{eqalign}
Since $X=\set{U(\ell+s)\setminus U(\ell) }_{s\geq 0}$ is independent of $\mathcal{U}([0,\ell])$, and so we have 
\begin{eqalign}
X\eqdis\set{\tilde{U}(s)\setminus \tilde{U}(0) }_{s\geq 0}=:\tilde{X}    
\end{eqalign}
for an independent realization $\tilde{U}$ of field $U$. We can use the tower property to split
\begin{eqalign}\label{eq:decspli2}
\eqref{eq:decspli1}=&\sum_{\ell\in D_{n}(0,\infty)}\Exp_{\tilde{X}}\Prob_{U}\spara{\int_{0}^{t}e^{U(\ell+s)\cap U(\ell)}\deta_{+}(s)\geq u, A(\ell),\eta\para{0,\ell-\frac{1}{2^{n}}}\leq a<\eta\para{0,\ell} }\\
=&\Exp_{\tilde{X}}\Prob_{U}\spara{\int_{0}^{t}e^{U(T_{n}+s)\cap U(T_{n})}\deta_{+}(s)\geq u, A(T_{n}) },
\end{eqalign}
where $\deta_{+}$ is independent with respect to $\mathcal{U}([0,T_{n}])$.
    
\end{proof}

\subsection{Nonlinear expected value}
 For the usual GMC we know that its expected value is linear $\Expe{\eta(a,b)}=b-a$ \cite[Theorem 2]{bacry2003log}.  Using the $\delta$-(SMP) property, we obtain a nonlinear relation for the expected value of the inverse.
\begin{proposition}\label{differencetermunshifted} For $a>0$ and $r\geq \delta $
\begin{eqalign}\label{eq:nonlinearexpect}
\Expe{\eta^{\delta}(Q^{\delta}(a),Q^{\delta}(a)+r)}-r=\Expe{Q^{\delta}(a)}-a&=\int_{0}^{\infty}\Proba{ Q_{R(t)}^{\delta}(a)\leq t \leq  Q^{\delta}(a)}\dt\\
&=\int_{0}^{\infty}\Proba{ \eta^{\delta}(t)\leq a \leq  \eta_{R(t)}^{\delta}(t) }\dt>0,
\end{eqalign}
for $\eta_{R(t)}(t):=\int_{[0,t]}e^{\gamma^{2}\Expe{U_{\e}(s)U_{\e}(t)}}e^{\gamma U_{\e}(s)}\ds$. In particular, for any $a>0$ we have $\Expe{Q^{\delta}(a)}>a$. 
\end{proposition}
\begin{remark}
This proposition shows that the GMC $\eta^{\delta}$ does \textit{not} satisfy a "strong" translation invariance i.e. $\Expe{\eta(Q^{\delta}(a),Q^{\delta}(a)+r)}\neq  r$. So the same is true for $Q^{\delta}(a,a+t)$ since
\begin{eqalign}
\Expe{Q^{\delta}(a,a+t)}=&\int_{0}^{\infty}\Proba{t>\eta^{\delta}(Q^{\delta}(a),Q^{\delta}(a)+r)}\dr\\
\neq& \int_{0}^{\infty}\Proba{t>\eta^{\delta}(0,r)}=\Expe{Q^{\delta}(t)}.
\end{eqalign}
It also shows that $\Expe{Q^{\delta}(a)}$ is a nonlinear function of $a$.
\end{remark}
\begin{proof}[Proof of \Cref{differencetermunshifted}]
This proof is in the spirit of the standard derivation of Wald's equation for Brownian motion \cite[theorem 2.44]{mortersperes2010brownian}. We first expand $\eta_{\e}$ and apply Tonelli
\begin{eqalign}
\Expe{\eta_{\e}(Q^{\delta}_{\e}(x)+r)}=&\int_{\delta}^{\infty}\Expe{ \ind{t\leq Q^{\delta}_{\e}(x)+r}e^{\thickbar{U}_{\e}^{\delta}(t)}  }\dt\\=&\int_{\delta}^{\infty}\Expe{ \ind{\eta_{\e}^{\delta}(t-r)\leq x}e^{\thickbar{U}_{\e}^{\delta}(t)}  }\dt.
\end{eqalign}
We note that the indicator $\ind{\eta_{\e}^{\delta}(t-r)\leq x}$ is independent of $\thickbar{U}_{\e}^{\delta}(t)$ and so we get the split into products
\begin{eqalign}
\Expe{\eta_{\e}^{\delta}(Q^{\delta}(x)+r)}=&\int_{0}^{\infty}\Proba{\eta_{\e}^{\delta}(t-r)\leq x}\Expe{e^{\thickbar{U}_{\e}^{\delta}(t)}  }\dt\\=&\int_{0}^{\infty}\Proba{\eta_{\e}^{\delta}(t-r)\leq x}\dt\\=&\Expe{Q^{\delta}_{\e}(x)}+r.
\end{eqalign}
So by taking limit in $\e\to 0$ we get the first relation. For the second equality we have
\begin{eqalign}
a&=\Expe{\eta^{\delta}_{\e}(Q^{\delta}_{\e}(a))}=\int_{0}^{\infty}\Expe{\ind{t\leq Q^{\delta}_{\e}(a)}e^{\thickbar{U}^{\delta}_{\e}(t)} }\dt\stackrel{}{=}\int_{0}^{\infty}\Proba{\eta_{\e,R(t)}^{\delta}(t)\leq a}\dt,
\end{eqalign}
where we used the tilting lemma  (\cite[Lemma 2.5]{berestycki2021gaussian}) in the last equality. So since
\begin{eqalign}
\eta_{\e,R(t)}(t)=\int_{[0,t]}e^{\gamma^{2}\Expe{U_{\e}(s)U_{\e}(t)}}e^{\gamma \bar{U}_{\e}(s)}\ds>  \int_{[0,t]}e^{\gamma \bar{U}_{\e}(s)}\ds=\eta_\e(t),  \forall \e\in[ 0 ,\delta),
\end{eqalign}
by the layercake representation and taking $\e\to 0$ we obtain strict positivity
\begin{eqalign}
0<\lim_{\e\to 0}\int_{0}^{\infty}\Proba{\eta_{\e}(t) \leq a\leq  \eta_{\e,R(t)}(t)}\dt=&\int_{0}^{\infty}\Proba{\eta(t) \leq a\leq  \eta_{R(t)}(t)}\dt\\
=&\int_{0}^{\infty}\Proba{ Q_{R(t)}(a)\leq t \leq  Q(a)}\dt\\
=&\Expe{Q(a)}-a.
\end{eqalign}
\end{proof}
\begin{remark}
Ideally, we would like to check whether the RHS of \Cref{eq:nonlinearexpect} is uniformly bounded in $a>0$
\begin{eqalign}
\supl{a>0}\int_{0}^{\infty}\Proba{\eta(t) \leq a\leq  \eta_{R(t)}(t)}\dt    <\infty\tor =\infty,
\end{eqalign}
but it is unclear of how the window $[\eta(t) ,\eta_{R(t)}(t)]$ grows as $t\to +\infty$.
\end{remark}
 \subsection{Large $a>0$}
Since it is unclear how the difference
\begin{eqalign}
\Expe{Q(a)}-a    
\end{eqalign}
behaves for fixed $a$, we can try to study it for large $a$. First, we will show that $\Expe{\frac{Q(a)}{a}}-1\to 0$. We apply the following limiting ergodic statements to GMC and its inverse \cite[lemma 1]{allez2013lognormal}.
\begin{lemma}\label{ergodiclemma}
Let $M$ be a stationary random measure on $\R$ admitting a moment of order $1+\delta$ for $\delta>0$. There is a nonnegative integrable random variable $Y\in L^{1+\delta}$ such that, for every bounded interval $I\subset \R$, $$\lim_{T \to \infty} \frac{1}{T} M\left(T I\right) = Y |I|\quad \text{almost surely and in }L^{1+\delta},$$
where  $|\cdot|$ stands for the Lebesgue measure on $\R$. As a consequence, almost surely the random measure $$A\in \mathcal{B}(\R)\mapsto \frac{1}{T}M(TA)$$ weakly converges towards $Y|\cdot|$ and $\E_Y[M(A)]=Y |A|$ ($\E_Y[\cdot]$ denotes the conditional expectation with respect to $Y$).
\end{lemma}
\noindent For GMC  and its inverse the $Y$ variable is identically one.
 \begin{proposition}
 We have the following a.s./$L^{\frac{1}{\beta}}$-limits for $x\geq 0$ and $\delta\geq\e\geq 0$
\begin{eqalign}\label{eq:GMCergodiclimit}
\lii{T}\frac{\eta_{\e}^{\delta}(Tx)}{T}= x,
\end{eqalign}
and
\begin{eqalign}\label{eq:inverseergodiclimit}
\lii{T}\frac{Q_{\e}^{\delta}(Tx)}{T}= x.
\end{eqalign}
\end{proposition}
\begin{proof}
One way to prove \Cref{eq:GMCergodiclimit} is to use the independence of distant GMCs. For simplicity we take $\delta=1,\ \e=0$. Split $\frac{\eta^{1}(0,n)}{n}$ into alternating even and odd intervals $I_{k}=[k,k+1]$ to get two independent sequences of $(\eta(I_{k}))_{k~even}$ and $(\eta(I_{k}))_{k~odd}$. Then we apply strong law of large numbers separately to each squence to get convergence to
\begin{eqalign}
\frac{\eta^{1}(0,n)}{n}\stackrel{a.s.}{\to}\frac{1}{2}\Expe{\eta^{1}(0,1)}+\frac{1}{2}\Expe{\eta^{1}(1,2)}=1.  \end{eqalign}
To establish \eqref{eq:inverseergodiclimit}, we take any sequence $t_{n}\to +\infty$ and rearrange
\begin{eqalign}\label{eq:inverseasrate}
\Proba{  \set{\omega \in \Omega : \limsup_{n\to\infty}\abs{ \frac{Q(t_{n}x)}{t_{n}x}(\omega) - 1 } >r }}
\end{eqalign}
to write it in terms of $\eta$ and thus show it tends to zero for each $r>0$
\begin{eqalign}
\eqref{eq:inverseasrate}\leq\Prob\Big[&  \Big\{\omega \in \Omega : \limsup_{n\to\infty} \frac{\eta(t_{n}x(1+r))}{t_{n}x(1+r)}(\omega) - 1  <\frac{-r}{1+r}\\
&\quad\tor  \limsup_{n\to\infty} \frac{\eta(t_{n}x(1-r))}{t_{n}x(1-r)}(\omega) - 1  >\frac{r}{1-r}\Big\}\Big] .
\end{eqalign}
\end{proof}
\noindent So we obtain that 
\begin{eqalign}
\Expe{\frac{Q(a)}{a}}-1\to 0.    
\end{eqalign}
as desired.

\noindent But ideally we want to study the difference $\Expe{Q(a)}-a $. Here we at least get a bound on the possible growth for integer $a=n$.
\begin{theorem}
For $\frac{1}{2}>\beta$, we have the bound
\begin{eqalign}
\abs{\Expe{Q(n)}-n}\leq c_{\epsilon} n^{1+\epsilon},    
\end{eqalign}
for small $\epsilon>0$ (sufficiently small so that $\frac{1}{2(1+\epsilon)}>\beta$) and the constant $c_{\epsilon}$ diverges as $\epsilon\to 0$. 
\end{theorem}
\begin{proof}
We start with reverting to GMC
\begin{eqalign}\label{eq:deviationsumbound}
\Expe{\abs{Q(n)-n}}\leq& 1+\sum_{k=1}^{\infty}\Proba{\abs{Q(n)-n}\geq k}\\
\leq& 1+\sum_{k=1}^{\infty}\Proba{n+k-\eta(n+k)\geq k}+\sum_{k=1}^{n}\Proba{\eta(n-k)-(n-k)\geq k}
\end{eqalign}
\noindent We will use \cite[theorem 3]{rosenthal1970subspaces}.
\begin{theorem}
For $p > 2$ and independent, centered random variables $X_i \in L^p$, we have the bound
\begin{eqalign}
\Expe{\left|\sum \limits_{i = 1}^n X_i\right|^p} \le C(p) \max\left\{\sum \limits_{i = 1}^n \Expe{ \left|X_i\right|^p}, \left(\sum \limits_{i = 1}^n \Expe{X_i^2}\right)^{p/2}\right\}.    
\end{eqalign}
\end{theorem}
\noindent In the setting of GMC, repeating as in proof of \cref{eq:inverseergodiclimit} with alternating intervals we obtain the bound
\begin{eqalign}
\Expe{\abs{\eta(n)-n}^{p}}\leq c n^{p/2}.    
\end{eqalign}
\noindent Therefore, we bound
\begin{eqalign}
\eqref{eq:deviationsumbound}\leq   1+c\sum_{k=1}^{\infty}\frac{(n+k)^{p/2}}{k^p}+c\sum_{k=1}^{n}\frac{(n-k)^{p/2}}{k^p}\leq c n^{p/2}\sum_{k=1}^{\infty}\frac{1}{k^{p/2}}.
\end{eqalign}
By setting $p=2(1+\e)$ for small $\e>0$ we obtain the bound.
    
\end{proof}


\section{Scaling laws}
In this section we develop the scaling laws of the inverse analogously to that of GMC. Here we will need the scaling for the truncated $U^{\delta}$ field.
\begin{remark}\label{covariancescaling}
The field $U^{\delta}_{\e}(x)$ satisfies for $\lambda\in (0,\infty)$ the scaling
\begin{eqalign}
U^{\delta}_{\e}(x)\eqdis U^{\lambda\delta}_{\lambda\e}(\lambda x)
\end{eqalign}
because the covariance satisfies: $R_{ \varepsilon}^{  \delta }(\abs{ x_{1}- x_{2}})=R_{ \lambda\varepsilon}^{  \lambda\delta }(\abs{\lambda x_{1}-\lambda x_{2}})$. The same is true for the intersected wedges $U(\ell+x)\cap U(\ell)$
\begin{eqalign}
U^{\delta}_{\e}(\ell +x)\cap U^{\delta}_{\e}(\ell)\eqdis  U^{\lambda\delta}_{\lambda\e}(\ell +\lambda x)\cap U^{\lambda\delta}_{\lambda\e}(\ell), \hfill \ell\in \mathbb{R},
\end{eqalign}
because the covariances again satisfy the scaling relation
\begin{eqalign}
\Expe{U^{\delta}_{\e}(\ell+s)\cap U^{\delta}_{\e}(\ell)U^{\delta}_{\e}(\ell+t)\cap U^{\delta}_{\e}(\ell)}=R_{ \varepsilon}^{  \delta }(\min(t,s))=R_{\lambda \varepsilon}^{ \lambda \delta }(\lambda\min(t,s)).
\end{eqalign}
\end{remark}
\begin{proposition}\label{it:scalingincov}  (\textit{Scaling relation}) For the inverse measure $Q^{\delta}$ for the field $U^{\delta}$ and $Q$ for the field $U^{1}$ we have
\begin{eqalign}
Q^{\delta}[x]\eqdis \delta Q\spara{\frac{1}{\delta}x}.
\end{eqalign}
\end{proposition}
\begin{proof}
By Remark \ref{covariancescaling} $\eta^{\delta}[x]\eqdis \lambda\eta^{\delta/\lambda}[\frac{1}{\lambda}x]  $
and so:
\begin{eqalign}
\Proba{Q^{\delta}(0,\lambda x)\geq t}=\Proba{ \lambda x\geq \eta^{\delta}(t)}=\Proba{ \lambda x\geq \lambda\eta^{\frac{\delta}{\lambda}}\para{\frac{1}{\lambda}t}}=\Proba{\lambda Q^{\frac{\delta}{\lambda}}(x)\geq t}.
\end{eqalign}
\end{proof}
 For the inverse corresponding to the exactly-scaled field $\omega$ and the field $U$ we have the following scaling laws.
\begin{proposition}\label{logscalinginver}  (\textit{Lognormal scaling relation}) For all $x\geq 0$ and $t\in [0,\delta]$, we have the relation
\begin{eqalign}\label{eq:inversescalinglaw}
\Proba{\delta\geq \frac{1}{\lambda}Q^{\delta}( x)\geq  t  }=\Proba{\delta\geq Q^{\delta,\lambda}(F_{\lambda} x)\geq  t }
\end{eqalign}
where if the underlying field is $\omega^{\delta, \lambda}$, then we use $F_{\lambda}:=G_{\lambda}$ as defined in \Cref{logfield} and if it is $U^{\delta, \lambda}$, then we use $F_{\lambda}:=c_{\lambda}$ as defined in \Cref{logfieldshifted}. More generally, for any $k\geq 0$ and $t\geq k\delta$ we have
\begin{eqalign}\label{eq:inversescalinglawversion}
\Proba{(k+1)\delta\geq \frac{1}{\lambda}Q^{\delta}( x+\eta^{\delta}(k\lambda \delta))\geq  t  }=\Proba{\delta\geq Q^{\delta,\lambda}(F_{\lambda} x)\geq  t- k\delta }.
\end{eqalign}
\end{proposition}
\begin{proof}[proof of \ref{logscalinginver}]
To  prove \Cref{eq:inversescalinglaw} , we make use of  \propref{exactscaling} (it is applicable since $x\leq\eta_{\omega}^{\delta}(\lambda\delta))$:
\begin{eqalign}
\Proba{\delta\geq \frac{1}{\lambda}Q^{\delta}_{\omega}( x)\geq  t  }&=\Proba{ \eta_{\omega}^{\delta}(\lambda\delta)\geq x\geq \eta^{\delta}_{\omega}(\lambda t) }\\
&\stackrel{Prop.\ref{exactscaling}}{\longequal}\Proba{ \eta_{\omega}^{\delta,\lambda}(\delta)\geq \para{ \lambda e^{\overline{\Omega_{\lambda}}}}^{-1}x\geq \eta^{\delta,\lambda}_{\omega}t   }\\
&=\Proba{\delta\geq Q^{\delta,\lambda}_{\omega}(\para{ \lambda e^{\overline{\Omega_{\lambda}}}}^{-1} x)\geq  t }.
\end{eqalign}
For the truncated field $U^{\delta}$ and measure $\eta^{d}$,  we simply use the \Cref{eq:truncatedscalinglawGMC} to get
\begin{eqalign}
 \eta^{\delta}(\lambda x)\eqdis \lambda e^{\overline{\gamma N(0,\ln\frac{1}{\lambda}-1+\lambda)}}\eta^{\delta,\lambda}(x),
\end{eqalign}
and proceed as above. For the general  $t\geq k\delta $, we have
\begin{eqalign}
&\Proba{(k+1)\delta\geq \frac{1}{\lambda}Q^{\delta}( x+\eta^{\delta}(k\lambda \delta))\geq k\delta+  t  } \\
=&\Proba{\eta^{\delta}\spara{k\lambda\delta,(k+1)\lambda\delta}\geq x\geq  \eta^{\delta}\spara{k\lambda\delta ,k\lambda\delta+\lambda t  }  } \\
=&\Proba{\delta\geq Q^{\delta}(F_{\lambda} x)\geq  t  },
\end{eqalign}
where we also used translation invariance by $k\lambda$.

\end{proof}
\begin{remark}\label{rem:scalinglawerror}
One way to use this estimate for the law of the inverse is to split
\begin{eqalign}
\Proba{\frac{1}{\lambda}Q^{\delta}( x)\geq  t  }&\leq
\Proba{\delta\geq \frac{1}{\lambda}Q^{\delta}( x)\geq  t  }+\Proba{ \frac{1}{\lambda}Q^{\delta}( x)\geq \delta}\\
&=\Proba{\delta\geq Q^{\delta,\lambda}_{\omega}(F_{\lambda} x)\geq  t }+\Proba{ x\geq \eta^{\delta}\para{\lambda\delta}}
\end{eqalign}
and then we estimate the second error term as $x\to 0$. However, for $\lambda=x$ and using the log-normal scaling law we get
\begin{eqalign}
\Proba{ x\geq \eta^{\delta}\para{x\delta}}=\Proba{\sqrt{\beta \ln\frac{1}{x}}+\frac{\ln\frac{1}{\eta^{\delta}(\delta)}}{\sqrt{\beta \ln\frac{1}{x}}}\geq N(0,1)  },
\end{eqalign}
for independent Gaussian $N(0,1)$ and this actually goes to 1 as $x\to 0$. So one needs to use the shifted version in \Cref{eq:inversescalinglawversion}.
\end{remark}

 Using the scaling laws for $\eta$ we also obtain a formula for the density of the inverse in a restricted range.

\begin{proposition}\label{densitylocal} Fix $x>0$. The density $\rho_{Q_{\omega}(x)}(t)$ of the inverse $Q_{\omega}(x)$ has an exact representation in the interval $t\in (0,1)$:
\begin{eqalign}
&\rho_{Q_{\omega}(x)}(t)=\Expe{\int\limits_{\ln\para{ \dfrac{x}{\eta(1)}  }}^{\infty}g(t,y)  \dy},\\
\tfor g(t,y) :=&\frac{1}{\sqrt{2\pi }2\gamma^{3} \para{ \ln(\frac{1}{t})}^{5/2}t }\expo{-\frac{\para{y+\para{1+\frac{\gamma^{2}}{2}}\ln\frac{1}{t}  }^{2}}{2\gamma^{2}\ln(\frac{1}{t})}} \\
&\cdot\para{\para{\ln\frac{1}{t}}^{2}+\gamma^{2}\ln\frac{1}{t} -y^{2} } 
\end{eqalign}
where the integral is over the interval $[\ln\para{ \frac{x}{\eta(1)}  },\infty)$ and the expectation is for the law of $\eta(1)$.
\end{proposition}
\begin{remark}
One can then further use the density results for $\eta(1)$ developed in \citep{remy2020distribution} to obtain a deterministic formula.
\end{remark}
\begin{proof}[Proof of \Cref{densitylocal}]
We start by inverting and using scaling law (\Cref{exactscaling} with $\sigma^{2}_{t}:=\frac{\gamma^{2}}{2}\ln(\frac{1}{t})$):
\begin{eqalign}\label{eq:cdfofinverse}
\Proba{Q_{\omega}(x)\leq t}=&\Proba{x\leq te^{\overline{\Omega}_{t}}\eta(1) }\\=& \Expe{\int_{\ln\para{ \dfrac{x}{\eta(1)}  }}^{\infty}\frac{1}{\sqrt{2\pi \sigma^{2}_{t}}}\expo{-\frac{\para{y+\para{1+\frac{\gamma^{2}}{2}}\ln\frac{1}{t}  }^{2}}{2\sigma^{2}_{t}}} \dy}.
\end{eqalign}
Next we justify swapping the derivative in $t$ and integration. \\
The t-derivative of the above CDF is:
\begin{eqalign}
&\rho_{Q_{\omega}(x)}(t)=\Expe{\int_{\ln\para{ \dfrac{x}{\eta(1)}  }}^{\infty}g(t,y) \dy}.
\end{eqalign}
We split $g(t,y)$ into two factors
\begin{eqalign}
&\expo{-\frac{y\para{1+\frac{\gamma^{2}}{2}}}{\gamma^{2}}}\\
&{\footnotesize\text{$\cdot\para{\frac{1}{\sqrt{2\pi }2\gamma^{3} \para{ \ln(\frac{1}{t})}^{5/2}t }\expo{-\frac{y^{2}+\para{\para{1+\frac{\gamma^{2}}{2}}\ln\frac{1}{t}  }^{2}}{2\gamma^{2}\ln(\frac{1}{t})}}  \para{\para{\ln\frac{1}{t}}^{2}+\gamma^{2}\ln\frac{1}{t} -y^{2} } }.$}}
\end{eqalign}
The second factor is zero at $t=0$ and $t=1$ and so by continuity we have bounded in $[0,1]$ i.e. by extreme value theorem there is some $M>0$ s.t.
\begin{eqalign}
g(t,y)\leq \expo{-\frac{y\para{1+\frac{\gamma^{2}}{2}}}{\gamma^{2}}} \cdot M.
\end{eqalign}
We apply the measure theoretic-version of Leibniz rule \citep{measleib,flanders1973differentiation}.
\end{proof}
\section{Comparison formulas between scales}\label{sec:transitionformuals}
A key tool in \cite[equation (69)]{AJKS} is the relationship between measures from different scales $\eta^{n}$ and $\eta^{m}$ for $m\geq n$:
\begin{eqalign}\label{eq:scalerelationseta}
A_{I} \eta^{m}(I)\leq \eta^{n}(I)\leq  G_{I} \eta^{m}(I).
\end{eqalign}
where $A_{I}:=\inf_{y\in I} e^{\bar{U}^{n}_{m}(y)}$ and $B_{I}:=\sup_{y\in I} e^{\bar{U}^{n}_{m}(y)}$. We will need an analogous relationship as well.
\begin{proposition}\label{it:shiftscalingdiff}\textit{(Shifting and Scaling between scales a.s.})  Fix $x_{0}\in [0,1]$ and let $Q_{\eta_{x_{0}}}:\Rplus\to \Rplus$ be the inverse of the shifted measure $\eta_{x_{0}}(0,x):=\int_{0}^{x}e^{\bar{U}(x_{0}+s)}\ds=\eta(x_{0},x_{0}+x)$. Then we have the following identity for each increment over $(a,b)$:
\begin{eqalign}\label{eq:shiftcenterrelation}
Q_{\eta}(\eta(0,x_{0})+a,\eta(0,x_{0})+b)= Q_{\eta_{x_{0}}}(a,b)=Q_{\eta_{x_{0}}}(b)-Q_{\eta_{x_{0}}}(a).
\end{eqalign}
For $c,y>0$ we have for $Q=Q_{\eta}$,corresponding to field $U$, and $Q^{n}=Q_{\eta^{n}}$,corresponding to field $U^{n}$, the relations:
\begin{eqalign}\label{eq:scalerelinverse}
Q([0,y])=Q^{n}([0,\frac{y}{G_{n}(y)}]) \tand  Q([c,c+y ])=Q^{n}\para{\frac{1}{G_{n}(c)}[c,c+y\frac{G_{n}(c)}{G_{n}(c,c+y)}]},
\end{eqalign}
where the scaling constants are provided from the integral mean value theorem
\begin{eqalign}\label{eq:constantGNIMVT}
G_{n}(c,c+y)&:=\frac{\eta(Q(c),Q(c)+Q([c,c+y ]))}{\eta^{n}(Q(c),Q(c)+Q([c,c+y ]))}\\
&=\frac{y}{\eta^{n}(Q(c),Q(c+y) )}\\
&=\expo{\overline{U_{n}}(\theta_{c,c+y} )  }\tforsome \theta_{c,c+y}\in [Q(c),Q(c+y)]
\end{eqalign}
and we let $G_{n}(y):=G_{n}(0,y)$. We also have a version for the normalized inverse:
\begin{eqalign}\label{eq:scalerelinversenormalized}
Q([c,c+y]\eta(1))=Q^{n}\para{\frac{\eta(1)}{G_{n}(c)}\spara{c,c+y\frac{G_{n}(c)}{G_{n}([c,c+y])}}}.
\end{eqalign}
\end{proposition}
\begin{proof}
\pparagraph{Shift relation (\Cref{eq:shiftcenterrelation})}
That shift relation simply says that $g(x)=Q_{\eta}(\eta(0,x_{0}),\eta(0,x_{0})+x)$ is the inverse of the shifted measure $\eta_{x_{0}}(0,x)$:
\begin{eqalign}
 Q_{\eta}(\eta(0,x_{0}),\eta(0,x_{0})+\eta_{x_{0}}(0,x))= Q_{\eta}(\eta(0,x_{0}+x))-x_{0}=x_{0}+x-x_{0}=x.
\end{eqalign}
\pparagraph{Scaling over general interval $[c,c+y]$ (\Cref{eq:scalerelinverse})}
We have the relation
\begin{eqalign}\label{eq1:inverserelatio}
Q([\eta(a),\eta(a)+\eta(a,a+x) ])=x=Q^{n}([\eta^{n}(a),\eta^{n}(a)+\eta^{n}(a,a+x) ]).
\end{eqalign}
Let $a:=Q(x),x:=Q(c+y)-Q(c)$ and so by integral mean value theorem we have:
\begin{eqalign}
c=\eta(a)=\int_{0}^{a}e^{U_{n}^{1}(x)}\deta^{n}(x)=G_{n}(c)\eta^{n}(a)\tand y=\eta(a,a+x)=G_{n}(c,c+y)\eta^{n}(a,a+x),
\end{eqalign}
which implies the desired result.
\pparagraph{Normalized scaling relation \Cref{eq:scalerelinversenormalized}}
We start by fixing deterministic $c,c+y\in [0,1]$. Then by the bijectivity of $\frac{\eta(x)}{\eta(1)}$, we can find $a,a+x\in [0,1]$ so that:
\begin{eqalign}
c=\frac{\eta(a)}{\eta(1)}\tand y=\frac{\eta(a,a+x)}{\eta(1)}\doncl a=Q(c\eta(1))\tand x=Q((c,c+y)\eta(1)).
\end{eqalign}
By \Cref{eq:constantGNIMVT} the relation is
\begin{eqalign}
Q([c,c+y]\eta(1))=x=Q^{n}(\frac{\eta(1)}{G_{n}(c)}[c,c+y\frac{G_{n}(c)}{G_{n}([c,c+y])}]),
\end{eqalign}
which is the result.
\end{proof}
\subsection{Localization formula}
 We also have a localization formula in the spirit of the localization trick \cite[lemma 3.1]{rhodesvargas2019tail}.
\begin{proposition}\label{it:Girsanovinv}(Localization) For any continuous bounded function $F$ on $\Rplus^{2}$ we have for $x,s\in \Rplus$ and $\e\geq 0$:
\begin{eqalign}\label{eq:continboundfunctional}
\Expe{\int F(Q_{\e}(x),z)e^{\gamma U_{\e}(s)}\ds}=\Expe{\int F(Q_{\e}(x;s),z)\ds},
\end{eqalign}
where $Q_{\e}(x;s)$ is the inverse in a-variable of $\eta_{\e}(a;s)=\int_{0}^{a}e^{\gamma^{2}\Expe{U_{\e}(r)U_{\e}(s)}}e^{\gamma U_{\e}(r)}\dr$ .  Furthermore for $F(y,z):=\ind{s\leq y}$ we have:
\begin{eqalign}\label{eq:indicatorfunctional}
\int\Expe{\ind{s\leq Q_{\e}(x)}e^{\gamma U_{\e}(s)}}\ds=\int\Proba{s\leq Q_{\e}(x;0)}\ds.
\end{eqalign}
\end{proposition}
\begin{proof}[proof of \ref{it:Girsanovinv}: Localization]
 \cite[lemma 2.1]{aru2020gaussian} states that for any continuous bounded functional $F(X,z)$ of the entire Gaussian field $\set{X(r)}_{r\in \R}:=\set{U_{\e}(r)}_{r\in \R}$ and $z\in \Rplus$
\begin{eqalign}
\Expe{\int F(X,z)e^{\gamma U_{\e}(s)}\ds}=\Expe{\int F(X+\gamma \Expe{X(\cdot)U_{\e}(s)},z)\ds}.
\end{eqalign}
This implies \Cref{eq:continboundfunctional} provided that the inverse $Q_{X}(a)$ of $\eta_{X}(b):=\int_{0}^{b} e^{X(r)}\dr$ is also a continuous function of $X$ in the sup norm $\supl{r\geq 0}\abs{X(r)-Y(r)}$.\\
Using the layercake representation for two Gaussian fields $X,Y$, the $Q_{X}(a)-Q_{Y}(a)$ equals
\begin{eqalign}
&\abs{\int_{0}^{\infty}\ind{\eta_{X}(r)\leq a} -\ind{\eta_{Y}(r)\leq a}    \dr} \\
\leq&\int_{0}^{\infty}\ind{\eta_{X}(r)\leq a\leq \eta_{Y}(r)} +\ind{\eta_{Y}(r)\leq a\leq \eta_{X}(r)}   \dr.
\end{eqalign}
By imposing $\sup_{r\geq 0}\abs{X(r)-Y(r)}\leq \delta$ for small $\delta>0$, we upper bound the first indicator as follows
\begin{eqalign}
 \ind{\eta_{X}(r)\leq a\leq \eta_{Y}(r)}\leq \ind{\eta_{X}(r)\leq a\leq \eta_{X}(r)e^{\gamma\delta}}=\ind{ae^{-\gamma\delta}\leq\eta_{X}(r)\leq a}.
\end{eqalign}
Similarly, the second indicator is bounded by $\ind{a\leq\eta_{X}(r)\leq ae^{\gamma\delta}}$.
Returning to the inverse this becomes the upper bound
\begin{eqalign}
\abs{Q_{X}(a)-Q_{Y}(a)}\leq & Q_{X}(a)-Q_{X}( ae^{-\gamma\delta})   +Q_{X}(ae^{\gamma\delta} )-Q_{X}(a)   \\
=&Q_{X}(ae^{\gamma\delta} )-Q_{X}(ae^{-\gamma\delta} ).
\end{eqalign}
So we got the claimed continuity of the inverse by the regular continuity of the inverse.
\pparagraph{Proof of \Cref{eq:indicatorfunctional}} Here we fix $\e>0$. By inverting to $\eta$ and applying tilting lemma we have:
\begin{eqalign}
\Expe{\ind{s\leq Q_{\e}(x)}e^{\gamma U_{\e}(s)}}=\Expe{\ind{\eta_{\e}(s)\leq x}e^{\gamma U_{\e}(s)}}=\Proba{\eta_{\e}(s;s)\leq x},
\end{eqalign}
where $\eta_{\e}(a;b)=\int_{0}^{a}e^{\gamma^{2}\Expe{U_{\e}(r)U_{\e}(b)}}e^{\gamma U_{\e}(r)}\dr$. So we have by translation invariance and change of variables:
\begin{eqalign}
&\int_{0}^{s}\expo{\gamma^{2}\Expe{U_{\e}(0)U_{\e}(s-r)}+ \gamma U_{\e}(r)}\dr\\
=& \int_{0}^{s}\expo{\gamma^{2}\Expe{U_{\e}(0)U(r)}+ \gamma U_{\e}(s-r)}\dr.
\end{eqalign}
Then using that the covariance only depends on the difference $\Expe{U_{\e}(s-a)U_{\e}(s-b)}=\Expe{U_{\e}(a)U_{\e}(b)}$, we obtain
\begin{eqalign}
 \eta_{\e}(s;s)\eqdis \int_{0}^{s}\expo{\gamma^{2}\Expe{U_{\e}(0)U_{\e}(r)}+ \gamma U_{\e}(r)}\dr=\eta_{\e}(s;0).
\end{eqalign}
We finally take limit $\e\to 0$ at each side.
\end{proof}


\part{General Moments for the inverse}\label{part:momentinverse}
 In this section we study the moments for the upper truncated inverse $Q^{\delta}=Q^{\delta}_{U}$ where the truncation can be for arbitrarily large $\delta\in (0,\infty)$.

\section{Inverse moments }\label{inversemoments}
 We start with computing the moments for the inverse $Q^{\delta}(x)$ because here we can simply use the relation $\set{Q^{\delta}(x)\geq t}=\set{x\geq \eta^{\delta}(t)}$ and thus turn the question of positive/negative moments of $Q$ into the negative/positive moments of $\eta$.
\begin{proposition}
\label{INVERSE_MOMENTS}

The inverse has finite moments $\Exp[(Q^{\delta}(0,x))^{p}]<\infty $ for the interval
\begin{eqalign}
p\in \para{-\frac{(1+\frac{\gamma^{2}}{2})^{2}}{2\gamma^{2}}, \infty}.  \end{eqalign}
\begin{itemize}

\item Positive moments: For general non-negative moments $p\geq 0$ and constants $\delta,x\geq 0$, we have the upper bound for $Q^{\delta}$:
\begin{eqalign}\label{eq:positivemomentsinverseQx}
\Exp[(Q^{\delta}(0,x))^{p}]\lesssim\branchmat{\delta^{p-q} x^{q}+\delta^{p-\wt{q}}x^{\wt{q}}&\delta\geq 1\\ x^{q}+x^{\wt{q}} &\delta \leq 1},
\end{eqalign}
where  $p,q,\wt{q}$ satisfy $q+\frac{\gamma^{2}}{2}(q^{2}+q)<p<\wt{q}$.
\item Negative moments: For $p\geq 0$ and $\delta,x\geq 0$, we have the upper bound for $Q^{\delta}$:
\begin{eqalign}\label{eq:negativemomentsinverseQx}
\Expe{(Q^{\delta}(0,x))^{-p}}\lesssim \branchmat{\delta^{q-p} x^{-q}+\delta^{\wt{q}-p}x^{-\wt{q}}&\delta\geq 1\\ x^{-q}+x^{-\wt{q}}&\delta\leq 1 }.
\end{eqalign}
When $\delta\geq 1$ $p,q,\wt{q}$ satisfy the following: when $q,\wt{q}\in (0,1)$, we have $\zeta(q):=q-\frac{\gamma^{2}}{2}(q^{2}-q)>p>\zeta(\wt{q})$ and when $q,\wt{q}\in (1, \frac{2}{\gamma^{2}})$, we have $\zeta(q)>p>\wt{q}$. \\
When $\delta\leq 1$
and $q,\wt{q}\in (1,\frac{2}{\gamma^{2}})$, we have the same constraint $\zeta(q)>p>\wt{q}$ but when  $q,\wt{q}\in (0,1)$, we have $\zeta(q)>p>\zeta(\wt{q})$.
\end{itemize}
\end{proposition}
\begin{remark}
The left endpoint in $\para{-\frac{(1+\frac{\gamma^{2}}{2})^{2}}{2\gamma^{2}}, \infty}$ is given by the supremum of the moment exponent of the $\eta$ measure:
\begin{eqalign}
\frac{(1+\frac{\gamma^{2}}{2})^{2}}{2\gamma^{2}}=\supl{q\in (0,\frac{2}{\gamma^{2}})}\zeta(q),
\end{eqalign}
where $\zeta(q):=q-\frac{\gamma^{2}}{2}(q^{2}-q)$ is the multifractal exponent.
\end{remark}

\begin{remark}\label{rem:preciseexponents}One can obtain precise exponents by optimizing. In the positive moments \cref{eq:positivemomentsinverseQx}, we optimize
\begin{eqalign}
q+\frac{\gamma^{2}}{2}(q^{2}+q)<p\Rightarrow     q=\frac{1}{2}\para{\sqrt{1+\frac{4p+2}{\beta}+\frac{1}{\beta^{2}}  }-1-\frac{1}{\beta}   }-\epsilon,
\end{eqalign}
for small $\e>0$. In the negative moments \cref{eq:negativemomentsinverseQx}, we optimize
\begin{eqalign}
q-\frac{\gamma^{2}}{2}(q^{2}-q)>p\Rightarrow     q=\frac{1}{2}\para{\sqrt{1+\frac{-4p+2}{\beta}+\frac{1}{\beta^{2}}  }-1-\frac{1}{\beta}   }+\epsilon.
\end{eqalign}
\end{remark}
 For computing the moments of the inverse we will use the estimates for $\eta$ in the appendix \Cref{momentseta} but for $g=0$.
\begin{proof}[proof of \Cref{INVERSE_MOMENTS}]
\proofparagraph{The positive moment case $p\geq 0$: bounds for \sectm{$\Exp[(Q^{\delta}[0,x])^{p}]$} for $\delta\in [0,\infty)$}\\
We apply the scaling law for the inverse, layer-cake principle, invert and finally bound by Markov inequality for $q,\wt{q}>0$ to find
\begin{eqalign}
&\Exp[(Q[0,x/\delta])^{p}]\\
&=\int_{0}^{1}\Proba{(Q[0,x/\delta])^{p}\geq t}\dt +\int_{1}^{\infty}\Proba{(Q[0,x/\delta])^{p}\geq t}\dt\\
&=\int_{0}^{1}\Proba{\frac{1}{\eta[0,t^{1/p}]}\geq \frac{\delta}{x} }\dt +\int_{1}^{\infty}\Proba{\frac{1}{\eta[0,t^{1/p}]}\geq \frac{\delta}{x} }\dt\\
&\stackrel{Markov}{\leq} x^{q}\delta^{-q}\int_{0}^{1}\Expe{\para{\frac{1}{\eta[0,  t]}}^{q}}pt^{p-1}\dt +x^{\wt{q}}\delta^{-\wt{q}}\int_{1}^{\infty}\Expe{\para{\frac{1}{\eta[0,  t]}}^{\wt{q}}}pt^{p-1}\dt.
\end{eqalign}
Using the moment estimates for $\eta$ we have for $t\in [0,1]$
\begin{eqalign}
\int_{0}^{1}\Expe{\para{\frac{1}{\eta[0,  t]}}^{q}}pt^{p-1}\dt\leq c\int_{0}^{1} t^{-q}t^{-\frac{\gamma^{2}}{2}(q^{2}+q)}pt^{p-1}  \dt,
\end{eqalign}
this integral is finite only for $1-p+q+\frac{\gamma^{2}}{2}(q^{2}+q)<1\Leftrightarrow q+\frac{\gamma^{2}}{2}(q^{2}+q)<p$. For the integral over $[1,\infty]$ we similarly have
\begin{eqalign}
\int_{1}^{\infty}\Expe{\para{\frac{1}{\eta[0,   t]}}^{\wt{q}}}pt^{p-1}  \dt\leq c\int_{1}^{\infty}t^{-\wt{q}}pt^{p-1}  \dt,
\end{eqalign}
which is finite only for $1-p+\wt{q}>1\Leftrightarrow \wt{q}>p$.

\proofparagraph{The positive moment case $p\geq 0$: bounds for \sectm{$\Exp[(Q^{\delta}[0,x])^{p}]$} for \sectm{$\delta\leq 1$}}
Here we split based on $1$ again but keeping $\delta$:
\begin{eqalign}
\Exp[(Q^{\delta}[0,x])^{p}]&=\int_{0}^{\infty}\Proba{(Q^{\delta}[0,x])^{p}\geq t}\dt\\
&\leq x^{q}\int_{0}^{1}\Expe{\para{\frac{1}{\eta^{\delta}[0,  t]}}^{q}}pt^{p-1}\dt +x^{\wt{q}}\int_{1}^{\infty}\Expe{\para{\frac{1}{\eta^{\delta}[0,  t]}}^{\wt{q}}}pt^{p-1}\dt.
\end{eqalign}
We avoided splitting based on $\delta$ because we want to allow $\delta\to 0^{+}$. By the moments for $\eta^{\delta}(t)$ with $t\leq 1$, we have
\begin{eqalign}
\int_{0}^{1}\Expe{\para{\frac{1}{\eta^{\delta}[0,  t]}}^{q}}pt^{p-1}\dt\leq c\int_{0}^{1} t^{-q}t^{-\frac{\gamma^{2}}{2}(q^{2}+q)}pt^{p-1}  \dt,
\end{eqalign}
this integral is finite only for $1-p+q+\frac{\gamma^{2}}{2}(q^{2}+q)<1\Leftrightarrow q+\frac{\gamma^{2}}{2}(q^{2}+q)<p$. For the integral over $[1,\infty]$ we bound by
\begin{eqalign}
\int_{1}^{\infty}\Expe{\para{\frac{1}{\eta^{\delta}[0,   t]}}^{\wt{q}}}pt^{p-1}  \dt\leq c\int_{1}^{\infty}t^{-\wt{q}}pt^{p-1}  \dt,
\end{eqalign}
which is finite only for $1-p+\wt{q}>1\Leftrightarrow \wt{q}>p$.


\proofparagraph{The negative moment case $-p\leq 0$: bounds for \sectm{$\Exp[(Q^{\delta}[0,x])^{-p}]$} for \sectm{$\delta\in [0,\infty)$}}
We again apply the scaling law of the inverse $Q^{\delta}(x)\eqdis \delta Q(x/\delta)$, layer-cake principle, invert and bound by Markov inequality to find
\begin{eqalign}
&\Exp[(Q[0,x/\delta])^{-p}]\\
&=\int_{0}^{1}\Proba{(Q[0,x/\delta])^{-p}\geq t}\dt +\int_{1}^{\infty}\Proba{(Q[0,x/\delta])^{-p}\geq t}\dt\\
&=\int_{0}^{1}\Proba{\eta[0,t^{-1/p}]\geq \frac{x}{\delta} }\dt +\int_{1}^{\infty}\Proba{\eta[0,t^{-1/p}]\geq \frac{x}{\delta} }\dt\\
&\leq x^{-\wt{q}}\delta^{\wt{q}}\int_{1}^{\infty}\Expe{\para{\eta[0,  t]}^{\wt{q}}}pt^{-p-1}\dt +x^{-q}\delta^{q}\int_{0}^{1}\Expe{\para{\eta[0,  t]}^{q}}pt^{-p-1}\dt
\end{eqalign}
for some $q,\wt{q}\in (0,\frac{2}{\gamma^{2}})$ that we will next constrain. The integral over $[0,1]$ is bounded by
\begin{eqalign}
\int_{0}^{1}\Expe{\para{\eta[0,  t]}^{q}}pt^{-p-1}\dt\leq c\int_{0}^{1}  t^{\zeta(q)}pt^{-p-1}  \dt,
\end{eqalign}
this integral is finite only for $1+p-\para{q-\frac{\gamma^{2}}{2}(q^{2}-q)}<1\Leftrightarrow q-\frac{\gamma^{2}}{2}(q^{2}-q)>p$. The integral over $[1,\infty]$ is bounded by
\begin{eqalign}
\int_{1}^{\infty}\Expe{\para{\eta[0,  t]}^{q}}pt^{-p-1}\dt\leq c\int_{1}^{\infty} t^{\wt{q}}pt^{-p-1}  \dt,
\end{eqalign}
which is finite only for $1+p-\wt{q}>1\Leftrightarrow p>\wt{q}$. \\
 In the case $q,\wt{q}\in (0,1)$ we can do better for the integral over $[1,\infty]$
\begin{eqalign}
\int_{1}^{\infty}\Expe{\para{\eta[0,  t]}^{\wt{q}}}pt^{-p-1}\dt\leq c\int_{1}^{\infty} t^{\zeta(\wt{q}) }pt^{-p-1}  \dt,
\end{eqalign}
which is finite only for $1+p-\zeta(\wt{q})>1\Leftrightarrow p>\zeta(\wt{q})$.
\proofparagraph{The negative moment case $-p\leq 0$: bounds for \sectm{$\Exp[(Q^{\delta}[0,x])^{-p}]$} for \sectm{$\delta\leq 1$}}
We again apply the layer-cake principle, invert and bound by Markov inequality to find
\begin{eqalign}
\Exp[(Q^{\delta}[0,x])^{-p}]&=\int_{0}^{1}\Proba{(Q^{\delta}[0,x])^{-p}\geq t}\dt +\int_{1}^{\infty}\Proba{(Q^{\delta}[0,x])^{-p}\geq t}\dt\\
&\leq x^{-\wt{q}}\int_{1}^{\infty}\Expe{\para{\eta^{\delta}[0,  t]}^{\wt{q}}}pt^{-p-1}\dt+ x^{-q}\int_{0}^{1}\Expe{\para{\eta^{\delta}[0,  t]}^{q}}pt^{-p-1}\dt.
\end{eqalign}
For the integral term over $[0,1]$ we upper bound by
\begin{eqalign}
\int_{0}^{1}\Expe{\para{\eta^{\delta}[0,  t]}^{q}}pt^{-p-1}\dt\leq c\int_{0}^{1}  t^{q}t^{-\frac{\gamma^{2}}{2}(q^{2}-q)}pt^{-p-1}  \dt,
\end{eqalign}
this integral is finite only for $1+p-\para{q-\frac{\gamma^{2}}{2}(q^{2}-q)}<1\Leftrightarrow q-\frac{\gamma^{2}}{2}(q^{2}-q)>p$. For the integral over $[1,\infty)$ and case $q,\wt{q}\in [1,\infty)$, we upper bound by
\begin{eqalign}
\int_{1}^{\infty}\Expe{\para{\eta^{\delta}[0,  t]}^{q}}pt^{-p-1}\dt\leq c\int_{1}^{\infty} t^{\wt{q}}pt^{-p-1}  \dt,
\end{eqalign}
which is finite only for $1+p-\wt{q}>1\Leftrightarrow p>\wt{q}$. For the case $q,\wt{q}\in (0,1)$ we upper bound by
\begin{eqalign}
\int_{1}^{\infty}\Expe{\para{\eta^{\delta}[0,  t]}^{\wt{q}}}pt^{-p-1}\dt\leq c\int_{1}^{\infty} t^{\wt{q}}t^{-\frac{\gamma^{2}}{2}(\wt{q}^{2}-\wt{q})}pt^{-p-1}  \dt,
\end{eqalign}
which is finite for $1+p-\para{\wt{q}-\frac{\gamma^{2}}{2}(\wt{q}^{2}-\wt{q})}>1\Leftrightarrow p>\wt{q}-\frac{\gamma^{2}}{2}(\wt{q}^{2}-\wt{q})$.
\end{proof}

\section{Randomly shifted GMC }
In this section we study the tail/small-ball and moments for the shifted GMC $\etamu{Q^{\delta}_{a},Q^{\delta}_{a}+t}{\delta}$ for all $t>0$. We will need these in studying the moments of the inverse increments due to the relation
\begin{eqalign}
\set{Q(a,a+x)>t}=\set{x>\eta(Q_{a},Q_{a}+t)}.
\end{eqalign}
For large $t\geq 2\delta$ we can just use \Cref{deltaSMP} but for small $0\leq t\leq 2\delta$ the correlation increases. This is because we see that by splitting around $x=Q^{\delta}(a)$
\begin{eqalign}
U(Q^{\delta}(a)+t)=U(Q^{\delta}(a)+t)\cap U(Q^{\delta}(a))+U(Q^{\delta}(a)+t)\setminus U(Q^{\delta}(a)),
\end{eqalign}
we get $U(Q^{\delta}(a)+t)\setminus U(Q^{\delta}(a))\neq \varnothing$ and so $\etamu{Q^{\delta}_{a},Q^{\delta}_{a}+t}{\delta_{0}}$ is correlated to $Q^{\delta}(a)$. The strategy here is to understand the "density" of the inverse (i.e. local behaviour $\set{a_{k}\leq Q^{\delta}_{a}\leq a_{k+1}}$ for some $a_{k}$) by translating to the local behaviour of GMC. We also introduce a parameter $\rho$ for computational purposes.
\begin{proposition}\label{prop:shiftedGMCmoments}
Let $t,\delta_{0},\delta,a>0$ and $\rho\in (0,1]$ with $\delta_{0}\geq \delta$.
\begin{itemize}

    \item  For strictly increasing nonnegative function $F_{1}$ we have
 \begin{eqalign}
&\Expe{F_{1}\para{\etamu{Q^{\delta}_{a},Q^{\delta}_{a}+t}{\delta_{0}}}} \\&\leq\Expe{F_{1}\para{\supl{0\leq T\leq \rho \delta}\etamu{T,T+t}{\delta_{0}}} }  C_{sup}(a,\delta,\rho),
\end{eqalign}
where 
\begin{eqalign}
 C_{sup}(a,\delta,\rho):= c\cdot  \para{1+c_{1}\frac{(a/\delta)^{q_{1}}}{\para{\rho}^{-\zeta(-q_{1})}}+c_{2}\frac{(a/\delta)^{q_2}}{\para{\rho}^{q_{2}}}.   \rho^{(bq_{2}-1)/b} },
\end{eqalign}
for $b\in (0,1)$, $-b\zeta(-q_{1})>1$ and $bq_{2}>1$ and $c>0$ is universal. In the context of \cref{prop:decouplingsplit}, we also have the bound
 \begin{eqalign}
&\Expe{F_{1}\para{\etamu{Q^{\delta}_{a},Q^{\delta}_{a}+t}{\delta_{0}}}} \\\leq&\Expe{F_{1}\para{\supl{0\leq T\leq \rho \delta}\para{\int_{0}^{t}e^{\overline{U^{\delta_{0}}(T+s)\cap U^{\delta_{0}}(T)}}\deta_{\tilde{U}}^{\delta_{0}}(s)}} }  C_{sup}(a,\delta,\rho).
\end{eqalign}
 \item For strictly decreasing positive function $F_{2}$ we have
\begin{eqalign}\label{shiftednegativeappinf}
&\Expe{F_{2}\para{\etamu{Q^{\delta}_{a},Q^{\delta}_{a}+t}{\delta_{0}}}}\\
\leq &\para{\Expe{\para{F_{2}\para{\infl{0\leq T\leq \rho\delta}\etamu{T,T+t}{\delta_{0}}}}^{1+\e}}}^{1+\e}C_{inf}(\alpha,\delta,\delta_{0},\rho),
\end{eqalign}
where
\begin{eqalign}
 &C_{inf,1}(\alpha,\delta,\delta_{0},\rho):= c\cdot \para{1+\wt{c}_{1}\frac{(a/\delta)^{\wt{q}_2}}{(\rho\frac{\delta_{0}}{\delta})^{\wt{q}_{2}}} \para{\rho}^{\frac{(c\wt{q}_{2}-1)}{c}}},\\
 &C_{inf,2}(\alpha,\delta,\delta_{0},\rho):=c\cdot\frac{(a/\delta)^{\wt{q}_{1}/\wt{q}_{12}}}{(\rho\frac{\delta_{0}}{\delta})^{-\zeta(-\wt{q}_{1})/\wt{q}_{12}}},\\
 &C_{inf}(\alpha,\delta,\delta_{0},\rho):=\max_{i=1,2}\para{C_{inf,i}(\alpha,\delta,\delta_{0},\rho)},
\end{eqalign}
for $c\in (0,1)$, $-c\zeta(-\wt{q}_{1})/\wt{q}_{12}>1$, $c\wt{q}_{2}>1$ and
$\dfrac{1}{\wt{q}_{11}}+\dfrac{1}{\wt{q}_{12}}=1$ and we took $\wt{q}_{11}=1+\e$ for arbitrarily small $\e>0$, and $c>0$ is universal.

\end{itemize}
\end{proposition}
\begin{remark}
We see that if we take $a\leq \delta$ the constants are uniformly bounded. This is reasonable since as $\delta\to 0$, we simply get Lebesgue measure.
\end{remark}
 In the \Cref{maxminmodGMC}, we study the above max/min quantities for GMC in the specific cases of moments and tail/small-ball. So we have the following corollary of \Cref{prop:shiftedGMCmoments} in the case of moments.
\begin{corollary}\label{cor:shiftedGMCmoments}
Let $t>0$.
\begin{itemize}
    \item  For positive moments $p\in [1,\beta^{-1})$ we have
 \begin{eqalign}
\Expe{\para{\etamu{Q^{\delta}_{a},Q^{\delta}_{a}+t}{\delta_{0}}}^{p}} &\leq
t^{\alpha(p)}C_{sup}(a,\delta,\rho),
\end{eqalign}
where $\alpha(p):=\zeta(p)-1$ when $t\leq \delta_{0}$ and $\alpha(p):=p$ when $t\geq \delta_{0}$.

\item  For positive moments $p\in (0,1)$  we have
 \begin{eqalign}\label{eq:momentsp01shifted}
&\Expe{\para{\etamu{Q^{\delta}_{a},Q^{\delta}_{a}+t}{\delta_{0}}}^{p}} \\
&\leq
C_{sup}(a,\delta,\rho)\branchmat{t^{p(1-12\sqrt{2}\gamma)}\para{\ln\frac{1}{t}}^{p/2} & t\leq \frac{\rho \delta}{e}\\ \para{\ln\frac{1}{t}}^{p/2} & t\geq \frac{\rho \delta}{e}}.
\end{eqalign}
For simplicity we further bound by $t^{p\alpha}$ for $\alpha=1-12\sqrt{2}\gamma-\epsilon$ for small $\e>0$ to avoid carrying the log-term around.

 \item  For negative moments we have
\begin{eqalign}\label{shiftednegativeappinfmom}
\Expe{\para{\etamu{Q^{\delta}_{a},Q^{\delta}_{a}+t}{\delta_{0}}}^{-p} } \leq &t^{\alpha_{\e}(p)}C_{inf}(\alpha,\delta,\delta_{0},\rho),
\end{eqalign}
where $\alpha_{\e}(p):=\frac{\zeta(p(1+\e))-1}{1+\e}$ for $t\leq 2\delta_{0}$ and $\alpha_{\e}(p):=p$ for $t\geq 2\delta_{0}$.
\end{itemize}
\end{corollary}

\begin{remark}
It is unclear as in \Cref{differencetermunshifted} whether there are sharper techniques that could eliminate the $a-$dependence on the bounds (i.e. translation invariant bounds).
\end{remark}
\subsection{Proof of \sectm{\Cref{prop:shiftedGMCmoments}}}
\begin{proof}
We start with decomposing the range of $Q^{\delta}_{a}$ using a diverging strictly increasing sequence $\set{a_{k}}_{k\geq 0}$, i.e. $a_{k}:=\rho\delta k^{b},k\geq 0$ for some $b\in (0,1)$
\begin{eqalign}
&\Expe{F\para{\etamu{Q^{\delta}_{a},Q^{\delta}_{a}+t}{\delta_{0}}}}\\
=&\sum_{k\geq 0}\Expe{F\para{\etamu{Q^{\delta}_{a},Q^{\delta}_{a}+t}{\delta_{0}}} \ind{a_{k}\leq Q^{\delta}_{a}\leq a_{k+1}}},
\end{eqalign}
where $F$ is either an increasing or decreasing function of the measure GMC.
\proofparagraph{Case of an increasing function F}
For the increasing case, we bound by the supremum.
\begin{multline}
\sum_{k\geq 0}\Expe{F\para{\etamu{Q^{\delta}_{a},Q^{\delta}_{a}+t}{\delta_{0}}} \ind{a_{k}\leq Q^{\delta}_{a}\leq a_{k+1}}}\\
\leq \sum_{k\geq 0}\Expe{\supl{T\in [a_{k},a_{k+1}]}F\para{\etamu{T,T+t}{\delta_{0}}} \ind{a_{k}\leq Q^{\delta}_{a}\leq a_{k+1}}}.
\end{multline}
At this step, we can instead use the decoupling described in \cref{prop:decouplingsplit}. Meaning that we can bound by
\begin{eqalign}
&\Expe{F\para{\etamu{Q^{\delta}_{a},Q^{\delta}_{a}+t}{\delta_{0}},Q^{\delta}_{a}+t}\ind{a_{k}\leq Q^{\delta}_{a}\leq a_{k+1} }}\\
\leq &\Expe{\supl{T\in [a_{k},a_{k+1}]}F\para{\int_{0}^{t}e^{\overline{U(T+s)\cap U(T)}}\deta_{\tilde{U}}(s)} \ind{a_{k}\leq Q^{\delta}_{a}\leq a_{k+1} }}.
\end{eqalign}
When $a_{k}=0$, we bound by $1$. When $a_{k}\geq \delta$, we keep one of the upper bounds and apply FKG inequality \Cref{FKGineq}
\begin{eqalign}
&\Expe{\supl{T\in [a_{k},a_{k+1}]}F\para{\etamu{T,T+t}{\delta_{0}}}  \ind{\eta^{\delta}(a_{k})\leq a}}\\
&\leq \Expe{\supl{T\in [a_{k},a_{k+1}]}F\para{\etamu{T,T+t}{\delta_{0}}}  }\Proba{\eta^{\delta}(a_{k})\leq a}.
\end{eqalign}
When $0< a_{k}\leq \delta$, we bound by using the scaling law $\eta^{\delta}(x)\eqdis \delta\eta^{1}(x/\delta)$ and Markov for $q_{1}>0$
\begin{eqalign}
\Proba{\eta^{\delta}(a_{k})\leq a}\leq \frac{(a/\delta)^{q_{1}}}{(a_{k}/\delta)^{-\zeta(-q_{1})}}
\end{eqalign}
and when $a_{k}\geq \delta$, we bound by a different exponent
\begin{eqalign}\label{eq:summterm}
\Proba{\eta^{\delta}(a_{k})\leq a}\leq \frac{(a/\delta)^{q_2}}{(a_{k}/\delta)^{q_{2}}},
\end{eqalign}
for some $q_{1},q_{2}>0$.
For $k\geq 0$ by subadditivity of $x^{b}$ we have
\begin{eqalign}
a_{k+1}-a_{k}\leq \rho\delta.
\end{eqalign}
For $k=0$, we bound by one 1. When $\rho\delta\leq a_{k}\leq \delta$, we get the sum
\begin{eqalign}
 \frac{(a/\delta)^{q_{1}}}{\para{\rho}^{-\zeta(-q_{1})}}\sum_{\rho\delta\leq a_{k}\leq \delta}\frac{1}{k^{-b\zeta(-q_{1})}}\leq \frac{(a/\delta)^{q_{1}}}{\para{\rho}^{-\zeta(-q_{1})}} ,
\end{eqalign}
where in order to have summability we require $-b\zeta(-q_{1})>1$. When $a_{k}\geq \delta$, we bound by using Markov for $q_{2}>0$
\begin{eqalign}
\sum_{a_{k}\geq \delta}\Proba{\eta^{\delta}(a_{k})\leq a}\leq\sum_{a_{k}\geq \delta} \frac{(a/\delta)^{q_2}}{(a_{k}/\delta)^{q_{2}}}=\frac{(a/\delta)^{q_2}}{\para{\rho}^{q_{2}}}\sum_{a_{k}\geq \delta}\frac{1}{k^{bq_{2}}}\leq \frac{(a/\delta)^{q_2}}{\para{\rho}^{q_{2}}} \rho^{-(1-bq_{2})/b},
\end{eqalign}
where in order to have summability we require $bq_{2}>1$.
\proofparagraph{Case of a decreasing function F}
Here we upper bound the summand by
\begin{multline}
\sum_{k\geq 0}\Expe{F\para{\etamu{Q^{\delta}_{a},Q^{\delta}_{a}+t}{\delta_{0}}} \ind{a_{k}\leq Q^{\delta}_{a}\leq a_{k+1}}}\\
\leq \sum_{k\geq 0}\Expe{F\para{\infl{T\in [a_{k},a_{k+1}]}\etamu{T,T+t}{\delta_{0}}} \ind{a_{k}\leq Q^{\delta}_{a}\leq a_{k+1}}}.
\end{multline}
For the case $a_{k}\geq 2\delta$, we simply remove a delta part $\eta^{\delta}(a_{k}-\delta )$ to decouple them:
\begin{eqalign}
&\Expe{F\para{\infl{T\in [a_{k},a_{k+1}]}\etamu{T,T+t}{\delta_{0}}} \ind{\eta^{\delta}(a_{k})\leq a}} \\ \leq& \Expe{F\para{\infl{T\in [a_{k},a_{k+1}]}\etamu{T,T+t}{\delta_{0}}} \ind{\eta^{\delta}(a_{k}-\delta)\leq a}}\\
=&\Expe{F\para{\infl{T\in [a_{k},a_{k+1}]}\etamu{T,T+t}{\delta_{0}}} }\Proba{\eta^{\delta}(a_{k}-\delta)\leq a}.
\end{eqalign}
 For small $a_{k}\in (0, 2\delta]$,
we just apply \Holder to avoid having $a$ in the denominator
\begin{eqalign}\label{negativeunshiftedmom}
\para{\Expe{\para{F\para{\infl{T\in [a_{k},a_{k+1}]}\etamu{T,T+t}{\delta_{0}}}}^{\wt{q}_{11} }}}^{1/\wt{q}_{11}} \para{\Proba{\eta^{\delta}(a_{k})\leq a} }^{1/\wt{q}_{12}}
\end{eqalign}
for $\dfrac{1}{\wt{q}_{11}}+\dfrac{1}{\wt{q}_{12}}=1$. By translation invariance since $a_{k+1}-a_{k}\leq \rho\delta$ we have
\begin{eqalign}
\eqref{negativeunshiftedmom}\leq \para{\Expe{\para{F\para{\infl{T\in [0,\rho\delta]}\etamu{T,T+t}{\delta_{0}}}}^{\wt{q}_{11} }}}^{1/\wt{q}_{11}} \para{\Proba{\eta^{\delta}(a_{k})\leq a} }^{1/\wt{q}_{12}}
\end{eqalign}
As above for all $k\geq 0$, we have $a_{k+1}-a_{k}\leq \rho\delta$. When $k=0$, we  bound by one. When $a_{k}\geq 2\delta$, we bound by Markov inequality for $\wt{q}_2>0$
\begin{eqalign}
\sum_{a_{k}\geq 2\delta}\Proba{\eta^{\delta}(a_{k}-\delta)\leq a}\leq\frac{(a/\delta)^{\wt{q}_2}}{(\rho\frac{\delta}{\delta})^{\wt{q}_{2}}}\sum_{a_{k}\geq 2\delta} \frac{(a/\delta)^{\wt{q}_2}}{( k^{b}-\rho^{-1})^{\wt{q}_{2}}}\leq\frac{(a/\delta)^{\wt{q}_2}}{(\rho)^{\wt{q}_{2}}} \para{\rho}^{\frac{-(1-b\wt{q}_{2})}{b}},
\end{eqalign}
where in order to have summability we require $c\wt{q}_{2}>1$. In the case of $0< a_{k}\leq 2\delta\Leftrightarrow 1\leq k\leq (\frac{2}{\rho})^{1/b}$, we use the \Holder approach and similarly get the sum
\begin{eqalign}
&\sum_{0<a_{k}\leq2 \delta}\para{\Proba{\eta^{\delta}(a_{k})\leq a} }^{1/\wt{q}_{12}}\\\leq &\frac{(a/\delta)^{\wt{q}_{1}/\wt{q}_{12}}}{(\rho)^{-\zeta(-\wt{q}_{1})/\wt{q}_{12}}}\cdot \para{ \sum_{0<a_{k}\leq \delta }\frac{1}{k^{-b\zeta(-\wt{q}_{1})/\wt{q}_{12}}}+\frac{(\rho)^{-\zeta(-\wt{q}_{1})/\wt{q}_{12}}}{(\rho)^{\wt{q}_{1}/\wt{q}_{12}}} \sum_{\delta\leq a_{k}\leq 2 \delta }\frac{1}{k^{b\wt{q}_{1}/\wt{q}_{12}}}}\\
\leq  &\frac{(a/\delta)^{\wt{q}_{1}/\wt{q}_{12}}}{(\rho)^{-\zeta(-\wt{q}_{1})/\wt{q}_{12}}},
\end{eqalign}
where there is no issue for summability but if $\rho\to 0$,  we would require $-b\zeta(-\wt{q}_{1})/\wt{q}_{12}>1$. \\
Finally, for arbitrarily small $\e>0$, we can take $\wt{q}_{11}=1+\e$. This is possible because we can find $\wt{q}_{1}$ large enough so that $-b\zeta(-\wt{q}_{1})>\wt{q}_{12}$ and $\frac{1}{\wt{q}_{12}}+\frac{1}{1+\e}=1$.
\end{proof}

\section{Inverse increments }\label{inversemomentsincrements}
Here we study moments of $Q^{\delta}(a,a+x)$. We conjecture that the moments of $Q(a,a+x)$ are same to those of the unshifted inverse $Q(x)$ and multiplied by the shift $a$. By letting $a=0$, as in \Cref{INVERSE_MOMENTS} we can at the very most only consider $p\in \para{-\frac{(1+\frac{\gamma^{2}}{2})^{2}}{2\gamma^{2}}, \infty}$.

\begin{proposition}\label{prop:momentsofshiftedinverse}
Let $p\in \left(-\dfrac{\left(1+\frac{\gamma^{2}}{2}\right)^{2}}{2\gamma^{2}}, \infty \right)$, the height $\delta\leq 1$ and $a,x>0$.
\begin{itemize}
    \item (Positive moments) for all $p>0$
\begin{eqalign}
\Expe{(Q^{\delta}(a,a+x))^{p}}\leq& c_{1}x^{p_{1}}C_{pos}(a,\delta)+c_{2}x^{p_{2}},
\end{eqalign}
for $p_{1},p_{2}$ constrained as 
\begin{eqalign}
 \zeta(-p_{1})+p>1\tand p_{2}>p.   
\end{eqalign}

\item (Negative moments)For all  $p\in \left(0,\dfrac{\left(1+\frac{\gamma^{2}}{2}\right)^{2}}{2\gamma^{2}} \right)$
\begin{eqalign}
\Expe{(Q^{\delta}(a,a+x))^{-p}}\leq& \tilde{c}_{1}x^{-\tilde{p}_{1}}+\tilde{c}_{2}x^{-\tilde{p}_{2}},
\end{eqalign}
for $\tilde{p}_{1},\tilde{p}_{2}$ constrained as
\begin{eqalign}
\zeta(\tilde{p}_{1})>p+1    \tand     p>\tilde{p}_{2}.
\end{eqalign}
\end{itemize}
The constants dependence is
\begin{eqalign}
C_{pos}(a,\delta)&:=\maxp{\delta^{\zeta(-p_{1})+p-2},C_{inf,1}(\alpha,\delta,\rho),C_{inf,2}(\alpha,\delta,\rho) }   \\
C_{neg}(a,\delta)&:=C_{sup}(a,\delta,\rho)
\end{eqalign}
with the constants $C_{inf,1}$ , $C_{inf,2}$,$C_{sup}$ coming from \Cref{prop:shiftedGMCmoments}. 
\end{proposition}
\begin{remark}
Here one can again obtain precise exponents by optimizing. In the positive moments \cref{eq:positivemomentsinverseQx}, we optimize
\begin{eqalign}
q+\frac{\gamma^{2}}{2}(q^{2}+q)+1<p\Rightarrow     q=\frac{1}{2}\para{\sqrt{1+\frac{4p-2}{\beta}+\frac{1}{\beta^{2}}  }-1-\frac{1}{\beta}   }-\epsilon,
\end{eqalign}
for small $\e>0$. The only difference is $4p+2\to 4p-2$. In the negative moments \cref{eq:negativemomentsinverseQx}, we optimize
\begin{eqalign}
q-\frac{\gamma^{2}}{2}(q^{2}-q)>p+1\Rightarrow     q=\frac{1}{2}\para{-\sqrt{1+\frac{-4p-2}{\beta}+\frac{1}{\beta^{2}}  }+1+\frac{1}{\beta}   }+\epsilon.
\end{eqalign}
The only difference is $-4p+2\to -4p-2$.
\end{remark}

\begin{remark}
Compared to the inverse moments \Cref{INVERSE_MOMENTS} at least one major difference is the appearance of the shift $a>0$ in the bound.As mentioned in \Cref{differencetermunshifted}, it is unclear if this is a fundamental aspect of GMC or whether there are sharper bounds. The evidence so far points to growth in $a>0$. \\
In terms of the appearance of $\delta$ on the bounds $C_{pos},C_{nega}$ can be uniform in $\delta\leq 1$: the constants $C_{inf,1}$,$C_{inf,2}$,$C_{sup}$ are uniformly bounded when $a\leq \delta$ and we can choose constants $p,\tilde{p}$ to have the exponents of $\delta$ in $C_{pos}$,$C_{nega}$ be positive.
\end{remark}

\begin{proof}
\pparagraph{Positive moments of the inverse increments for $\delta\leq 1$}As in inverse moments \Cref{INVERSE_MOMENTS}, we similarly use layercake representation and Markov to upper bound in terms of the shifted $\eta$:
\begin{eqalign}\label{positivemoments}
&\Expe{\para{Q^{\delta}(a,a+x)}^{p}}\\
=&\int_{0}^{B_{\delta}}\Proba{x\geq \eta^{\delta}\para{ Q^{\delta}_{a}, Q^{\delta}_{a}+t} }pt^{p-1}\dt +\int_{B_{\delta}}^{\infty}\Proba{x\geq \eta^{\delta}\para{ Q^{\delta}_{a}, Q^{\delta}_{a}+t} }pt^{p-1}\dt,
\end{eqalign}
where $B_{\delta}:=\minp{2,\maxp{2\delta,1}}\in [1,2]$. This means that $B_{\delta}-\delta=\minp{2-\delta,\maxp{\delta,1-\delta}}\geq \frac{1}{2}$ when $\delta\leq 1$ and so $t\in [B_{\delta}-\delta,\infty)$ is bounded away from zero. First we study the $[B_{\delta},\infty)$-integral case. Here we will use the \Cref{deltaSMP} by lower bounding: 
\begin{eqalign}
 \Proba{x\geq \eta^{\delta}\para{ Q^{\delta}_{a}, Q^{\delta}_{a}+t} }\leq \Proba{x\geq \eta^{\delta}\para{ Q^{\delta}_{a}+\delta, Q^{\delta}_{a}+t}  }= \Proba{x\geq\eta^{\delta}\para{ 0, t-\delta} }.
\end{eqalign}
Therefore, by Markov inequality for $p_{2}>0$
\begin{eqalign}
 \int_{B_{\delta}}^{\infty}\Proba{x\geq\eta^{\delta}\para{ 0, t-\delta} }pt^{p-1}\dt\leq x^{p_{2}}\int_{B_{\delta}-\delta}^{\infty}\Expe{\para{\eta^{\delta}\para{0,t}}^{-p_{2}}}p(t+\delta)^{p-1}\dt
\end{eqalign}
and using the estimate $\Expe{\para{\eta^{\delta}[0,  t]}^{-p_{2}}}\leq c_{1} t^{-p_{2}}$, we require $1-p+p_{2}>1\doncl p_{2}>p$ to get a finite integral. \\
 For the first integral over $t\in [0,B_{\delta}]$, we require a bound on the  shifted moments $\Expe{\para{ \etamu{Q^{\delta}_{a},Q^{\delta}_{a}+t}{\delta}}^{-q } }$ and so we use \Cref{cor:shiftedGMCmoments}. In particular, here we still need to bound
\begin{eqalign}\label{eq:positivesmallt}
\int_{0}^{B_{\delta}}\Proba{x\geq \eta^{\delta}\para{ Q^{\delta}_{a}, Q^{\delta}_{a}+t}}t^{p-1}\dt.    
\end{eqalign}
For the integral over $0\leq t\leq 2\delta$, we apply Markov for $p_{1}>0$ in \Cref{eq:positivesmallt} to get an upper bound: 
\begin{eqalign}
\int_{0}^{\delta}t^{ \frac{\zeta(-p_{1}(1+\e))-1}{(1+\e)}}t^{p-1}\dt.    
\end{eqalign}
We get the following constraint:
\begin{eqalign}
\frac{\zeta(-p_{1}(1+\e))-1}{(1+\e)}+p-1>-1.  
\end{eqalign}
 For the remaining integral over $2\delta\leq t\leq B_{\delta}$, we get a $\delta^{-p_{1}-1+p-1}$.


\pparagraph{Negative moments of the inverse increments for $\delta\leq 1$}
Here we similarly obtain:
\begin{eqalign}
\Expe{\para{Q^{\delta}(a,a+x)}^{-p}}\leq&\int_{0}^{B_{\delta}}\Proba{ \eta^{\delta}\para{ Q^{\delta}_{a}, Q^{\delta}_{a}+t} \geq x}t^{-p-1}\dt\\
&+x^{-\tilde{p}_{2}}\int_{B_{\delta}}^{\infty}\Expe{\para{ \etamu{Q^{\delta}_{a},Q^{\delta}_{a}+t}{\delta}}^{\tilde{p}_{2} } }pt^{-p-1}\dt. 
\end{eqalign}
For the second integral since $\delta\leq 1$ we separate the part away from $\delta$
\begin{eqalign}
&p2^{(\tilde{p}_{2}-1)\vee 0}\int_{B_{\delta}}^{\infty}\para{\Expe{\para{ \etamu{Q^{\delta}_{a},Q^{\delta}_{a}+\delta}{\delta}}^{\tilde{p}_{2} } }}\frac{\dt}{t^{p+1}}\\
&+p2^{(\tilde{p}_{2}-1)\vee 0}\int_{B_{\delta}}^{\infty}\para{\Expe{\para{ \etamu{Q^{\delta}_{a}+\delta,Q^{\delta}_{a}+t}{\delta}}^{\tilde{p}_{2} } }}\frac{\dt}{t^{p+1}},
\end{eqalign}
For the second term we use the \Cref{deltaSMP} and \Cref{momentseta} to bound
\begin{eqalign}
\Expe{\para{ \etamu{Q^{\delta}_{a}+\delta,Q^{\delta}_{a}+t}{\delta}}^{\tilde{p}_{2} } }=&\Expe{\para{ \etamu{0,t-\delta}{\delta}}^{\tilde{p}_{2} } }\\
\leq &c\maxp{(t-\delta)^{\tilde{p}_{2}\para{1-\frac{\gamma^{2}}{2}(\tilde{p}_{2}-1)}},(t-\delta)^{\tilde{p}_{2}}    }.
\end{eqalign}
Since the lower bound is $B_{\delta}\geq 1$, if $\delta\to 0^{+}$, there is no risk of a divergent upper bound. So as in the inverse moments, we simply require $p>\tilde{p}_{2}$ for finiteness. For the first integral over $t\in [0,B_{\delta}]$, we require a bound on the  shifted moments $\Expe{\para{ \etamu{Q^{\delta}_{a},Q^{\delta}_{a}+t}{\delta}}^{q } }$. In particular, here we still need to bound
\begin{eqalign}
\int_{0}^{B_{\delta}}\Proba{ \eta^{\delta}\para{ Q^{\delta}_{a}, Q^{\delta}_{a}+t} \geq x}pt^{-p-1}\dt    
\end{eqalign}
using \Cref{cor:shiftedGMCmoments} as above. For the integral over $[0,\delta]$, we apply Markov for get the constraint 
\begin{eqalign}
\zeta(\tilde{p}_{1})-1-p-1>-1\doncl   \zeta(\tilde{p}_{1})>1+p.
\end{eqalign}
For the integral $[1,B_{\delta}]$ there is no constraint. 

\end{proof}
\subsection{Lower truncated inverse}
In this section we briefly mention the moments bounds for the lower truncated inverse $Q_{n}$. For the lower truncated shifted GMC $\eta_{n}$ we have the following analogue.
\begin{proposition}\label{shiftedmomentslowtrunc}
In the case of $F_{1}(x):=x^{p}$ for $p\in[1,\frac{\gamma^{2}}{2})$ and $F_{2}(x):=x^{-p}$ for $p>0$, the lower truncated shifted $F_{1}\para{\eta_{n}\spara{Q_{n}(a),Q_{n}(a)+t}}$ and $F_{2}\para{\eta_{n}\spara{Q_{n}(a),Q_{n}(a)+t}}$ satisfy the same bounds as in \Cref{prop:shiftedGMCmoments}.
\end{proposition}
\begin{proof}
As in the proof \Cref{prop:shiftedGMCmoments} we reduce $\para{\eta_{n}\spara{Q_{n}(a),Q_{n}(a)+t}}^{p}$ to the sup of $\para{\eta_{n}(T,T+t)}^{p}$. So because those functions are convex and $X_{n}:=\eta_{n}$ is a martingale, we have by Jensens (as mentioned in \cite[Lemma 6.]{bacry2003log}) that
\begin{eqalign}
\Expe{\sup_{T\in I_{k}}\para{\eta_{n}\spara{T,T+t}}^{p}}\leq \Expe{\sup_{T\in I_{k}}\para{\eta(T,T+t)}^{p}    }.
\end{eqalign}
Similarly, for the summation terms such as $\Proba{\eta_{n}^{\delta}(a_{k})\leq a}$ as in \Cref{eq:summterm}), we take Markov and apply Jensen's as above
\begin{equation*}
 \Expe{\para{\eta_{n}^{\delta}(a_{k})}^{-q}}\leq \Expe{\para{\eta^{\delta}(a_{k})}^{-q}}.
\end{equation*}
Terms like $\Proba{a\leq \eta^{\delta}(b_{k+1}) }$ are not part of the final estimates, but if we did, we just have to apply Markov with $p\in [1,\frac{2}{\gamma^{2}})$ in order to have convexity.
\end{proof}
 Therefore, we get the following proposition.
\begin{proposition}\label{inversemomentslowtrunc}
The lower truncated $Q_{n}(a,b)$ satisfies the same bounds as in \Cref{INVERSE_MOMENTS} but in the case of negative moments we further constrain $\wt{p}_{i}>1$.
\end{proposition}
\begin{proof}
As in \Cref{INVERSE_MOMENTS}, when studying the positive moments of $Q_{n}(a,b)$, we get negative moments of $\eta_{n}$, which are always convex functions and so the same bounds from \Cref{shiftedmomentslowtrunc}. For the negative moments of $Q_{n}(a,b)$ we get the positive moments of $\eta_{n}$
\begin{eqalign}\label{negativemomentsinv}
&\Expe{\para{Q^{\delta}_{x}\bullet Q^{\delta}_{a}}^{-p}}\lessapprox x^{-\wt{p}_{i}}\sum_{i=1,2,3} \int_{I_i}\Expe{\para{ \etamu{Q^{\delta}_{a},Q^{\delta}_{a}+t}{\delta_{0}}}^{\wt{p_{i}} } }pt^{-p-1}\dt
\end{eqalign}
for some intervals $I_i$. In order, to again use convexity, we simply need to restrict $\wt{p}_{i}>1$ in order to use the convexity needed in Jensens in \Cref{shiftedmomentslowtrunc}.
\end{proof}
In the case of $\wt{p}_{i}\in (0,1)$, we have the following.
\begin{proposition}\label{inversemomentslowtruncp01}
We have the following bound for $p\in ( 0,1)$
\begin{eqalign}
\Expe{(Q^{\delta}_{n}(a,a+x))^{-p}}\leq& \wt{c}_{1}x^{-\wt{p}_{1,1}}+\wt{c}_{2}x^{-\wt{p}_{2}}.
\end{eqalign}
with the only change of now requiring $p>\zeta(\wt{p}_{2})$ instead of $p>\wt{p}_{2}$ due to \Cref{truncatedlinear}.
\end{proposition}
\begin{proof}
Here ,we simply apply Jensen's to return to the previous case for $p=1$:
\begin{eqalign}
\Expe{\sup_{T\in I_{k}}\para{\eta_{n}\spara{T,T+t}}^{p}}
&\leq \para{\Expe{\sup_{T\in I_{k}}\para{\eta_{n}(T,T+t)}    }}^{p}\\
&\leq \para{\Expe{\sup_{T\in I_{k}}\para{\eta(T,T+t)}    }}^{p}.
\end{eqalign}
where we can then bound from \Cref{prop:maxmoduluseta}. And in the places where we use the strong Markov property we need to use the lower truncated bounds in \Cref{momentseta}:  for the positive moments in $p\in (0,1)$ of lower truncated GMC $\eta_{n}(0,t):=\int_{0}^{t}e^{U_{n}(s)}\ds$, we have the bound
\begin{eqalign}\label{truncatedlinear}
\Expe{\para{\eta_{n}(0,t)}^{p}}\lessapprox~t^{\zeta(p)}, \forall t\geq 0.
\end{eqalign}
\end{proof}
\newpage
\part{Cauchy sequence result}\label{part:rateofconvergence}\label{rateconv}
In this section we study the decay rate of $\Proba{\abs{Q_{n+m}(x)-Q_{n}(x)}\geq r}$ and of the $L^{\ell}$-difference \\
$\Expe{\abs{Q_{n+m}(x)-Q_{n}(x)}^{\ell}}$ for $\ell\in (1,2)$, $x\in [0,1] ,r>0$ as $n\to +\infty$. Using the Cauchy-criterion  we will obtain that $Q$ is the $L^{\ell}$-Cauchy limit of $Q_{n}$. This is analogous to studying the rate of $L^{2}$-convergence for GMC measures $\eta_{n}\stackrel{L^{2}}{\to} \eta$ which is known to be  $\Expe{(\eta_{\e}(1)-\eta_{\e/2}(1))^{2}}\leq c\e^{2-\gamma^{2}}$ \citep{berestycki2021gaussian,rhodes2015lectures}. Extending it to $\ell=2$, will likely require the use of good/bad thick points as in the GMC case.
\begin{theorem}[Cauchy sequence]\label{rateconvCauchy}\label{LellCauchy}Fix $\gamma<\frac{1}{12\sqrt{2}}$, $\delta>0$ and let any strictly decreasing sequence $\delta_{n}=2^{-\delta n}$.
 Fix an interval $[\e,M]$ for some $\e,M>0$ and $x\in [\e,M]$. Then the sequence  $\set{Q_{n}(x)}_{n\geq 1}$ is $L^{\ell}$-Cauchy for any $\ell\in (1,2)$
\begin{eqalign}
Q_{n}(x)\stackrel{L^{\ell}}{\to}Q(x).
\end{eqalign}
\end{theorem}
\begin{remark}
This result is not being used anywhere else. It is interesting that even a simple looking statement as this seems to be very tricky to obtain.\\
The reason for $\gamma<\frac{1}{12\sqrt{2}}$ originates from the use of moments in $p\in (0,1)$ in \cref{prop:maxmoduluseta}.
\end{remark}

\subsection{Proof of \sectm{\Cref{LellCauchy}}}
By standard Cauchy series argument, it is enough to establish a decaying bound
\begin{eqalign}
 \Expe{\abs{Q_{n+m}(x)-Q_{n}(x)}^{\ell}}\lessapprox  c 2^{a_{2}m}2^{-a_{1} n}
\end{eqalign}
for some exponents $a_{1},a_{2}>0$. 
To show this, we will prove that
\begin{eqalign}\label{eq:rateCauchysummable}
 \Expe{\abs{Q_{n+m}(x)-Q_{n}(x)}^{\ell}}\leq c_{x}\para{ G_{n}^{-2}\para{\frac{\delta_{n}}{\delta_{n+m}} }^{\gamma^{2}}\delta_{n}^{p_{1}\alpha}+G_{n}^{p_{2}}},
\end{eqalign}
where the constant \textit{depends} on $x$ in both increasing/decreasing way i.e. $c_{x}\approx x^{-p_{1}}+x^{p_{2}}$, the exponents $p_{1},p_{2}\in (0,1)$ are some small enough numbers and for any sequence $G_{n}\to 0$. In particular to ensure that the overall bound limits to zero as $n\to \infty$, we require
\begin{eqalign}
\lim_{n\to +\infty}\para{G_{n}^{-2}\para{\frac {\delta_{n}}{\delta_{n+m}}}^{\gamma^{2}}\delta_{n}^{p_{1}\alpha}}=0.
\end{eqalign}


We state here the two main  \cref{lemm:precomposedetainv},\cref{backwardshiftedGMC} that we will use and that we will prove below. The first lemma simply obtains a rate on $\eta_{n+m}(Q_{n}(x))\to x. $
\begin{customlemma}{27}
Fix strictly decreasing sequence $\delta_{n}\to 0$ satisfying $\frac{\delta_{n+m}}{\delta_{n}}=\delta_{1+m}$ for every $n,m\geq 1$. There exist some $p\in (0,1)$ such that for $\eta_{n+m}(Q_{n}(x))  $ we have
\begin{eqalign}\label{Bl2}
\Expe{\para{\eta_{n+m}(Q_{n}(x))-x}^{2}}&\leq C_{sup}(x,1,1/2 )\para{\frac{\delta_{n}}{\delta_{n+m}} }^{2\beta}\delta_{n}^{p\alpha}x^{2-p}\\
\end{eqalign}
where $C_{sup}(x,1,1/2 )$ are from \Cref{cor:shiftedGMCmoments}.
\end{customlemma}
The second lemma is more technical. It involves studying the negative moment of a \textit{backward shifted} GMC $\eta(Q-r,Q)$.
\begin{customlemma}{28}
For $x,r>0$ we have a split bound: $\twhen r\geq 1/2$
\begin{eqalign}\label{eq:backshifgmc}
&\Proba{\eta_{n+m}\spara{\maxp{0,Q_{n}(x)-r},Q_{n}(x)}\leq
\wt{g}_{n,r}, Q_{n}(x)\geq r }    \\
\leq &B_{n,r}^{\frac{\lambda_{1}p_{2}-1}{\lambda_{1}p_{2}}}\minp{\para{x/r}^{\frac{1-\lambda_{1}p_{2}}{\lambda_{1}}},B_{n,r}^{1/\lambda_{1}p_{2}-1} (x/r)^{p_{2}}}+B_{n,r}^{\frac{\lambda_{1}p_{2}-1}{\lambda_{1}p_{2}}}\maxp{\para{x/r}^{1/\lambda_{1}},B_{n,r}^{1/\lambda_{1}p_{2}}},\\
&\tand \twhen r<1/2\\
&\Proba{\eta_{n+m}\spara{\maxp{0,Q_{n}(x)-r},Q_{n}(x)}\leq
\wt{g}_{n,r}, Q_{n}(x)\geq r }    \\
\leq &\wt{ B}_{n,r}^{\frac{\lambda_{2}\wt{p}_{2}-1}{\lambda_2\wt{p}_{2}}}\minp{\para{x/(1-r)}^{\frac{1-\lambda_{2}\wt{p}_{2}}{\lambda_{2}}},\wt{ B}_{n,r}^{1/\lambda_{2}\wt{p}_{2}-1} (x/(1-r))^{\wt{p}_{2}}}\\
&+\wt{ B}_{n,r}^{\frac{\lambda_{2}\wt{p}_{2}-1}{\lambda_2\wt{p}_{2}}}\maxp{\para{x/(1-r)}^{1/\lambda_{2}},\wt{ B}_{n,r}^{1/\lambda_{2}\wt{p}_{2}}} +\wt{g}_{n,r}^{p_{3}}  r^{\zeta(-p_{3})-1},
\end{eqalign}
for $B_{n,r}:=g_{n,r}^{p_{1}}  r^{-p_{1}}$ and $\wt{B}_{n,r}:=g_{n,r}^{\wt{p}_{1}}r^{\zeta(-\wt{p}_{1})-1}$ and the exponents are positive and satisfy the constraints $\lambda_{i}\in (0,1),\lambda_{1}p_{2}>1$ and $\lambda_{2}\wt{p}_{2}>1$.
\end{customlemma}

\subsection{Proof of main estimate \sectm{\cref{eq:rateCauchysummable}}}
\begin{proof}
We start with using layercake and splitting into two different large deviations
\begin{eqalign}
\Expe{\abs{Q_{n+m}(x)-Q_{n}(x)}^{\ell}}=&\int_{0}^{\infty}\Proba{\abs{Q_{n+m}(x)-Q_{n}(x)}\geq r}\ell r^{\ell-1}\dr\\
\leq &\int_{0}^{\infty}\Proba{Q_{n+m}(x)-Q_{n}(x)\geq r}\ell r^{\ell-1}\dr\\
&+\int_{0}^{\infty}\Proba{Q_{n+m}(x)-Q_{n}(x)\leq -r}\ell r^{\ell-1}\dr.
\end{eqalign}
\proofparagraph{Step 1: Cauchy difference $Q_{n+m}(x)-Q_{n}(x)\geq r $ for $r\geq 2$}
We start with rewriting the event $\set{Q_{n+m}(x)-Q_{n}(x)\geq r }$ in terms of $\eta_{n+m}$
\begin{eqalign}\label{eq:shiftedforwardevent}
\set{Q_{n+m}(x)\geq r+Q_{n}(x)}=&\set{x\geq \int_{0}^{Q_{n}(x)+r}e^{U^{1}_{n+m}(s)}\ds }\\
=&\set{x-\eta_{n+m}(Q_{n}(x))\geq \int_{Q_{n}(x)}^{Q_{n}(x)+r}e^{U^{1}_{n+m}(s)}\ds }.
\end{eqalign}
Here we can use \Cref{deltaSMP} by removing the $\delta=1$-common part and then apply Markov inequality
\begin{eqalign}
&\Proba{x-\eta_{n+m}(Q_{n}(x))\geq \int_{Q_{n}(x)}^{Q_{n}(x)+r}e^{U^{1}_{n+m}(s)}\ds } \\
\leq &\Proba{x-\eta_{n+m}(Q_{n}(x))\geq \int_{Q_{n}(x)+1}^{Q_{n}(x)+r}e^{U^{1}_{n+m}(s)}\ds } \\
\leq &\Expe{\abs{x-\eta_{n+m}(Q_{n}(x))}^{2}\para{\int_{Q_{n}(x)+1}^{Q_{n}(x)+r}e^{U^{1}_{n+m}(s)}\ds}^{-2} }\\
=&\Expe{\abs{x-\eta_{n+m}(Q_{n}(x))}^{2}}\Expe{\para{\eta_{n+m}(r-1)}^{-2} }.
\end{eqalign}
 For the first factor we apply \Cref{lemm:precomposedetainv} 
\begin{eqalign}\label{eq:firstterml2}
\Expe{\para{x-\eta_{n+m}(Q_{n}(x))}^{2}}\leq  C_{sup}(x,1,1/2 )\para{\frac{\delta_{n}}{\delta_{n+m}} }^{2\beta}\delta_{n}^{p\alpha}x^{2-p}.
\end{eqalign}
For the second factor we use the negative moments \Cref{momentseta} to bound $\Expe{\para{\eta_{n+m}(r-1)}^{-2} }\leq c(r-1)^{-2}$ because $r\geq 2$. So since $\ell\in (1,2)$, we get finiteness for the integral over $[2,\infty)$.
\proofparagraph{Step 2: Cauchy difference $Q_{n+m}(x)-Q_{n}(x)\geq r $ for $r\in [0,2]$}
In this case we bound by two terms:
\begin{eqalign}
&\Proba{x-\eta_{n+m}(Q_{n}(x))\geq \int_{Q_{n}(x)}^{Q_{n}(x)+r}e^{U^{1}_{n+m}(s)}\ds } \\
&\leq\Proba{x-\eta_{n+m}(Q_{n}(x))\geq g_{n,r}}+\Proba{g_{n,r}\geq \int_{Q_{n}(x)}^{Q_{n}(x)+r}e^{U^{1}_{n+m}(s)}\ds}
\end{eqalign}
for some $g_{n,r}$ to be chosen below. For the first probability term we again apply Markov inequality for $p=2$ and the  \Cref{lemm:precomposedetainv}
\begin{eqalign}\label{eq:firstterml23}
\frac{1}{g_{n,r}^{2}}\Expe{\para{x-\eta_{n+m}(Q_{n}(x))}^{2}}\leq \frac{1}{g_{n,r}^{2}} C_{sup}(x,1,1/2 )\para{\frac{\delta_{n}}{\delta_{n+m}} }^{2\beta}\delta_{n}^{p\alpha}x^{2-p}.
\end{eqalign}
For the second probability term, we use  \Cref{cor:shiftedGMCmoments} for $F_{2}:=x^{-p_{1}}$ and any $p_{1}>0$ (but for the lower truncated as mentioned in \Cref{shiftedmomentslowtrunc}):
\begin{eqalign}\label{eq:secondtermshifted}
\Proba{g_{n,r}\geq \eta_{n+m}^{1}\spara{Q_{n}(x),Q_{n}(x)+r}}&\lessapprox g_{n,r}^{p_{1}}r^{\alpha(p_{1})}C_{inf}(x,\rho),
\end{eqalign}
where $\alpha_{\e}(p):=\frac{\zeta(-p(1+\e))-1}{1+\e}$ for $t\leq 2\delta$ and $\alpha_{\e}(p):=-p$ for $t\geq 2\delta$.
\pparagraph{Integrating over $r\in [0,2]$} The integral over $[1,2]$ has no singularity, so we don't focus on it. For the first term in \Cref{eq:firstterml2} we require
\begin{eqalign}
\int_{0}^{1}\frac{1}{g_{n,r}^{2}}r^{\ell-1}\dr   <\infty.
\end{eqalign}
For the term in \Cref{eq:secondtermshifted}
we require
\begin{eqalign}
\int_{0}^{1-\rho} g_{n,r}^{p_{1}}r^{\zeta(-p_{1})-1}r^{\ell-1}\dr  <\infty,
\end{eqalign}
where we ignored the $\epsilon>0$ for simplicity. By letting $g_{n,r}:=G_{n}r^{\alpha_{1}}$ for some $\alpha_{1}>0,G_{n}>0$, we require in summary:
\begin{eqalign}
&\alpha_{1}<\frac{\ell}{2},\qquad \tand   -p_{1}\alpha_{1}-\zeta(-p_{1})+1<\ell.
\end{eqalign}
We fix $1<\wt{q}_{11}<2$. Since $-\zeta(-p_{1})=p_{1}+\frac{\gamma^{2}}{2}(p_{1}^{2}+p_{1})$ is a increasing function of $p_{1}>0$, we can take small enough $p_{1}$ so that the above LHS sum of exponents is less than $\ell\in (1,2)$.
\proofparagraph{Step 3: Cauchy difference $Q_{n+m}(x)-Q_{n}(x)\leq -r $}
Ideally, as in \Cref{eq:shiftedforwardevent}, we would like to study
\begin{eqalign}
\set{Q_{n+m}(x)-Q_{n}(x)\leq -r}=\set{x-\eta_{n}(Q_{n+m}(x))\geq \int_{Q_{n+m}(x)}^{Q_{n+m}(x)+r}e^{U^{1}_{n}(s)}\ds}
\end{eqalign}
in order to use \Cref{deltaSMP}. However, as explained below in trying to study the deviation
\begin{eqalign}
\eta_{n}(Q_{n+m}(x))-x=\int_{0}^{x} e^{-U^{n}_{n+m}(Q_{n+m}(s))}\ds -x,
\end{eqalign}
we now come across the issue of entangled $-U^{n}_{n+m}(Q_{n+m}(s))$ i.e. we cannot longer condition on the scales of $Q_{n+m}$ as in the proof of \Cref{lemm:precomposedetainv}.  So instead we work with a backward shifted GMC
\begin{eqalign}\label{backwarcshiftCauc}
\set{0\leq Q_{n+m}(x)\leq Q_{n}(x)-r}=\set{\eta_{n+m}(Q_{n}(x))-x\geq \int_{Q_{n}(x)-r}^{Q_{n}(x)}e^{U^{1}_{n+m}(s)}\ds, Q_{n}(x)\geq r}.
\end{eqalign}
We repeat the split
\begin{eqalign}\label{eq:twoprobabilityterms}
\Proba{x-\eta_{n+m}(Q_{n}(x))\geq \wt{g}_{n,r}}+\Proba{\wt{g}_{n,r}\geq \int_{\maxp{0,Q_{n}(x)-r}}^{Q_{n}(x)}e^{U^{1}_{n+m}(s)}\ds, Q_{n}(x)\geq r}.
\end{eqalign}
For the first probability term in \cref{eq:twoprobabilityterms} we again use the same \Cref{lemm:precomposedetainv}
\begin{eqalign}
\Proba{x-\eta_{n+m}(Q_{n}(x))\geq \wt{g}_{n,r}}\leq \frac{1}{\wt{g}_{n,r}^{2}}  C_{sup}(x,1,1/2 )\para{\frac{\delta_{n}}{\delta_{n+m}} }^{2\beta}\delta_{n}^{p\alpha}x^{2-p}.
\end{eqalign}
For the second probability term in \cref{eq:twoprobabilityterms} we use the \Cref{backwardshiftedGMC}. For $x,r>0$ we have a split bound:
\begin{eqalign}\label{eq:Bnrbound}
&\Proba{\eta_{n+m}\spara{\maxp{0,Q_{n}(x)-r},Q_{n}(x)}\leq
\wt{g}_{n,r}, Q_{n}(x)\geq r }\\\\
\lessapprox& {\branchmat{ B_{n,r}^{\frac{\lambda_{1}p_{2}-1}{\lambda_{1}p_{2}}}\minp{\para{x/r}^{\frac{1-\lambda_{1}p_{2}}{\lambda_{1}}},B_{n,r}^{1/\lambda_{1}p_{2}-1} (x/r)^{p_{2}}}& \tcwhen r\geq 1/2\\
+B_{n,r}^{\frac{\lambda_{1}p_{2}-1}{\lambda_{1}p_{2}}}\maxp{\para{x/r}^{1/\lambda_{1}},B_{n,r}^{1/\lambda_{1}p_{2}}} & \\
\wt{ B}_{n,r}^{\frac{\lambda_{2}\wt{p}_{2}-1}{\lambda_2\wt{p}_{2}}}\minp{\para{x/(1-r)}^{\frac{1-\lambda_{2}\wt{p}_{2}}{\lambda_{2}}},\wt{ B}_{n,r}^{1/\lambda_{2}\wt{p}_{2}-1} (x/(1-r))^{\wt{p}_{2}}} & \tcwhen r<1/2\\
+\wt{ B}_{n,r}^{\frac{\lambda_{2}\wt{p}_{2}-1}{\lambda_2\wt{p}_{2}}}\maxp{\para{x/(1-r)}^{1/\lambda_{2}},\wt{ B}_{n,r}^{1/\lambda_{2}\wt{p}_{2}}}  &\\
+\wt{g}_{n,r}^{p_{3}}  r^{\zeta(-p_{3})}&}},
\end{eqalign}
for $B_{n,r}:=\wt{g}_{n,r}^{p_{1}}  r^{-p_{1}}$ and $\wt{B}_{n,r}:=\wt{g}_{n,r}^{\wt{p}_{1}}r^{\zeta(-\wt{p}_{1})-1}$ and the exponents are positive and satisfy the constraints $\lambda_{i}\in (0,1),\lambda_{1}p_{2}>1$ and $\lambda_{2}\wt{p}_{2}>1$.
\pparagraph{Integral over $r\geq 1$}
From the first term we obtain the constraint
\begin{eqalign}
 \int_{1}^{\infty}\frac{1}{\wt{g}_{n,r}^{2}}r^{\ell-1}\dr<\infty.
\end{eqalign}
From the second term we use \cref{eq:Bnrbound}, it contains min and max term. For the minimum term we obtain the constraint
\begin{eqalign}
 \int_{1}^{\infty}\minp{\para{\wt{g}_{n,r}^{p_{1}}  r^{-p_{1}}}^{1-1/\lambda_{1}p_{2}} \para{x/r}^{\frac{1-\lambda_{1}p_{2}}{\lambda_{1}}},(x/r)^{p_{2}}}r^{\ell-1}\dr<\infty.
\end{eqalign}
For the maximum term, since $\frac{x}{r}\leq M$ and $B_{n,r}\to 0$ as $r\to +\infty$(from the choice below), we can ignore the maximum coefficient and just write $ B_{n,r}^{\frac{\lambda_{1}p_{2}-1}{\lambda_{1}p_{2}}}$ to bound by
\begin{eqalign}
 \int_{1}^{\infty}\para{\wt{g}_{n,r}^{p_{1}}  r^{-p_{1}}}^{\frac{\lambda_{1}p_{2}-1}{\lambda_{1}p_{2}}}r^{\ell-1}\dr<\infty.
\end{eqalign}
By letting $\wt{g}_{n,r}=G_{n}r^{c_{2}}$ for some $c_{2}>0,G_{n}>0$, we require in summary:
\begin{eqalign}
&c_{2}>\frac{\ell}{2} \tand  p_{2}-\ell>0\\
&\para{(1-c_{2})p_{1}+1}(\frac{\lambda_{1}p_{2}-1}{\lambda_{1}p_{2}}) -\ell>0\\
&\para{(1-c_{2})p_{1}}(\frac{\lambda_{1}p_{2}-1}{\lambda_{1}p_{2}}) -\ell>0\\
\end{eqalign}
So here we take $1-c_{2}=\e>0$ for small enough $\e>0$ (which is possible since $\ell<2$) and then take $p_{1},p_{2}$ large enough so that the LHS of the second constraint is larger than $\ell$.
\pparagraph{Integral over $r\in [\frac{1}{2},1]$}
Here we don't have any constraints for the integral over $r$.
\pparagraph{Integral over $r\in [0,\frac{1}{2}]$}
From the first term we get the constraint
\begin{eqalign}
\int_{0}^{1/2}\frac{1}{\wt{g}_{n,r}^{2}}r^{\ell-1}\dr   <\infty.
\end{eqalign}
From the second term we get the constraint
\begin{eqalign}
&\int_{0}^{1/2} \minp{\para{\wt{g}_{n,r}^{\wt{p}_{1}}r^{\zeta(-\wt{p}_{1})-1}}^{1-1/\lambda_{2}\wt{p}_{2}} \para{x/(1-r)}^{\frac{1-\lambda_{2}\wt{p}_{2}}{\lambda_{2}}},(x/(1-r))^{\wt{p}_{2}}}r^{\ell-1}\\
&+\wt{g}_{n,r}^{\wt{p}_{1}}r^{\zeta(-\wt{p}_{1})-1}r^{\ell-1}+\wt{g}_{n,r}^{p_{3}}  r^{\zeta(-p_{3})-1}r^{\ell-1}\dr   <\infty.
\end{eqalign}
By letting $\wt{g}_{n,r}=G_{n}r^{\alpha_{4}}$ for some $\alpha_{4}>0,G_{n}\in (0,1)$, we require in summary:
\begin{eqalign}
&1)\alpha_{4}<\frac{\ell}{2},\qquad 2) (\alpha_{4}\wt{p}_{1}+\zeta(-\wt{p}_{1})-1)(\frac{\lambda_{2}\wt{p}_{2}-1}{\lambda_{2}\wt{p}_{2}}) +\ell>0 \\
&3)\wt{p}_{1}\alpha_{4}+\zeta(-\wt{p}_1)-1+\ell>0\tand p_{3}\alpha_{4}+\zeta(-p_{3})-1+\ell>0.
\end{eqalign}
First maximize $\alpha_{4}=\frac{\ell}{2}-\e_{4}$ for arbitrarily small $\e_{4}$. For constraint (2), we take $\wt{p}_{2}$ arbitrarily large and $\wt{p}_{1}$ small enough so that the first term is smaller than $\ell>1$. For (3) we take $\wt{p}_{1},p_{3}$ small enough so that they are smaller than $\ell$.
\pparagraph{Uniformity in $x\in [\e,M]$}
In the bounds involving
\begin{eqalign}
&\minp{\para{\wt{g}_{n,r}^{p_{1}}  r^{-p_{1}}}^{1-1/\lambda_{1}p_{2}} \para{x/r}^{\frac{1-\lambda_{1}p_{2}}{\lambda_{1}}},(x/r)^{p_{2}}}\\
&\minp{\para{\wt{g}_{n,r}^{\wt{p}_{1}}r^{\zeta(-\wt{p}_{1})-1}}^{1-1/\lambda_{2}\wt{p}_{2}} \para{x/(1-r)}^{\frac{1-\lambda_{2}\wt{p}_{2}}{\lambda_{2}}},(x/(1-r))^{\wt{p}_{2}}}
\end{eqalign}
we see that if restrict $x\in [\e,M]$ for fixed $\e,M>0$, then we get uniform limits $G_{n}^{p_{1}(1-\frac{1}{\lambda_{1}p_{1}})},G_{n}^{\tilde{p}_{1}(1-\frac{1}{\lambda_{2}\tilde{p}_{2}})}$ respectively as $n\to +\infty$.

\end{proof}

\subsection{Precomposed term \sectm{$\eta_{n+m}(Q_{n}(x))$}}
We state here one of the lemmas used. Here we study the deviation of
\begin{eqalign}\label{precompo1}
\eta_{n+m}(Q_{n}(x))-x\to 0 \tas n\to +\infty,
\end{eqalign}
which is reasonable because as $n\to +\infty$, the upper truncated $\eta^{n}([0,x])$ converges to $x$. We have a nice formula using inverse function theorem. Since $Q_{n}(x)$ is the inverse of the differentiable $\eta_{n}(x)$, we have
\begin{eqalign}
 \frac{dQ_{n}(x)}{dx}=  e^{-U^{1}_{n}(Q_{n}(x))}.
\end{eqalign}
Here recall $U_{n}(x):=U_{\delta_{n}}^{1}(x)$. Therefore, by change of variables we have
\begin{eqalign}
\eta_{n+m}(Q_{n}(x))&=\int_{0}^{Q_{n}(x)} e^{U^{1}_{n+m}(s)}\ds\\
&=\int_{0}^{x} e^{U^{1}_{n+m}(Q_{n}(s))}e^{-U^{1}_{n}(Q_{n}(s))}\ds\\
&=\int_{0}^{x} e^{U^{n}_{n+m}(Q_{n}(s))}\ds.
\end{eqalign}
In the case of $n=\infty$, we can also use the change of variables formula \cite[(4.9) Proposition.]{revuz2013continuous}: Consider any increasing, possibly infinite, right-continuous function $A:[0,\infty)\to [0,\infty]$ with inverse $C_{s}$, then if $f$ is a nonnegative Borel function on $[0,\infty)$ we have
\begin{eqalign}
\int_{[0,\infty)}f(u)dA_{u}=    \int_{[0,\infty)}f(C_{s})\ind{C_{s}<\infty}\ds.
\end{eqalign}
Here though we cannot use the decoupling technique because the presence of $Q_{n}$ creates long range correlation. However,  by conditioning on $U^{1}_{n}$ we have that $\Exp_{U^{1}_{n}}\spara{\eta_{n+m}(Q_{n}(x))}=x$.
\begin{lemma}\label{lemm:precomposedetainv}
Fix strictly decreasing sequence $\delta_{n}\to 0$ satisfying $\frac{\delta_{n+m}}{\delta_{n}}=\delta_{1+m}$ for every $n,m\geq 1$. There exist some $p\in (0,1)$ such that for $\eta_{n+m}(Q_{n}(x))  $ we have
\begin{eqalign}
\Expe{\para{\eta_{n+m}(Q_{n}(x))-x}^{2}}&\leq C_{sup}(x,1,1/2 )\para{\frac{\delta_{n}}{\delta_{n+m}} }^{2\beta}\delta_{n}^{p\alpha}x^{2-p}\\
\end{eqalign}
where $C_{sup}(x,1,1/2 )$ are from \Cref{cor:shiftedGMCmoments}.
\end{lemma}

\begin{proof}
Due to the context of  \Cref{cor:shiftedGMCmoments} for $p\in (0,1)$, for simplicity we assume $x\leq \frac{1}{e}$ since it corresponds to the more singular behaviour. By first conditioning on $U^{1}_{n}$ because it is independent of $U^{n}_{n+m}$ we have:
\begin{eqalign}
\Expe{\para{\eta_{n+m}(Q_{n}(x))-x}^{2}}&=\Expe{\Exp_{U^{1}_{n}}\spara{\para{\int_{0}^{x} e^{U^{n}_{n+m}(Q_{n}(s))}\ds -x}^{2}}}\\
&=\Expe{\Exp_{U^{1}_{n}}\spara{\para{\eta_{n+m}(Q_{n}(x))}^{2}}}-x^{2}\\
&=\iint_{[0,x]^{2}}\Expe{\Exp_{U^{1}_{n}}\spara{e^{U^{n}_{n+m}(Q_{n}(s))+U^{n}_{n+m}(Q_{n}(t))}}   }\ds\dt -x^{2}\\
&\leq \Expe{\iint\limits_{[0,x]^{2}\cap \Delta_{n} }\para{\frac{\delta_{n}}{\abs{Q_{n}(s)-Q_{n}(t)}\vee  \delta_{n+m}} }^{\gamma^{2}}\ds\dt}- \Exp\abs{\Delta_{n}},
\end{eqalign}
where $\Delta_{n}:=\set{(s,t)\in [0,x]^{2}: \abs{Q_{n}(s)-Q_{n}(t)}\leq \delta_{n} }$ and we used that $\Exp_{U^{1}_{n}}\spara{\eta_{n+m}(Q_{n}(x))}=x$ and we kept only the logarithmic part of the covariance.
First we estimate the expected size of the diagonal set
\begin{eqalign}
\Exp\abs{\Delta_{n}}=  2\iint\limits_{[0,x]^{2}, s<t } \Proba{Q_{n}(s,t)\leq  \delta_{n}}\ds\dt=2\iint\limits_{[0,x]^{2}, s<t } \Proba{t-s\leq \eta_{n}(Q_{n}(s),Q_{n}(s)+\delta_{n})}\ds\dt
\end{eqalign}
and then bound by Markov for $p\in (0,1)$:
\begin{eqalign}
\iint\limits_{[0,x]^{2}, s<t }\frac{1}{(t-s)^{p}}\Expe{\para{\eta_{n}(Q_{n}(s),Q_{n}(s)+\delta_{n})}^{p}}\ds\dt.
\end{eqalign}
Here we apply \Cref{cor:shiftedGMCmoments} for $p\in (0,1)$
\begin{eqalign}
\Expe{\para{\eta_{n}(Q_{n}(s),Q_{n}(s)+\delta_{n})}^{p}} \leq&
\Expe{\supl{0\leq T\leq 1}\para{\eta_{n}(T,T+\delta_{n})}^{p}  } C_{sup}(x,1,1/2 )\\
\leq &c\delta_{n}^{p\alpha} C_{sup}(x,1,1/2 )=:c_{x}\delta_{n}^{p\alpha}.
\end{eqalign}
So then we are just left with the integral that evaluates to $\iint\limits_{[0,x]^{2}, s<t }\frac{1}{(t-s)^{p}}\ds\dt=\frac{1}{(1-p)(2-p)}x^{2-p}$ and all together
\begin{eqalign}\label{eq:estimatefirst}
\Exp\abs{\Delta_{n}}\leq   c_{x}\delta_{n}^{p\alpha}\frac{1}{(1-p)(2-p)}x^{2-p}.
\end{eqalign}
\pparagraph{Double integral term:} Next we study the double integral term. Here we again apply the estimate in \Cref{cor:shiftedGMCmoments} for $p\in (0,1)$
\begin{eqalign}\label{eq:estimatesecon}
&\Expe{\iint\limits_{[0,x]^{2}\cap \Delta_{n} }\Exp_{U^{1}_{n}}\spara{\para{\frac{\delta_{n}}{\maxp{\abs{Q_{n}(s)-Q_{n}(t)},  \delta_{n+m}} } }^{\gamma^{2}}}\ds\dt}\\
&\leq 2\para{\frac{\delta_{n}}{\delta_{n+m}} }^{\gamma^{2}} \iint\limits_{[0,x]^{2}, s<t } \Proba{Q_{n}(s,t)\leq  \delta_{n}}\ds\dt \\
&\leq c_{x}\para{\frac{\delta_{n}}{\delta_{n+m}} }^{\gamma^{2}}\delta_{n}^{p\alpha}\frac{1}{(1-p)(2-p)}x^{2-p}\\
&= c_{x}\para{2^{\delta m} }^{\gamma^{2}}2^{-\delta p\alpha}\frac{1}{(1-p)(2-p)}x^{2-p}.
\end{eqalign}
Since the estimate in \Cref{eq:estimatefirst} is strictly smaller than \Cref{eq:estimatesecon}, we keep the latter.
\end{proof}
\subsection{Backward-shifted GMC}
We state here another of the lemmas used. For the \Cref{rateconvCauchy}, we also need to estimate in \Cref{backwarcshiftCauc} a backward-shifted GMC.
\begin{lemma}\label{backwardshiftedGMC}
For $x,r>0$ we have a split bound: $\twhen r\geq 1/2$
\begin{eqalign}
&\Proba{\eta_{n+m}\spara{\maxp{0,Q_{n}(x)-r},Q_{n}(x)}\leq
\wt{g}_{n,r}, Q_{n}(x)\geq r }    \\
\leq &B_{n,r}^{\frac{\lambda_{1}p_{2}-1}{\lambda_{1}p_{2}}}\minp{\para{x/r}^{\frac{1-\lambda_{1}p_{2}}{\lambda_{1}}},B_{n,r}^{1/\lambda_{1}p_{2}-1} (x/r)^{p_{2}}}+B_{n,r}^{\frac{\lambda_{1}p_{2}-1}{\lambda_{1}p_{2}}}\maxp{\para{x/r}^{1/\lambda_{1}},B_{n,r}^{1/\lambda_{1}p_{2}}},\\
&\tand \twhen r<1/2\\
&\Proba{\eta_{n+m}\spara{\maxp{0,Q_{n}(x)-r},Q_{n}(x)}\leq
\wt{g}_{n,r}, Q_{n}(x)\geq r }    \\
\leq &\wt{ B}_{n,r}^{\frac{\lambda_{2}\wt{p}_{2}-1}{\lambda_2\wt{p}_{2}}}\minp{\para{x/(1-r)}^{\frac{1-\lambda_{2}\wt{p}_{2}}{\lambda_{2}}},\wt{ B}_{n,r}^{1/\lambda_{2}\wt{p}_{2}-1} (x/(1-r))^{\wt{p}_{2}}}\\
&+\wt{ B}_{n,r}^{\frac{\lambda_{2}\wt{p}_{2}-1}{\lambda_2\wt{p}_{2}}}\maxp{\para{x/(1-r)}^{1/\lambda_{2}},\wt{ B}_{n,r}^{1/\lambda_{2}\wt{p}_{2}}} +\wt{g}_{n,r}^{p_{3}}  r^{\zeta(-p_{3})-1},\\
\end{eqalign}
for $B_{n,r}:=g_{n,r}^{p_{1}}  r^{-p_{1}}$ and $\wt{B}_{n,r}:=g_{n,r}^{\wt{p}_{1}}r^{\zeta(-\wt{p}_{1})-1}$ and the exponents are positive and satisfy the constraints $\lambda_{i}\in (0,1),\lambda_{1}p_{2}>1$ and $\lambda_{2}\wt{p}_{2}>1$.

\end{lemma}
\begin{proof}
\pparagraph{Case \sectm{$r\geq 1/2$}}
As in \Cref{prop:shiftedGMCmoments}. we decompose the variable $Q_{n}(x)$ using the events $\ind{b_{k}\leq Q_{n}(x)<b_{k+1}}$
\begin{eqalign}
&\Proba{\eta_{n+m}\para{Q_{n}(x)-r,Q_{n}(x)}\leq g_{n,r},Q_{n}(x)\geq r } \\
&\leq \sum_{k\geq 0}
\Proba{\infl{T\in [b_{k},b_{k+1}] }\eta_{n+m}\para{T-r,T}\leq g_{n,r},b_{k}\leq Q_{n}(x)<b_{k+1}},
\end{eqalign}
where we set $b_{k}:=rk^{\lambda_{1}}$ for some $\lambda_{1}\in (0,1)$ and $k\geq 1$. By MVT $b_{k+1}-b_{k}\leq \frac{r}{k^{1-\lambda_{1}}}\leq r$. Using  translation invariance and \Cref{prop:minmodeta} for $r\geq 2\delta:=2$ we bound
\begin{eqalign}
\Proba{\infl{T\in [b_{k},b_{k+1}] }\eta_{n+m}\spara{T-r,T}\leq g_{n,r}}=&\Proba{\infl{T\in [0,b_{k+1}-b_{k}] }\eta_{n+m}\spara{T,T+r}\leq g_{n,r}}\\
\leq &\Proba{\infl{T\in [0,r] }\eta_{n+m}\spara{T,T+r}\leq g_{n,r}}\\
\leq & cg_{n,r}^{p_{1}}  r^{-p_{1}} =:B_{n,r}.
\end{eqalign}
If $r\in [1/2,2]$, we skip the first few terms that might satisfy $rk^{\lambda_{1}}\leq 2\doncl k\leq (r/2)^{-1/\lambda_{1}}\leq 4^{1/\lambda_{1}}\leq 5$ and just bound by one; we ignore this in the overall bound since $B_{n,r}<B_{n,r}^{\frac{\lambda_{1}p_{2}-1}{\lambda_{1}p_{2}}}$. For $k\geq 3$ we have
\begin{eqalign}
\Proba{b_{k}\leq Q_{n}(x)<b_{k+1}}\leq\Proba{\eta_{n}(b_{k})\leq x}\leq c x^{p_{2}}  \para{rk^{\lambda_{1}} }^{-p_{2}}=:R_{x,r}k^{-\lambda_{1}p_{2}}=:R_{k},
\end{eqalign}
where $R_{x,r}:=c (x/r)^{p_{2}}$. So because of their correlation we will bound by the minimum
\begin{eqalign}
\Proba{\infl{T\in [b_{k},b_{k+1}] }\eta_{n+m}\para{T-r,T}\leq g_{n,r},b_{k}\leq Q_{n}(x)<b_{k+1}}\leq \minp{B_{n,r}, R_{x,r}k^{-\lambda_{1}p_{2}}}.
\end{eqalign}
For $R_{k}\leq B_{n,r}\doncl k\geq N_{n,x}:= \maxp{\para{R_{x,r}  B_{n,r}^{-1}}^{1/(\lambda_{1}p_{2})},1}$ we have the sum
\begin{eqalign}
R_{x,r}\sum_{k\geq N_{n,x}}k^{-\lambda_{1}p_{2}}\leq R_{x,r} N_{n,x}^{1-\lambda_{1}p_{2}},
\end{eqalign}
where $\lambda_{1}p_{2}>1$. For  $R_{k}\geq B_{n,r}\doncl k\leq N_{n,x}$,  we are left with $B_{n,r}N_{n,x}$. All together
\begin{eqalign}
&\sum_{k\geq 1}
\Proba{\infl{T\in [b_{k},b_{k+1}] }\eta_{n+m}\para{T-r,T}\leq g_{n,r},b_{k}\leq Q_{n}(x)<b_{k+1}}\\
&\leq  R_{x,r} N_{n,x}^{1-\lambda_{1}p_{2}}+B_{n,r}N_{n,x}\\
&\leq \minp{B_{n,r}^{1-1/\lambda_{1}p_{2}} \para{x/r}^{\frac{1-\lambda_{1}p_{2}}{\lambda_{1}}},(x/r)^{p_{2}}}+B_{n,r}\maxp{\para{R_{x,r}  B_{n,r}^{-1}}^{1/(\lambda_{1}p_{2})},1}\\
&\leq B_{n,r}^{\frac{\lambda_{1}p_{2}-1}{\lambda_{1}p_{2}}}\para{\minp{\para{x/r}^{\frac{1-\lambda_{1}p_{2}}{\lambda_{1}}},B_{n,r}^{1/\lambda_{1}p_{2}-1} (x/r)^{p_{2}}}+\maxp{\para{x/r}^{1/\lambda_{1}},B_{n,r}^{1/\lambda_{1}p_{2}}}  }.
\end{eqalign}
Here we see that for fixed $x\in [\e,M]$, as $n\to +\infty$, it goes to zero overall. Whereas if we fix $n$ and then take $x\to 0$, due to the maximum it does not to zero but it is uniformly bounded.
\pparagraph{Case $r\in [0,1/2)$}
Here we need to decompose the range $Q_{n}(x)\in [r,1-r]\cup [1-r,\infty)$ \begin{eqalign}
&\Proba{\infl{T\in [r,1-r] }\eta_{n+m}\spara{T-r,T}\leq g_{n,r},r\leq Q_{n}(x)<1-r}\\
+&\sum_{k\geq 1}
\Proba{\infl{T\in [b_{k},b_{k+1}] }\eta_{n+m}\spara{T-r,T}\leq g_{n,r}, b_{k}\leq Q_{n}(x)<b_{k+1}}.
\end{eqalign}
 For the first term we use the infimum estimate \Cref{prop:minmodeta}
\begin{eqalign}
&\Proba{\infl{T\in [r,1-r] }\eta_{n+m}\spara{T-r,T}\leq g_{n,r}}    \leq g_{n,r}^{p_{3}}  r^{\zeta(-p_{3})-1}.
\end{eqalign}
Next we study the second term above. We set $b_{k}:=(1-r)k^{\lambda_{2}}\geq r$ for some $\lambda_{2}\in (0,1)$ and $k\geq 1$ and so $b_{k+1}-b_{k}\leq \frac{(1-r)}{k^{1-\lambda_{2}}}\leq 1-r$. For the first event we again use translation invariance by $r$ and apply \Cref{prop:minmodeta} for any $\wt{p}_{1}> 0$
\begin{eqalign}\label{eq:infimumboundbacks}
g_{n,r}^{\wt{p}_{1}}\Expe{\para{\infl{T\in[0,1-r] }\etamu{T,T+r}{1}}^{-\wt{p}_{1}}}\lessapprox~g_{n,r}^{\wt{p}_{1}}r^{\zeta(-\wt{p}_{1})-1}=:\wt{B}_{n,r}.
\end{eqalign}
Similarly
\begin{eqalign}
\Proba{b_{k}\leq Q_{n}(x)<b_{k+1}}\leq\Proba{\eta_{n}(b_{k})\leq x}\leq c x^{\wt{p}_{2}}  \para{(1-r)k^{\lambda_{2}} }^{-\wt{p}_{2}}=:\wt{R}_{x,r}k^{-\lambda_{2}\wt{p}_{2}}=:\wt{R}_{k},
\end{eqalign}
where again $\lambda_{2}\wt{p}_{2}>1$. We again bound by the minimum $\minp{\wt{B}_{n,r},\wt{R}_{k}}$ and so we are left with
\begin{eqalign}
&\sum_{k\geq 1}
\Proba{\infl{T\in [b_{k},b_{k+1}] }\eta_{n+m}\para{T-r,T}\leq g_{n,r},b_{k}\leq Q_{n}(x)<b_{k+1}}\\
\leq &   \wt{R}_{x,r}\wt{N}_{n,x}^{1-\lambda_{2}\wt{p}_{2}}+\wt{B}_{n,r}\wt{N}_{n,x}\\
&\leq \wt{ B}_{n,r}^{\frac{\lambda_{2}\wt{p}_{2}-1}{\lambda_2\wt{p}_{2}}}\minp{\para{x/(1-r)}^{\frac{1-\lambda_{2}\wt{p}_{2}}{\lambda_{2}}},\wt{ B}_{n,r}^{1/\lambda_{2}\wt{p}_{2}-1} (x/(1-r))^{\wt{p}_{2}}}\\
&+\wt{ B}_{n,r}^{\frac{\lambda_{2}\wt{p}_{2}-1}{\lambda_2\wt{p}_{2}}}\maxp{\para{x/(1-r)}^{1/\lambda_{2}},\wt{ B}_{n,r}^{1/\lambda_{2}\wt{p}_{2}}}  .
\end{eqalign}
where $\wt{N}_{n,x}:= \maxp{\para{c (x/(1-r))^{\wt{p}_{2}}  \wt{B}_{n,r}^{-1}}^{1/(\lambda_{2}\wt{p}_{2})}, 1}$.
\end{proof}

\part{Further directions and Appendix}
\section{Further directions }\label{furtherresearchdirections}
\begin{enumerate}
    \item \textbf{Constraint on $\gamma$.} One seemingly major research problem is extending the constraints $\gamma<1$ and $\gamma\in (0,\sqrt{\sqrt{11}-3})\approx (0,0.77817...)$. It is unclear how to remove these given the current techniques. One source of those constraints are in the maximum modulus estimates where we get singular factors. So it seems that one would have  to avoid the modulus techniques all together.
    \item \textbf{Dependence on shift:} For the tail and moments of the shifted GMC $\eta(Q_{a},Q_{a}+L)$ there is still a question of whether there will be some bound that is uniform in $a>0$ and so possibly even be able to take limit $a\to +\infty$. Some counter evidence to that is \Cref{differencetermunshifted}
\begin{eqalign}
\Expe{\eta^{\delta}(Q^{\delta}(a),Q^{\delta}(a)+r)}-r&=\Expe{Q^{\delta}(a)}-a\\
&=\int_{0}^{\infty}\Proba{ Q_{R(t)}(a)\leq t \leq  Q(a)}\dt\\
&=\int_{0}^{\infty}\Proba{ \eta(t)\leq a \leq  \eta_{R(t)}(t) }\dt>0,
\end{eqalign}
where we have $a$-dependence on the right but it is unclear whether it grows as $a$ grows. So here we would need that term to be uniformly bounded in $a>0$. One possible route is by studying the small-ball expansions of
\begin{eqalign}
\Proba{\eta(t) \leq a}\tand \Proba{\eta_{R(t)}(t) \leq a}
\end{eqalign}
as $t\to +\infty$, analogously to the tail expansion results \citep{rhodesvargas2019tail,wong2020universal}.
An alternative route in the spirit of \Cref{ergodiclemma}, is to study the tail of the above integral
\begin{eqalign}
\int_{M}^{\infty}\Proba{\eta(t) \leq a\leq  \eta_{R(t)}(t)}\dt
\end{eqalign}
and to use $\eta^{1}(t)\eqdis t\eta^{1/t}(1)$. In fact, heuristically $a\to+\infty$ we have
\begin{eqalign}
\int_{M}^{\infty}\Proba{\eta^{1/t}(1) \leq \frac{a}{t}\leq  \eta_{R}^{1/t}(1)}\dt&\approx      \int_{M}^{\infty}\Proba{1 \leq \frac{a}{t}\leq  1+c}\dt\\
&=\int_{\frac{a}{1+c}}^{a}\dt=a(1-\frac{1}{1+c})\to +\infty ,
\end{eqalign}
where $c:=\int_{0}^{\delta}\frac{1}{s^{\gamma^{2}}}\ds-\delta$.

\item \textbf{Correlated Brownian motions:} This is an attempt to study the distribution of the randomly-shifted GMC
\begin{eqalign}
\Proba{\eta(Q(a),Q(a)+L)\geq t}.    
\end{eqalign}
We start with approximating $Q_{a}$. We will work with a discrete approximation $T_{n}\downarrow Q(a)$:
\begin{eqalign}
T_{n}:=\frac{m+1}{2^{n}}\twhen \frac{m}{2^{n}}\leq Q(a) < \frac{m+1}{2^{n}}
\end{eqalign}
and so for fixed $n$ the range of $T_{n}$ is $D_{n}(0,\infty)$, all the dyadics of nth-scale in the open $(0,\infty)$. So we  approximate $\Proba{\eta(Q(a),Q(a)+L)\geq t}$ by:
\begin{eqalign}\label{jointevent}
\Proba{\eta(T_n,T_n+L)\geq t}&=\sum_{\ell\in D_{n}(0,\infty)}\Proba{\eta(\ell,\ell+L)\geq t, T_{n}=\ell}\\
&=\sum_{\ell\in D_{n}(0,\infty)}\Proba{\eta(\ell,\ell+L)\geq t, \eta(0,\ell-2^{-n})\leq a\leq \eta(0,\ell)}.
\end{eqalign}
By expanding these GMC integrals
\begin{eqalign}
\eta(\ell,\ell+L)=\int_{[0,L]}e^{U(\ell+s)}\ds \tand  \eta(0,\ell)=\int_{0}^{\ell}e^{U(\ell-s)}\ds ,
\end{eqalign}
we see that they approach the field $U(\ell)$ from the right and left respectively. So naturally we consider the following two corresponding time-changed Brownian motions
\begin{eqalign}
B^{1}_{t}:=U(\ell+e^{-t})\cap     U(\ell)\tand B^{2}_{t}:=U(\ell-e^{-t})\cap U(\ell)
\end{eqalign}
and write
\begin{eqalign}
&\eta(\ell,\ell+L)=\int_{[0,L]}e^{B_{\ln(1/ s)}^{1}-\frac{\gamma^{2}}{2}\ln(1/ s)+U(\ell+s)\setminus U(\ell)}\ds\\
\tand & \eta(0,\ell)=\int_{0}^{\ell}e^{B_{\ln(1/ s)}^{2}-\frac{\gamma^{2}}{2}\ln(1/ s)+U(\ell-s)\setminus U(\ell)}\ds.
\end{eqalign}
From here the problem becomes about studying those two correlated Brownian motions
\begin{eqalign}
\tfor i=1,2,\Expe{B^{i}_{t}B^{i}_{w}}&=\branchmat{0&, \min(t,w)\leq \ln\frac{1}{\delta}\\
\min(t,w)+\ln \delta&, \min(t,w)\geq \ln\frac{1}{\delta}},  \\   \tand\Expe{B^{1}_{t}B^{2}_{w}}&=\branchmat{0&, e^{-t}+e^{-w}\geq \delta\\
\ln \dfrac{\delta}{e^{-t}+e^{-w}}&, e^{-t}+e^{-w}\leq \delta}.
\end{eqalign}
Returning to the joint event \cref{jointevent} we first upper bound by the supremum of $B^{1}$
\begin{eqalign}\label{eq:supremumapprodriftBM}
\Proba{\expo{\sup_{s\in [0,L]}B_{\ln(1/ s)}^{1}-\frac{\gamma^{2}}{2}\ln(1/ s)} \geq \frac{t}{c_{L}},
\eta_{2}(\ell-2^{-n})\leq a \leq \eta_{2}(\ell)},
\end{eqalign}
where $c_{L}:=\int_{[0,L]}e^{U(\ell+s)\setminus U(\ell)}\ds$ is independent of everything else and thus deterministic by conditioning it and $\eta_{2}(x):= \int_{0}^{x}e^{B_{\ln(1/ s)}^{2}-\frac{\gamma^{2}}{2}\ln(1/ s)+U(\ell-s)\setminus U(\ell)}\ds$. So here we are studying the joint law of the supremum $M_{1,\infty}:=sup_{s\in [0,L]}\para{B_{\ln(1/ s)}^{1}-\frac{\gamma^{2}}{2}\ln(1/ s)}$ and of the path trajectory  $\set{B_{\ln(1/ s)}^{2}-\frac{\gamma^{2}}{2}\ln(1/ s)}_{s>0}$. A similar joint event in the case of $B^{1}=B^{2}$ has been studied in the William's decomposition \citep{williams1974path}, \cite[ch.VI 3.13]{revuz2013continuous} (for a version for \Levy~  processes see \citep{duquesne2003path}). This could possibly go through understanding the joint law of the two suprema $(M_{1,\infty},M_{2,\infty})$ for $M_{i,\infty}:=sup_{s\in [0,L]}\para{B_{\ln(1/ s)}^{i}-\frac{\gamma^{2}}{2}\ln(1/ s)}$, which even in the constant correlation case is tricky (\citep{debicki2020exact}). Also since for a Brownian motion with negative drift $a=\beta$, the tail of the maximum satisfies
\begin{eqalign}
\Proba{\sup_{t\geq 0}B_{t}-\beta t\geq u}=e^{-\frac{u\beta^{2}}{2\gamma^{2}}}=e^{-\frac{u}{4}},    
\end{eqalign}
we see in \cref{eq:supremumapprodriftBM}, that if somehow we get around the $\ell$-event and we take moments
\begin{eqalign}
    t^{-p}\Expe{\expo{p\sup_{s\in [0,L]}B_{\ln(1/ s)}^{1}-\frac{\gamma^{2}}{2}\ln(1/ s)} }\Expe{c_{L}^{p}},
\end{eqalign}
then the p-moment of the maximum is finite only for $p\in (0,\frac{1}{4})$.

\item \textbf{Cauchy uniformity in $x\in [0,1]$} The result shown in \cref{rateconv} is done for each fixed $x$. It will be interesting to further do a diagonalization argument to get uniformity in the Cauchy-sequence convergence.

\item \textbf{Multivariable Kahane} This is an attempt to study shifted GMC moments for $q\geq 1$
        \begin{equation*}
        \Expe{(\eta(Q_a,Q_a+x))^{q}}.  
        \end{equation*}
Here we are in the convex case. The problem with applying scaling law is that $Q_{a}$ is dependent on the entire history $U(0,Q_{a})$ in a nonlinear way. If we study the supremum 
\begin{equation*}
        \Expe{(\sup_{0\leq T\leq L}\eta(T,T+x))^{q}}, 
        \end{equation*}
then it will be interesting if there is an extension of Kahane inequality. Following \cite{Biskup2017},for two Gaussian fields $X,Y$ and $q=1$, we attempted using the softmax $f(x_1,...,x_{n})=\frac{1}{R}\ln\para{\sum_{i=1}^{n}e^{Rx_{i}}}$ approximation and the regular Kahane-interpolation proof. But we ended up with
\begin{eqalign}\label{eq:Kahanediff}
&\Exp\sum_{k=1}^{n}p_{k}\int_{T_{k}}^{T_{k}+x}\int_{T_{k}}^{T_{k}+x}(C_{X}(t,s)-C_{Y}(t,s))\deta_{Z_{\lambda}}(s)\deta_{Z_{\lambda}}(t)\\    
-&\sum_{k,\ell=1}^{n}p_{k}p_{\ell}\int_{T_{\ell}}^{T_{\ell}+x}\int_{T_{k}}^{T_{k}+x}(C_{X}(t,s)-C_{Y}(t,s))\deta_{Z_{\lambda}}(s)\deta_{Z_{\lambda}}(t),  
\end{eqalign}
for $p_{k}:=\frac{e^{R\eta_{k}}}{\sum_{i=1}^{n}e^{R\eta_{i}}}$ and $\eta_{i}:=\eta(T_{i},T_{i}+x)$, some generic sequence $T_{i}\in [0,L]$ and $Z_{\lambda}:=\lambda X+\sqrt{1-\lambda^{2}}Y$. The nonnegative of this difference \cref{eq:Kahanediff} might follow if we impose more conditions on the covariances than just $C_{X}(t,s)-C_{Y}(t,s)\geq 0$. For example, if we have $C_{X}(t,s)-C_{Y}(t,s)=f(t)f(s)$, then \cref{eq:Kahanediff} will be nonnegative by the Cauchy-Schwartz inequality. But this condition is not possible in our particular case where we deal with differences $\ln\frac{1}{\abs{t-s}}$.

\item  \textbf{Kahane-Malliavin:} This is another attempt to study shifted GMC moments for $q\geq 1$
        \begin{equation*}
        \Expe{(\eta(Q_a,Q_a+x))^{q}}  
        \end{equation*}
        in the spirit of convexity inequalities like Kahane-Malliavin interpolation \citep{nourdin2014comparison} and integration by parts formula \citep{decreusefond2008hitting}. We attempted to take Malliavin derivatives of shifted GMC by first mollifying as done in \citep{decreusefond2008hitting} (this is the content of the work \citep{kojar2023inverse}). The fundamental issue is that hitting time is not Malliavin differentiable. Also, as observed in \citep{cui2017first} in the case of Brownian motion, the law of $U(Q_{a}+s)$ is far from being Gaussian. So it is unclear what type of calculus to use here.

    \item\textbf{Good-$\lambda$ inequality} This is yet another attempt to study the shifted GMC.     The good$-\lambda$ inequality method (\cite[chapter IV 4.8 def]{revuz2013continuous}) might also be relevant here.  By moderate function $F$ we mean the increasing, continuous function vanishing at $0$ and such that
\begin{eqalign}
\sup_{x> 0}  \frac{F(\alpha x)}{F(x)}=\gamma<\infty \tforsome \alpha>1.
\end{eqalign}
In particular $F(x)=x^{p}$ for $0<p<\infty$ are moderate functions.
\begin{proposition}
If the positive variables $(X,Y)$ satisfy for all $\lambda>0, \delta\in (0,a]$ ,for some fixed $a>0$,
\begin{eqalign}
\Proba{X\geq \beta\lambda, Y\leq \delta \lambda}\leq \phi(\delta)    \Proba{X\geq \lambda}
\end{eqalign}
for some fixed $\beta>1$ and function $\phi:(0,a]\to [0,\infty)$ decaying $\liz{\delta}\phi(\delta)=0$, then for moderate functions F we have the inequality
\begin{eqalign}
\Expe{F(X)}\leq c \Expe{F(Y)}
\end{eqalign}
for some constant $c$ depending only on $\phi,\beta\tand \gamma$.
\end{proposition}

So one question is whether we can prove a good lambda inequality for $X=\eta(Q(a),Q(a)+L)$ and $Y=\int_{Q(a)}^{Q(a)+L}e^{U(Q(a)+s)\setminus U(Q(a)) }\ds$:
\begin{eqalign}
&\Proba{\eta(Q(a),Q(a)+L)\geq \beta\lambda,\int_{Q(a)}^{Q(a)+L}e^{U(Q(a)+s)\setminus U(Q(a)) }\ds\leq \delta \lambda}\\
&\leq \phi(\delta)    \Proba{\eta(Q(a),Q(a)+L)\geq \lambda}.
\end{eqalign}

    \item \textbf{ Density for the inverse $Q(x)$ }\\
    The density of the total mass of the GMC was  conjectured in the unit circle in \citep{fyodorov2008freezing} and unit interval in \citep{ostrovsky2009mellin,fyodorov2009statistical} and computed rigorously in \citep{remy2017fyodorov,chhaibi2019circle} and \citep{remy2020distribution}, respectively. This was done by proving certain relations between fusion/singular moments and in turn that they satisfy certain differential equations akin to the Belavin-Polyakov-Zamolodchikov (BPZ) differential equations. For obtaining the density of the inverse, one can use the relation
    \begin{eqalign}
    \Proba{Q(x)\geq t}=\Proba{x\geq \eta(t)}=\int_{0}^{x}\rho(t,y)\dy,   \end{eqalign}
    where $\rho(t,y)$ is the density for $\eta(t)$. So if we have that density, we can also obtain the one for the inverse. More directly, it could be that the inverse $Q$ of the singular integrals also satisfies some analogous recurrence relations and differential equations.

    \item \textbf{Inverse-Liouville Brownian motion}\\
    Let $(B_{t})_{t\geq 0}$ be a 1d Brownian motion in interval $D\subset \T$ with $B_{0}=z$ that is independent of GMC measure $\eta$, then the following process is called the 1d-Liouville Brownian motion (LBM)
\begin{eqalign}
Z_{\e,t}:=B_{\eta_{\e}^{-1}(t)},
\end{eqalign}
with
\begin{eqalign}
\eta_{\e}(t):=\int^{t\wedge T_{r,D}}\expo{\gamma U_{\e}(B_{s})-\frac{\gamma^{2}}{2}\ln(1/\e) } ds,
\end{eqalign}
and $T_{r,D}:=\inf\{s>0: dist(B_{s},\partial D)\leq r\}$. Using the techniques from \citep{jackson2018liouville,berestycki2015diffusion} one can similarly show that the process $Z_{\e,t}\codis Z$ is continuous and that LBM lives on the thick points. (In fact, local time is continuous in 1d and so the proofs simplify considerably.) So naturally one can study the (LBM) parametrized by the inverse map
\begin{eqalign}
Z_{\e,t}^{Q}
:=B_{Q_{\e}^{-1}(t)}=B_{\eta_{\e}(t)},
\end{eqalign}
where $Q_{\e}$ is the inverse of $\eta_{\e}$. The continuity $Z_{\e,t}^{Q}\codis Z^{Q}$ here follows from that of $\eta_{\e}(t)$. A similar result of LBM living on the thick points $T_{\frac{1}{\gamma}}$ might still follow.

\item \textbf{Inverse function theorem (IFT)}\\
It might be interesting to carefully study the inverse function theorem in the context of the singular function $\eta(x)$. Since the lower truncated measure $\eta_{\e}[0,x]$ has a continuous density, it is natural to apply the inverse function theorem and find the following recursion formula
\begin{eqalign}
Q_{\e}(I)=\int_{I}\e^{-\frac{\gamma^{2}}{2}}\expo{-U_{\e}(Q_{\e}(\theta))  }\dtheta.
\end{eqalign}
Because the $\e$ shows up in a singular way, this most likely doesn't converge. If it does a have limit, it will be at the location where $U_{\e}(Q_{\e}(\theta))\to +\infty$ in order to counter the blowing up regularization. And since $Q_{\e}(\theta)$ are likely contained in the support of the GMC measure and hence the thick-points $T_{\gamma}$ (i.e. $x\in [0,1]$ such that $\liminf\limits_{\e\to 0}\frac{U_{\e}(x)}{\ln(\frac{1}{\e})}=\gamma$), the continuous recursion formula seems reasonable. By differentiating the recursion formula from IFT we have the following autonomous ordinary differential equation:
\begin{eqalign}
    \frac{\dint Q_{\e}(t)}{\dint t} = \e^{-\frac{\gamma^{2}}{2}}\expo{-U_{\e}(Q_{\e}(t))  }  \twith Q_{\e}(0)=0.
\end{eqalign}
For very small $\e>0$, the equilibrium points satisfy $\frac{\dint Q_{\e}(t_{*})}{\dint t} = 0$ and so they require that $Q_{\e}(t_{*})$ should be approximating the thick points $T_{\gamma}$ as $\e\to 0$.

    \item \textbf{Martingales inside the GMC measure}\\
    Since we are studying the hitting time $Q$, we naturally searched for any martingale structure in GMC as a function of position.  As described in the proof of \cite[Lemma 5]{bacry2003log}, there are martingales hiding inside the field $U(s)$ for $s\in [0,t]$ for small enough $t$ i.e. we can decompose
    \begin{eqalign}
    U(s)=U_{l}(s)+U_{i}+U_{r}(s),
    \end{eqalign}
    where $U_{i}=\bigcap_{s\in [0,t]}U(s)\neq \varnothing$ (which is indeed non-empty for small enough $t$) and $U_{l}$ and $U_{r}$ are the remaining left and right fields. These are martingales in time i.e. $\Expe{U_{l}(s-\delta)|\mathcal{F}_{l}(s)}=U_{l}(s)$ and  $\Expe{U_{r}(s+\delta)|\mathcal{F}_{r}(s)}=U_{r}(s)$. Therefore, the GMC measures $\eta_{l}[x]=\int_{0}^{x} e^{\overline{U_{l}}(s)}\ds$ and $\eta_{r}=\int_{0}^{x} e^{\overline{U_{r}}(s)}\ds$ for $x<t$ will be submartingales. Therefore, they will satisfy many of the known theorems about stopping times (such as optional stopping theorem). We have not found any connection between the hitting time for the full measure $\eta$ and those of $\eta_{l},\eta_{r}$ because we have not found any way to express $\eta$ in terms of them or at least bound by since their densities are distributions (so there might not be any relation or inequality between them). Possibly by conditioning on the fiels $U_{i},U_{r}$ but leaving $U_{l}$ random, we can study the stopping times of this measure and try to relate it to the original one. Similarly when we use the decomposition $U_{s}=B_{s}+Y_{s}$, where $B_{s}$ is Brownian motion and $Y_{s}$ is an independent distributional logarithmic field, as studied in \citep{kupiainen2020integrability}, we can condition on $Y_{s}$ and try to express the stopping time of $\eta$ in terms of the integrated Geometric Brownian motion (IGBM) but with a base measure which is fractal (IGBM with Lebesgue measure and its hitting times have been studied already see here \citep{cui2017first}).  The field $U(s)\cap U(t)$ shows up naturally when studying the GMC-measure $\eta(0,t)$ as a function of time.
\begin{lemma}
Let $X_{t}:=\eta(t)-t$ and condition on the filtration $\mc{F}_{t}=\sigma(\set{U_{s}}_{0\leq s\leq t})$ to find
\begin{eqalign}
\Expe{\int_{0}^{t+r}e^{U(s)}\ds -(t+r)| \mc{F}_{t}}=&\int_{0}^{t}e^{U(s)}\ds -t+\int_{t}^{t+\delta\wedge r}e^{U(s)\cap U(t)}-1\ds\\
=:&\eta(t)-t+R_{t,r}.
\end{eqalign}
So we see that $X_{t}$ is away from having any sub/supermartingale property precisely due to the fluctuations of $e^{U(s)\cap U(t)}-1$.
\end{lemma}
\begin{proof}
We split the integral conditional expectation above into the three pieces
\begin{eqalign}
&\Expe{\int_{\min(t+\delta,t+r)}^{t+r}e^{U(s)}\ds -r| \mc{F}_{t}}+\Expe{\int_{t}^{\min(t+\delta,t+r)}e^{U(s)}\ds | \mc{F}_{t}}\\
&+\Expe{\int_{0}^{t}e^{U(s)}\ds -t| \mc{F}_{t}}.
\end{eqalign}
The first term is independent of $\mc{F}_{t}$ and the third term is completely determined by $\mc{F}_{t}$
\begin{eqalign}
&(t+r-(t+\delta\wedge r)) -r+\Expe{\int_{t}^{t+\delta\wedge r}e^{U(s)}\ds | \mc{F}_{t}}+\int_{0}^{t}e^{U(s)}\ds -t.
\end{eqalign}
So it remains to study the middle term. We split the field $U(s)=U(s)\cap U(t)+U(s)\setminus U(t)$, where the first term is completely determined by $\mc{F}_{t}$ and the second term is independent of it. Therefore, since $\Expe{e^{U_{\e}(s)\setminus U_{\e}(t)}}=1 $ we indeed find
\begin{eqalign}
\Expe{\int_{0}^{t+r}e^{U(s)}\ds -(t+r)| \mc{F}_{t}}=\int_{0}^{t}e^{U(s)}\ds -t+\int_{t}^{t+\delta\wedge r}e^{U(s)\cap U(t)}-1\ds.
\end{eqalign}
\end{proof}

\end{enumerate}

\begin{appendices}

\section{Gaussian fields }

\subsection{Gaussian field estimates}
\pparagraph{Concentration inequalities}We will need the supremum estimate for Gaussian fields called Borel-Tsirelson-Ibragimov-Sudakov (BTIS) inequality to get the upper bound.
\begin{theorem}\label{suprconcen}\cite[theorem 4.1.1]{adler2009random} 
Let $S=\bigcup_{k}I_{k}$ and suppose that the Gaussian field $Y_{t},t\in S$ has $L$-Lipschitz continuous covariance $\Expe{\abs{Y_{t}-Y_{s}}^{2}}\leq L\abs{s-t}$ and $Y_{t_{0}}=0$ with $t_{0}\in S$. Then
\begin{eqalign}
\Proba{\supl{t\in S}Y_{t}\geq \sqrt{L\abs{S}}u  }\leq c(1+u)e^{-u^{2}/2},
\end{eqalign}
for some uniform constant $c$.
\end{theorem}
 We also have bounds for the expected value of the supremum \cite[Theorem 6.5]{van2014probability}. Consider centered Gaussian field $\para{X_{t}}_{t\in T}$ in some index set $T$. Let
\begin{eqalign}
d(s,t):=\sqrt{\Expe{\abs{X_{s}-X_{t}}^{2}}}
\end{eqalign}
denote a pseudo-metric on $T$ and the \textit{entropy number} $N(T, d, \e)$ equal the minimum number of balls of size $\e>0$ needed to cover $T$ in the $d$-metric. We have the following bounds on the expected supremum and its tail.
\begin{theorem}\label{thm:entropyintegralest}(Entropy integral)\cite[Corollary 5.25]{van2014probability}For a separable subGaussian process $\para{X_{t}}_{t\in T}$ on the metric space $(T,d)$, we have the following estimate
\begin{eqalign}
\Expe{\sup_{t}X_t}\leq 12\int_{0}^{\diam(T)}\sqrt{\ln N(T, d, \e)}\dint\epsilon. \end{eqalign}
\end{theorem}
\begin{theorem}(Local chaining inequality)\cite[Proposition 5.35]{van2014probability}For a separable subGaussian process $\para{X_{t}}_{t\in T}$ on the metric space $(T,d)$, we have the following estimate for each fixed $t_{0}\in T$
\begin{eqalign}
\Proba{\supls{t\in T\\ d(s,t)\leq r}\para{X_{t}-X_{t_{0}}}\geq  C\int_{0}^{r}\sqrt{\ln N(T, d, \e)}\dint\epsilon+x}\leq C\expo{-\frac{x^{2}}{Cr^{2}}},
\end{eqalign}
for some universal constant $C$.
\end{theorem}
\begin{theorem}\label{them:BTIS}(BTIS)\cite[Theorem 2.3.1]{adler2009random}
Let $X_{t}$ be a centered Gaussian process, a.s. bounded on $T$. Let $\norm{X}:=\supls{t\in T}X_{t}$ and $\sigma_{T}^{2}:=\sup_{t\in T}\Expe{X_{t}^{2}}$. Then for all $u>\Expe{\norm{X}}$
\begin{eqalign}
 \Proba{\supls{t\in T}X_{t}\geq u  }\leq 2\expo{-\frac{(u-\Expe{\norm{X}})^{2}}{2\sigma_{T}^{2}}}. 
\end{eqalign}
\end{theorem}
\pparagraph{Comparison inequalities}
 We will use the Kahane's inequality (eg. \cite[Lemma 4]{wong2020universal},\cite[corollary A.2]{robert2010gaussian}).
\begin{theorem}(Kahane Inequality)\label{Kahanesinequality}
Let $\rho$ be a Radon measure on a domain $D\subset\R^{n}$, $X(\cdot)$ and $Y(\cdot)$ be two continuous centred Gaussian fields on $D$, and $F: \Rplus \to \R$ be some smooth function with at most polynomial growth at infinity. For $t \in [0,1]$, define $Z_t(x) = \sqrt{t}X(x) + \sqrt{1-t}Y_t(x)$ and
\begin{eqalign}
\varphi(t) := \EE \left[ F(W_t)\right], \qquad
W_t := \int_D e^{Z_t(x) - \frac{1}{2}\EE[Z_t(x)^2]} \rho(dx).
\end{eqalign}

 Then the derivative of $\varphi$ is given by
\begin{eqalign}\label{eq:Kahane_int}
\begin{split}
\varphi'(t) & = \frac{1}{2} \int_D \int_D \left(\EE[X(x) X(y)] - \EE[Y(x) Y(y)]\right) \\
& \qquad \qquad \times \EE \left[e^{Z_t(x) + Z_t(y) - \frac{1}{2}\EE[Z_t(x)^2] - \frac{1}{2}\EE[Z_t(y)^2]} F''(W_t) \right] \rho(dx) \rho(dy).
\end{split}
\end{eqalign}

 In particular, if
\begin{eqalign}
\EE[X(x) X(y)] \le \EE[Y(x) Y(y)] \qquad \forall x, y \in D,
\end{eqalign}

 then for any convex $F: \Rplus \to \R$
\begin{eqalign}\label{eq:Gcomp}
\EE \left[F\left(\int_D e^{X(x) - \frac{1}{2} \EE[X(x)^2]}\rho(dx)\right)\right]
\le
\EE \left[F\left(\int_D e^{Y(x) - \frac{1}{2} \EE[Y(x)^2]}\rho(dx)\right)\right].
\end{eqalign}

 and the inequality is reversed if $F$ is concave instead. In the case of distributional fields $X,Y$, we use this result for the mollifications $X_{\e},Y_{\e}$ and pass to the limit due to the continuity of $F$.
\end{theorem}
 The interpolation has implications for the max/min of Gaussian fields.
\begin{theorem}\label{th:slepian}(Slepian's lemma)\cite[Theorem 7.2.1]{vershynin2018high}
Let $\set{X_t}_{t\in T}$ and $\set{Y_t}_{t\in T}$ be two mean zero Gaussian processes. Assume that for all $t, s \in T$, in some index set $T$, we have
\begin{eqalign}
\Expe{X_{t}^{2}}=\Expe{Y_{t}^{2}}    \tand \Expe{X_{t}X_{s}}\geq \Expe{Y_{t}Y_{s}}.
\end{eqalign}
Then for every $\tau\in \R$, we have
\begin{eqalign}
\Proba{\sup_{t\in T}X_{t}\geq \tau }\leq \Proba{\sup_{t\in T}Y_{t}\geq \tau }\tand \Proba{\inf_{t\in T}X_{t}\leq -\tau }\leq \Proba{\inf_{t\in T}Y_{t}\leq -\tau }.
\end{eqalign}
\end{theorem}
 We will also consider another comparison inequality for increasing/decreasing functionals which generalizes the Harris-inequality \cite[Theorem 2.15]{boucheron2013concentration} that is called FKG inequality for Gaussian fields \citep{pitt1982positively},\cite[Theorem 3.36]{berestycki2021gaussian}.
\begin{theorem}[FKG inequality]\label{FKGineq}
Let $\set{Z(x)}_{x\in U}$ be an a.s. continuous centred Gaussian field on $U \subset \R^d$ with $\Expe{Z(x)Z(y)} \geq 0$ for all $x, y \in U$. Then, if $f, g$ are two bounded, increasing measurable functions,
\begin{eqalign}
\Expe{f\para{\set{Z(x)}_{x\in U}}g\para{\set{Z(x)}_{x\in U}}}    \geq \Expe{f\para{\set{Z(x)}_{x\in U}}}\Expe{g\para{\set{Z(x)}_{x\in U}}}
\end{eqalign}
and the opposite inequality if one function is increasing and the other one is decreasing
\begin{eqalign}
\Expe{f\para{\set{Z(x)}_{x\in U}}g\para{\set{Z(x)}_{x\in U}}}    \leq \Expe{f\para{\set{Z(x)}_{x\in U}}}\Expe{g\para{\set{Z(x)}_{x\in U}}}.
\end{eqalign}
\end{theorem}
\begin{proof}
The second case just follows by setting $\wt{g}=-g$. Here is a second proof of that second case too. In \citep{pitt1982positively}, they assume increasing for both $f,g$. However, in the proof they obtain the formula
\begin{eqalign}
F'(\lambda)=\frac{1}{\lambda}\int_{R^k}\phi(\bar{x})\para{\sum_{i,j}\sigma_{i,j}\frac{\partial f(\bar{x})}{\partial x_{i}}\sum_{i,j}\sigma_{i,j}\frac{\partial g(\bar{x})}{\partial x_{i}} }    \dint\bar{x}.
\end{eqalign}
So if we have $f$ increasing and $g$ decreasing, we get $F'(\lambda)\leq 0$ and so
\begin{eqalign}
F(1)\leq F(0).
\end{eqalign}
\end{proof}
\section{Moment estimates for GMC }
\subsection{Moments bounds of GMC}
 For positive continuous bounded function $g_{\delta}:\Rplus\to \Rplus$ with $g_{\delta}(x)=0$ for all $x\geq \delta$ and uniform bound $M_{g}:=\sup_{\delta\geq 0}\norm{g_{\delta}}_{\infty}$,  we will need the following moment estimates for $\eta^{\delta}=\eta_{g}^{\delta}$ in order to compute the moments of the inverse $Q_{g}^{\delta}$. As mentioned in the notations we need to study more general fields such as $U^{\delta,\lambda}$ in \Cref{eq:truncatedscaled}
\begin{eqalign}
&\Expe{U_{ \varepsilon}^{\delta,g }(x_{1} )U_{ \varepsilon}^{  \delta,g }(x_{2} )  }\\
=&\left\{\begin{matrix}
\ln(\frac{\delta }{\varepsilon} )-\para{\frac{1 }{\e}-\frac{1}{\delta}}\abs{x_{2}-x_{1}}+g_{\delta}(\abs{x_{2}-x_{1}})&\tifc \abs{x_{2}-x_{1}}\leq \varepsilon\\
 \ln(\frac{\delta}{\abs{x_{2}-x_{1}}})-1+\frac{\abs{x_{2}-x_{1}}}{\delta}+g_{\delta}(\abs{x_{2}-x_{1}}) &\tifc \e\leq \abs{x_{2}-x_{1}}\leq \delta\\
  0&\tifc \delta\leq \abs{x_{2}-x_{1}}
\end{matrix}\right.    .
\end{eqalign}
The following proposition is studying the moment bounds for this truncated case and it is the analogue of \cite[lemma A.1]{david2016liouville} where they study the exact scaling field in 2d.
\begin{proposition}\label{momentseta}
We have the following estimates:
\begin{itemize}
    \item  For $q\in (-\infty,0)\cup [1,\frac{2}{\gamma^{2}})$ we have
\begin{eqalign}
\Expe{\para{\eta^{\delta}[0,  t]}^{q}}\leq c_{1}\branchmat{t^{\zeta(q)}\delta^{\frac{\gamma^{2}}{2}(q^{2}-q)} &t\in [0,\delta]\\ t^{q}& t\in [\delta,\infty] },
\end{eqalign}
where $\zeta(q):=q-\frac{\gamma^{2}}{2}(q^{2}-q)$ is the multifractal exponent and $c_{1}:=e^{\frac{\gamma^{2}}{2}(q^{2}-q)M_{g}}\Expe{\para{\int_{0}^{1}e^{ \overline{\omega}^{1}(x)}\dx}^{q}}$ and $\omega^1$ is the exact-scaling field of height $1$.

\item  When $t,\delta\in [0,1]$ then we have a similar bound with no $\delta$
\begin{eqalign}
\Expe{\para{\eta^{\delta}[0,  t]}^{q}}\leq c_{1} t^{\zeta(q)},
\end{eqalign}
but this exponent is not that sharp: if $\delta\leq t\leq 1$ and $q>1$, we have $q>\zeta(q)$ and so the previous bound is better.

\item Finally, for all $t\in [0,\infty)$ and $q\in [0,1]$, we have
\begin{eqalign}
\Expe{\para{\eta^{\delta}[0,  t]}^{q}}\leq c_{2,\delta } t^{\zeta(q)}\delta^{\frac{\gamma^{2}}{2}(q^{2}-q)},
\end{eqalign}
where $c_{2,\delta}:=e^{\frac{\gamma^{2}}{2}(q^{2}-q)M_{g}}\Expe{\para{\int_{0}^{1}e^{ \overline{\omega}^{\delta}(x)}\dx}^{q}}$ where $\omega^{\delta}$ the exact-scaling field of height $\delta$.
\end{itemize}
Furthermore, for the positive moments in $p\in (0,1)$ of lower truncated GMC $\eta_{n}(0,t):=\int_{0}^{t}e^{U_{n}(s)}\ds$, we have the  bound
\begin{eqalign}\label{lowertrunp01}
\Expe{\para{\eta_{n}(0,t)}^{p}}\lessapprox~t^{\zeta(p)}, \forall t\geq 0.
\end{eqalign}
When $p\in(-\infty,0)\cup (1,\frac{2}{\gamma^{2}})$, we have
\begin{eqalign}
\Expe{\para{\eta_{n}(0,t)}^{p}}\leq \Expe{\para{\eta(0,t)}^{p}} , \forall t\geq 0.
\end{eqalign}
For the field $X(s):=U(s)\setminus U(0)$, $p\in [1,\beta^{-1})$ and $\alpha\leq1$ we have the estimate
\begin{eqalign}\label{eq:singularalphaintegralestimate}
\Expe{\para{\int_{0}^{t}\frac{1}{x^{\alpha}}e^{ X(x) }dx}^{p} }\leq  c   t^{p(1-\alpha)}.
\end{eqalign}

\end{proposition}

\begin{proof}
\proofparagraph{Case $t\in [0,\delta]$ and $q\in (-\infty,0]\cup [1,\frac{2}{\gamma^{2}})$}
 For $q\in (-\infty,0]\cup [1,\frac{2}{\gamma^{2}})$ we have that $F(x)=x^{q}$ is convex and so we need to get upper bounds on the covariance to apply Kahane's inequality \Cref{Kahanesinequality}. By change of variables we have $\eta(  t)=  t\int_{0}^{1}e^{ \overline{U}^{\delta,g}(  tx)}\dx$. The covariance of the field $U^{\delta,g}(  tx)$ covariance can be upper bounded:
\begin{eqalign}
\ln\frac{\delta}{t\abs{x-y}}-1+    \frac{t\abs{x-y}}{\delta}+g_{\delta}\para{t\abs{x-y}}\leq \ln\frac{\delta}{t}+ \ln\frac{1}{\abs{x-y}}+M_{g},
\end{eqalign}
for all $x,y\in [0 ,1]$. When both $\delta, t\leq 1$, then we can further simply ignore the nonpositive term $\ln\delta\leq 0$. Therefore, we apply Kahane's inequality \Cref{Kahanesinequality} for $F(x)=x^{ q}$ and the fields $U^{\delta,g}(tx)$ and $N(0,\ln\frac{\delta}{t}+M_{g} )+\omega^{1}(x)$ to upper bound
\begin{eqalign}\label{eq:tsmallerthandeltaqbiggerthan1ornegative}
  \Expe{\para{\int_{0}^{1}e^{ \overline{U}^{\delta}(tx)}\dx}^{q}}
  &\leq \Expe{\para{e^{\gamma\overline{N}(0,\ln\frac{\delta}{t}+M_{g}) }}^{q}} \Expe{\para{\int_{0}^{1}e^{ \overline{\omega}^{1}(x)}\dx}^{q}}\\&=ce^{\frac{q^{2}\gamma^{2}}{2}\ln\frac{\delta}{t} +\frac{q\gamma^{2}}{2}\ln\frac{\delta}{t} }=c\para{\frac{\delta}{t}}^{\frac{\gamma^{2}}{2}(q^{2}-q)}.
\end{eqalign}
\proofparagraph{Case $t\in [\delta,\infty)$ and $q\in (-\infty,0]\cup [1,\frac{2}{\gamma^{2}})$}
Here for the field $U^{\delta,g}( tx)$ we can upper bound its covariance:
\begin{eqalign}
\para{\ln\frac{\delta}{t\abs{x-y}}-1+    \frac{t\abs{x-y}}{\delta}+g_{\delta}\para{t\abs{x-y}}}\one_{\abs{x-y}\leq \minp{1,\frac{\delta}{t}}}\leq \ln\frac{1}{\abs{x-y}}+M_{g}
\end{eqalign}
for all $\abs{x-y}\leq 1$. So we will apply Kahane's inequality for the fields $U^{ \delta}(tx)$ and $N(0,M_{g})+\omega^{1}(x)$ to upper bound:
\begin{eqalign}
  \Expe{\para{\int_{0}^{1}e^{ \overline{U}^{\delta,g}(tx)}\dx}^{q}}\leq  \Expe{\para{e^{\gamma\overline{N}(0,M_{g}) }}^{q}} \Expe{\para{\int_{0}^{1}e^{ \overline{\omega}^{1}(x)}\dx}^{q}}=c.
\end{eqalign}
\proofparagraph{Case $t\in [0,\delta]$ and $q\in [0,1)$}
For $q\in (0,1)$ we have that $F(x)=x^{q}$ is concave and so we need to get lower bounds on the covariance. For $t\in [0,1]$ the field's $U^{\delta,g}(  tx)$ covariance has a lower bound:
\begin{eqalign}\label{lowercovar}
&\para{\ln\frac{\delta}{t\abs{x-y}}-1+    \frac{t\abs{x-y}}{\delta}+g_{\delta}\para{t\abs{x-y}}}\one_{\abs{x-y}\leq 1}\\
&\geq \ln\frac{\delta}{t}+\para{ \ln\frac{1}{\abs{x-y}}}\one_{\abs{x-y}\leq 1}-1,
\end{eqalign}
for all $x,y\in [0,1]$. Therefore, we apply Kahane's inequality for the fields $U^{ \delta}(tx)+N(0,1)$ and $\wt{N}(0,\ln\frac{\delta}{t})+\omega^{1}(x)$, where all four are independent, to upper bound
\begin{eqalign}
  \Expe{\para{\int_{0}^{1}e^{ \overline{U}^{ \delta}(tx)}\dx}^{q}}\leq &\frac{\Expe{\para{e^{\gamma\overline{\wt{N}}(0,\ln\frac{\delta}{t}) }}^{q}}}{\Expe{\para{e^{\gamma\overline{N}(0,1 }}^{q}}} \Expe{\para{\int_{0}^{1}e^{ \overline{\omega}^{1}(x)}\dx}^{q}}\\
  =&\frac{c}{e^{\frac{\gamma^{2}}{2}(q^{2}-q)}}\para{\frac{\delta}{t}}^{\frac{\gamma^{2}}{2}(q^{2}-q)}.
\end{eqalign}
\proofparagraph{Case $t\in [\delta,\infty)$ and $q\in [0,1)$}
 Here we use the same covariance lower bound \Cref{lowercovar} but because of $\ln \frac{\delta}{t}\leq 0$ we instead apply Kahane's inequality to the fields $U^{ \delta}(tx)+N(0,\ln\frac{t}{\delta}+1)$ and $\omega^{1}(x)$ to upper bound:
\begin{eqalign}\label{largetKahane}
  \Expe{F\para{\int_{0}^{1}e^{ \overline{U}^{ \delta}(tx)}\dx}}\leq&   \frac{1}{\Expe{F\para{e^{\gamma\overline{N}(0,\ln \frac{t}{\delta}+1) }}}} \Expe{F\para{\int_{0}^{1}e^{ \overline{\omega}^{1}(x)}\dx}}\\
  =&\frac{c}{e^{\frac{\gamma^{2}}{2}(q^{2}-q)}}\para{\frac{\delta}{t}}^{\frac{\gamma^{2}}{2}(q^{2}-q)}.
\end{eqalign}
\proofparagraph{Case $t\in [0,\infty)$ and $q\in [0,1)$:lower truncated $\eta_{\e}$}
As above we first do change for variables and study the field $U_{\e}^{\delta}(tx)$. The covariance for the truncated $U_{\e}^{\delta}(tx)$ is
\begin{eqalign}
\Expe{U_{ \varepsilon}^{  \delta }(tx_{1} )U_{ \varepsilon}^{  \delta }(tx_{2} )  }=\left\{\begin{matrix}
\ln(\frac{\delta }{\varepsilon} )-\para{\frac{1}{\e}-\frac{1}{\delta}}t\abs{x_{2}-x_{1}} &\tifc \abs{x_{2}-x_{1}}\leq \frac{\varepsilon}{t}\\
 \ln(\frac{\delta}{t\abs{x_{2}-x_{1}}}) +\frac{t\abs{x_{2}-x_{1}}}{\delta}-1&\tifc \frac{\varepsilon}{t}\leq \abs{x_{2}-x_{1}}\leq \frac{\delta}{t}\\
 0&\tifc  \frac{\delta}{t}\leq \abs{x_{2}-x_{1}}
\end{matrix}\right.    .
\end{eqalign}
Here we use a similar lower bound in $\e\leq t\abs{x-y}\leq \delta$ phase:
\begin{eqalign}
\ln\frac{\delta}{t\abs{x-y}}-1+    \frac{t\abs{x-y}}{\delta}\geq \ln\frac{\delta}{t\abs{x-y}}-1,
\end{eqalign}
and below it, we use that $\frac{t\abs{x-y}}{\e}\leq 1$ and so again
\begin{eqalign}
\ln\frac{\delta}{\e}  -\para{\frac{1}{\e}-\frac{1}{\delta}}t\abs{x_{2}-x_{1}}\geq   \ln\frac{\delta}{\e} -1.
\end{eqalign}
So we use the fields $N(0,1)+U^{\delta}_{\e}(tx)$ and $\omega^{\delta}_{\e}(tx)$. Then when $t\leq \delta$, we are in the position to apply the exact scaling law for $\omega$ to get the bound $t^{\zeta(p)}$. And when $t\geq \delta$, then as in \Cref{largetKahane} we similarly apply Kahane to the fields $\omega^{ \delta}(tx)+N(0,\ln\frac{t}{\delta}+1)$ and $\omega^{1}(x)$, to get again $t^{\zeta(p)}$.

\proofparagraph{Case $t\in [0,\infty)$ and $q\in(-\infty,0)\cup (1,\frac{2}{\gamma^{2}})$ : lower truncated $\eta_{\e}$}
When $q\in(-\infty,0)\cup (1,\frac{2}{\gamma^{2}})$, we can apply the logic of \cite[Lemma 6.]{bacry2003log}, which is basically using that $X_{n}:=\eta_{\e}$ is a martingale and so its convex function is a submartingale and thus bound by the GMC's moments:
\begin{eqalign}
\Expe{\para{\eta_{\e}(0,t)}^{q}}\leq \Expe{\para{\eta(0,t)}^{q}}.
\end{eqalign}

\proofparagraph{Case $t\in [0,\delta]$ and $p\in [1,\frac{2}{\gamma^{2}})$: singular/fusion $\eta$}
We only study $X(s):=U(s)\setminus U(0)$ since it is similar. For the covariance we have for $\delta\geq t-s>0$
\begin{eqalign}
\Expe{X(t)X(s)}=&\Expe{\para{U(t)\setminus U(0)} \cdot\para{  U(s)\setminus U(0) }   }\\
=&\Expe{\para{U(t)-U(t)\cap U(0)} \cdot  \para{U(s)-U(s)\cap U(0)}}\\
=&\Expe{U(t)U(s)}-\Expe{\para{U(t)\cap U(0)}^{2}}\\
=&\ln(\frac{\maxp{t,s}}{\abs{t-s}}) -\frac{\maxp{t,s}-\abs{t-s}}{\delta}.
\end{eqalign}
We observe that the logarithmic factor is scale-invariant
\begin{eqalign}
 \ln(\frac{\lambda\maxp{t,s}}{\lambda\abs{t-s}})=\ln(\frac{\maxp{t,s}}{\abs{t-s}}).   
\end{eqalign}
We can dominate $U(\lambda s)\setminus U(0)$  by a Gaussian field $Y$ with covariance $\ln(\frac{\maxp{t,s}}{\abs{t-s}})1_{\abs{t-s}\leq \delta}$
\begin{eqalign}\label{eq:singularintegral1}
\Expe{\para{\int_{0}^{t}\frac{1}{x^{\alpha}}e^{ \overline{X}(x) }dx}^{p} }=&\frac{1}{t^{(\alpha-1)p}}\Expe{\para{\int_{0}^{1}\frac{1}{x^{\alpha}}e^{\overline{X}(tx)}dx}^{p} }\\
\leq &\frac{1}{t^{(\alpha-1)p}}\Expe{\para{\int_{0}^{1}\frac{1}{x^{\alpha}}e^{\overline{Y}(x)}dx}^{p} }\\
\end{eqalign}
It suffices to show finiteness of the above expectation for $Y$. We follow \cite[Lemma A.1]{david2016liouville}
\begin{eqalign}
&\para{\Expe{\para{\int_{0}^{1}\frac{1}{x^{\alpha}}e^{ \overline{Y}^{1}(x)}\dx}^{p}} }^{1/p}\\  \leq &\para{\Expe{\para{\int_{0}^{1/2}\frac{1}{x^{\alpha}}e^{ \overline{Y}^{1}(x)}\dx}^{p}} }^{1/p}+c\para{\Expe{\para{\int_{1/2}^{1}e^{ \overline{Y}^{1}(x)}\dx}^{p}} }^{1/p}\\   \leq &2^{\alpha-1}\para{\Expe{\para{\int_{0}^{1}\frac{1}{x^{\alpha}}e^{ \overline{Y}^{1}(x)}\dx}^{p}} }^{1/p}+C,
\end{eqalign}
So for finiteness we require $\alpha-1<0$.

\end{proof}

\section{Moments of the maximum and minimum of modulus of GMC }\label{maxminmodGMC}
In this section we study tail estimates and small ball estimates of the maximum/minimum of shifted GMC.
\subsection{Positive moments modulus}
First we nee to estimate the supremum of the field $X_{T}:=\gamma U(T+s)\cap U(T)$ using  \cref{them:BTIS} and \cref{thm:entropyintegralest}. 
\begin{lemma}\label{lem:supremumintersectGaussian}Fix $L,s$ and $p\geq 0$. Then we have
\begin{eqalign}
e^{-p\beta \ln\frac{\delta}{s}}\Expe{\sup_{T\in [0,L]}e^{p\gamma  U(T+s)\cap U(T)}}\leq& \branchmat{c\frac{1}{s^{12\sqrt{2}p\gamma+\beta p(p-1) }}\sqrt{\ln\frac{\delta}{s}}& \frac{L}{s}\geq e\\c\frac{1}{s^{\beta p(p-1) }}\sqrt{\ln\frac{\delta}{s}}& \frac{L}{s}\leq e},  
\end{eqalign}
for some constant $c$. 
\end{lemma}
\begin{proof}
We begin by studying the case $p=1$ and estimating the entropy number. The field has covariance and $L^2$-difference
\begin{eqalign}
&\Expe{X_{T}X_{S}}=\gamma^{2}R_{s}^{\delta}(s+\abs{T-S})\leq \gamma^{2}\para{\ln\frac{\delta}{s+\abs{T-S}}}1_{s+\abs{T-S}\leq \delta},\\
&\Expe{\para{X_{T}-X_{S}}^{2}}\leq \gamma^{2}\para{\ln\para{\frac{\abs{T-S}}{s}+1}}1_{s+\abs{T-S}\leq \delta},\\
&\sigma_{L}^{2}=\sup_{t\in [0,L]}\Expe{X_{t}^{2}}=\gamma^{2} \ln\frac{\delta}{s},
\end{eqalign}
where $R$ is defined in \cref{eq:Ucovariance} and so we let
\begin{eqalign}
d(T,S):=\sqrt{\gamma^{2} \para{\ln\para{\frac{\abs{T-S}}{s}+1}}1_{s+\abs{T-S}\leq \delta}}.    
\end{eqalign}
Therefore, 
\begin{eqalign}
d(0,T)\leq \epsilon \doncl T\leq s\para{e^{\frac{1}{\gamma^{2}}\epsilon^{2}}-1}=:r_{\epsilon,s} \end{eqalign}
and so for the space $[0,L]$, we have the entropy number $N(T,d,\epsilon)=\frac{L}{r_{\epsilon,s}}$. We also have the diameter
\begin{eqalign}
\diam([0,L])=d(0,L)=\sqrt{ \gamma^{2}\para{\ln\para{\frac{L}{s}+1}}1_{s+L\leq \delta}}.  
\end{eqalign}
Here we need to split cases over $\frac{L}{s}\geq e$ and $\frac{L}{s}\leq e$. In the first case we will get a singularity from the expected supremum, and in the second case it will be bounded. The integral estimate in \cref{thm:entropyintegralest} is
\begin{eqalign}
&\int_{0}^{\diam([0,L])}\sqrt{\ln N(T, d, \e)}\dint\epsilon\\
=&\int_{0}^{\diam([0,L])}\sqrt{\ln\para{\frac{L}{s\para{e^{\frac{1}{\gamma^{2}}\epsilon^{2}}-1}}}}\dint\epsilon\\
\leq &\gamma\sqrt{2}\sqrt{\frac{L}{s}} \int_{0}^{\sqrt{\frac{s}{L\gamma^{2}}}\diam([0,L])}\sqrt{\ln\frac{1}{\epsilon}}\dint\epsilon\\
\leq &\sqrt{2} \diam([0,L])\sqrt{\ln\frac{1}{\sqrt{\frac{s}{L\gamma^{2}}}\diam([0,L])}}\\
= &\gamma \sqrt{ \para{\ln\para{\frac{L}{s}+1}}\ln\frac{L}{s \ln\para{\frac{L}{s}+1}}}\\
\leq &\gamma\para{\ln\para{\frac{L}{s}+1}} \sqrt{ 1+\para{\ln\para{\frac{L}{s}+1}}^{-1}\ln\frac{1}{ \ln\para{\frac{L}{s}+1}}}.
\end{eqalign}
When we have $\frac{L}{s}\geq e$, then we can further upper bound 
\begin{eqalign}
\Expe{\norm{X}}\leq 12\int_{0}^{\diam([0,L])}\sqrt{\ln N(T, d, \e)}\dint\epsilon\leq     12\sqrt{2}\gamma \ln\para{\frac{L}{s}+1}.
\end{eqalign}
When we have $\frac{L}{s}\leq e$, then we simply bound $\diam([0,L])\leq \sqrt{ \gamma^{2}\para{\ln\para{e+1}}}=B_{e} $ and so
\begin{eqalign}
\int_{0}^{\diam([0,L])}\sqrt{\ln N(T, d, \e)}\dint\epsilon\leq &\int_{0}^{B_{e}}\sqrt{\ln\para{\frac{e}{\para{e^{\frac{1}{\gamma^{2}}\epsilon^{2}}-1}}}}\dint\epsilon,
\end{eqalign}
which is a finite integral. We return to the main supremum and use \cref{them:BTIS}
\begin{eqalign}
&\Expe{\sup_{T\in [0,L]}e^{\gamma U(T+s)\cap U(T)}}\\
=&\int_{0}^{\infty}\Proba{\sup_{T\in [0,L]}\gamma U(T+s)\cap U(T)\geq \ln t}\dt \\
\leq &1+\int_{0}^{\infty}\Proba{\sup_{T\in [0,L]}\gamma U(T+s)\cap U(T)\geq u}e^{u}\du \\
\leq&e^{\Expe{\norm{X}}}+ \int_{\Expe{\norm{X}}}^{\infty}\Proba{\sup_{T\in [0,L]}\gamma U(T+s)\cap U(T)\geq u}e^{u}\du \\
\stackrel{\cref{them:BTIS}}{\leq}&e^{\Expe{\norm{X}}}+ \int_{\Expe{\norm{X}}}^{\infty}2\expo{-\frac{(u-\Expe{\norm{X}})^{2}}{2\sigma_{L}^{2}}}e^{u}\du \\
=&e^{\Expe{\norm{X}}}+ e^{\Expe{\norm{X}}}\int_{0}^{\infty}2\expo{-\frac{u^{2}}{2\sigma_{L}^{2}}}e^{u}\du \\
\leq &e^{\Expe{\norm{X}}}+ e^{\Expe{\norm{X}}}\sqrt{2\pi\sigma_{L}^{2}}e^{\frac{1}{4}2\sigma_{L}^{2}}\\
\leq& e^{ 12\sqrt{2}\gamma \ln\para{\frac{L}{s}+1}}\para{1+\sqrt{2\pi}\sqrt{\gamma^{2} \ln\frac{\delta}{s}}e^{\frac{\gamma^{2}}{2} \ln\frac{\delta}{s}}}.
\end{eqalign}
So for the normalized supremum we have the bound
\begin{eqalign}
e^{-\beta \ln\frac{\delta}{s}}\Expe{\sup_{T\in [0,L]}e^{\gamma U(T+s)\cap U(T)}}\leq&  e^{ 12\sqrt{2}\gamma \ln\para{\frac{L}{s}+1}}\para{e^{-\beta \ln\frac{\delta}{s}}+\sqrt{2\pi}\sqrt{\gamma^{2} \ln\frac{\delta}{s}}} \\
\leq& c\frac{1}{s^{12\sqrt{2}\gamma}}\para{1+\sqrt{\ln\frac{\delta}{s}}},
\end{eqalign}
for some constant $c$. \\
In the case of general $p>0$, we let $\tilde{\gamma}=p\gamma$ to get
\begin{eqalign}
\Expe{\sup_{T\in [0,L]}e^{p\gamma U(T+s)\cap U(T)}}\leq& e^{ 12\sqrt{2}p\gamma \ln\para{\frac{L}{s}+1}}\para{1+\sqrt{2\pi}\sqrt{p^{2}\gamma^{2} \ln\frac{\delta}{s}}e^{\frac{p^{2}\gamma^{2}}{2} \ln\frac{\delta}{s}}}
\end{eqalign}
and so when normalized we have the bound
\begin{eqalign}
e^{-p\beta \ln\frac{\delta}{s}}\Expe{\sup_{T\in [0,L]}e^{p\gamma  U(T+s)\cap U(T)}}\leq& c\frac{1}{s^{12\sqrt{2}p\gamma+\beta p(p-1) }}\sqrt{\ln\frac{\delta}{s}}.
\end{eqalign}

\end{proof}

\begin{proposition}\label{prop:maxmoduluseta}
\pparagraph{Moments $p\in [1,\frac{2}{\gamma^{2}})$ using union bound }
For $L,\delta,x\geq 0$  and $\delta\leq 1$ we have
\begin{eqalign}\label{eq:maxmodulusetapone}
\Expe{\para{\supl{T\in[0,L] }\etamu{T,T+x}{\delta}}^{p}}\leq c\ceil{\frac{L}{x}}\Expe{\para{\etamu{0,x}{\delta}}^{p} }.    
\end{eqalign}
\pparagraph{Moments $p\in [1,\frac{2}{\gamma^{2}})$ using decoupling}
We have the bound
\begin{eqalign}
&\Expe{\para{\int_{0}^{x}\sup_{T\in [0,L]}e^{\overline{U(T+s)\cap U(T)}} \deta_{+}(s)}^{p}  }\leq \branchmat{c x^{p(1-\alpha)} & \frac{L}{x}\geq e\\c x^{p(1-\tilde{\alpha})} & \frac{L}{x}\leq e},    
\end{eqalign}
for $\alpha:=12\sqrt{2}\gamma+\beta (p-1)+\epsilon<1$ and $\tilde{\alpha}:=\beta (p-1)+\epsilon<1$.
\pparagraph{Moments $p\in (0,1)$}
We have the bound
\begin{eqalign}
\Expe{\para{\int_{0}^{x}\sup_{T\in [0,L]}e^{\overline{U(T+s)\cap U(T)}} \deta_{+}(s)}^{p}  }\leq & \branchmat{c x^{p(1-12\sqrt{2}\gamma)}\para{\ln\frac{1}{x}}^{p/2},     & \frac{L}{x}\geq e\\c \para{\ln\frac{1}{x}}^{p/2},     & \frac{L}{x}\leq e},
\end{eqalign}
where $c$ is a constant that diverges as $12\sqrt{2}\gamma\to 1$.
\end{proposition}
\begin{proof}
\pparagraph{Moments $p\in [1,\frac{2}{\gamma^{2}})$ using union bound }
We cover the interval $[0,L+x]$ by $I_{k}:=[kx,(k+1)x], k\in [0,n]$ for $n:=\ceil{\frac{L}{x}}$. Given any $T\in [0,L]$, we can find $I_{k},I_{k+1}$ so that they cover $[T,T+x]$ and so we bound by
\begin{eqalign}
\etamu{T,T+x}{\delta}\leq \etamu{I_{k}}{\delta}  +\etamu{I_{k+1}}{\delta}.   
\end{eqalign}
Therefore, we have the bound 
\begin{eqalign}\label{eqsupr}
&\Expe{\para{\supl{T\in[0,L] }\etamu{T,T+x}{\delta}}^{p}}\\
&\leq c_{p} \Expe{\para{\maxl{k=0,..,\frac{n}{2}}\etamu{I_{2k}}{\delta}}^{p}}+c_{p}\Expe{\para{\maxl{k=0,..,\frac{n}{2}}\etamu{I_{2k+1}}{\delta}}^{p}}.
\end{eqalign}
So from the union bound we have
\begin{eqalign}
\eqref{eqsupr}\leq c_{p} n \Expe{\para{\etamu{0,x}{\delta}}^{p} }  \end{eqalign}
\proofparagraph{Case $p\in (0,1)$}
Here, we start by applying Jensen's inequality
\begin{eqalign}\label{eq:sintegral}
\Expe{\para{\int_{0}^{x}\sup_{T\in [0,L]}e^{\overline{U(T+s)\cap U(T)}} \deta_{+}(s)}^{p}  }\leq \para{\int_{0}^{x}\Expe{\sup_{T\in [0,L]}e^{\overline{U(T+s)\cap U(T)}}} \ds}^{p}  ,    
\end{eqalign}
where $\deta_{+}(s):=e^{\overline{U(s)\setminus U(0)}}$ is independent. First, we study the case $\frac{L}{x}\geq e$. Using the \Cref{lem:supremumintersectGaussian} for $p=1$, we estimate
\begin{eqalign}\label{eq:sintegral2}
\eqref{eq:sintegral}\leq &c\para{\int_{0}^{x}\frac{1}{s^{12\sqrt{2}\gamma}}\para{1+\sqrt{\ln\frac{\delta}{s}}} \ds}^{p}\\
\leq &c x^{p(1-12\sqrt{2}\gamma)}\para{\ln\frac{1}{x}}^{p/2},    
\end{eqalign}
where $c=c(\gamma)$  diverges as $12\sqrt{2}\gamma\to 1$. For the case $\frac{L}{x}\leq e$, we simply bound by
\begin{eqalign}\label{eq:sintegral3}
\eqref{eq:sintegral}\leq &c\para{\int_{0}^{x}\para{1+\sqrt{\ln\frac{\delta}{s}}} \ds}^{p}\\
\leq &c \para{\ln\frac{1}{x}}^{p/2},    
\end{eqalign}
\proofparagraph{Case $p\in [1,\frac{1}{\beta})$ using decoupling}
Here we start with applying Minkowski's inequality
\begin{eqalign}\label{eq:sintegral10}
&\Expe{\para{\int_{0}^{x}\sup_{T\in [0,L]}e^{\overline{U(T+s)\cap U(T)}} \deta_{+}(s)}^{p}  }\\
\leq &\Expe{\para{\int_{0}^{x}\para{\Expe{\sup_{T\in [0,L]}e^{p\overline{U(T+s)\cap U(T)}}}}^{1/p} \deta_{+}(s)}^{p}}.    
\end{eqalign}
Using \Cref{lem:supremumintersectGaussian} for $\frac{L}{x}\geq e$, we estimate
\begin{eqalign}\label{eq:sintegral4}
\eqref{eq:sintegral10}\leq&c \Expe{\para{\int_{0}^{x}\frac{1}{s^{12\sqrt{2}\gamma+\beta (p-1) }}\para{\sqrt{\ln\frac{\delta}{s}}}^{1/p} \deta_{+}(s)}^{p}}.    
\end{eqalign}
To avoid the logarithmic factor we add and remove $s^{\epsilon}$ for $\epsilon>0$ and apply \cref{eq:singularalphaintegralestimate}
\begin{eqalign}\label{eq:sintegral5}
\eqref{eq:sintegral4}\leq&c x^{p(1-\alpha)}   
\end{eqalign}
for $\alpha:=12\sqrt{2}\gamma+\beta (p-1)+\epsilon<1$. For the case $\frac{L}{x}\leq e$, we instead get the exponent $\tilde{\alpha}:=\beta (p-1)+\epsilon$.\\
Of course, one can improve the bound by keeping the logarithmic factor and prove finiteness of the singular integral in \cref{eq:singularalphaintegralestimate}.

\end{proof}
\subsection{Negative moments modulus}
 Next we study the negative moments for the minimum of the modulus of GMC. 
\begin{proposition}\label{prop:minmodeta}
We have for $p>0$
\begin{eqalign}
\Expe{\para{\infl{T\in[0,L] }\etamu{T,T+x}{\delta}}^{-p}}\leq &c\ceil{\para{\frac{L}{x}}}\Expe{\para{\eta^{\delta}\spara{0,\frac{x}{2}}}^{-p} }\\
\leq& c\frac{L}{\delta} x^{\alpha_{2\delta}(p)},
\end{eqalign}
where $\alpha_{2\delta}(p)=\zeta(-p)-1$ if $\frac{x}{2}\leq \delta$ and $\alpha_{2\delta}(p)=-p$ if $\frac{x}{2}\geq \delta$.
\end{proposition}

\begin{proof}
We cover the interval $[0,L+x]$ by four sequences of intervals for $n:=\ceil{\frac{L}{x}}\geq 1$.\\
A generic interval $[T,T+x]$ always contains an interval of size $\frac{x}{2}$:
\begin{align*}\label{eq:minintervals}
&\tif T\in  [kx,kx+\frac{x}{2}], \tthen  \text{ the interval } J_{k}^{1}:=[kx+\frac{x}{2},x(k+1)]\subset [T,T+x],  \\
&\tif T\in  [xk+\frac{x}{2},x(k+1)], \tthen \text{ the interval }  J_{k}^{2}:=[x(k+1),x(k+1+\frac{1}{2})]\subset [T,T+x],  
\end{align*}
and we denote each of these collection as $C_{i}:=\set{J^{i}_{k}}_{k},i=1,2$. The GMC evaluated over them is \iid. By the union bound as we did \cref{eqsupr} we get:
\begin{eqalign}
  \Expe{\supl{T\in[0,\frac{L}{x}] }\para{\eta^{\delta}\spara{T,T+1}}^{-p}}\leq& \sum_{i=1,2}\Expe{\maxl{I\in C_{i} }\para{\eta_{\omega}^{1}\spara{I}}^{-p} }\\\leq &cn\Expe{\para{\eta^{\delta}\spara{0,\frac{x}{2}}}^{-p} }.
\end{eqalign}
\end{proof}
\begin{remark}
Here one can further study the small ball $\Proba{u\geq \min_{0\leq T\leq L}\etamu{T,T+x}{\delta} } $ and pull out the infimum of the field $U_{x}=U_{x}^{\delta}$ \begin{eqalign}
\Proba{u\geq \min_{0\leq T\leq L}G_{[T,T+x]}\etamu{T,T+x}{x} },   
\end{eqalign}
where $G_{[T,T+x]}:=\min_{0\leq s \leq x }e^{U_{x}(T+s)}$.
\end{remark}
\end{appendices}


\bibliography{sn-bibliography.bib}

\end{document}